\documentclass[11pt]{article}
\DeclareMathAlphabet{\mathpzc}{OT1}{pzc}{m}{it}
\usepackage{aliases}
\usepackage{subfiles}
\usepackage{fullpage}
\usepackage{hyperref}
\usepackage{subfigure}
\usepackage{authblk}

\numberwithin{equation}{section} \allowdisplaybreaks
\newcommand{\net}{\text{VarMiON }}

\begin{document}

\title{Variationally Mimetic Operator Networks}
\date{}


\author[1]{Dhruv Patel \thanks{dvpatel@stanford.edu}}
\author[2]{Deep Ray \thanks{deepray@umd.edu}}
\author[3]{Michael R A Abdelmalik \thanks{M.Abdel.Malik@tue.nl}}
\author[4]{Thomas J R Hughes \thanks{hughes@oden.utexas.edu}}
\author[5]{Assad A Oberai \thanks{aoberai@usc.edu}}
\affil[1]{Department of Mechanical Engineering, Stanford University}
\affil[2]{Department of Mathematics, University of Maryland at College Park}
\affil[3]{Department of Mechanical Engineering, Eindhoven University of Technology}
\affil[4]{Oden Institute for Computational Engineering and Sciences, University of Texas at Austin}
\affil[5]{Department of Aerospace and Mechanical Engineering, University of Southern California}

\maketitle

\begin{abstract}
    In recent years operator networks have emerged as promising deep learning tools for approximating the solution to partial differential equations (PDEs). These networks map input functions that describe material properties, forcing functions and boundary data to the solution of a PDE. This work describes a new architecture for operator networks that mimics the form of the numerical solution obtained from an approximate variational or weak formulation of the problem. The application of these ideas to a generic elliptic PDE leads to a variationally mimetic operator network (\net). Like the conventional Deep Operator Network (DeepONet) the \net is also composed of a sub-network that constructs the basis functions for the output and another that constructs the coefficients for these basis functions. However, in contrast to the DeepONet, the architecture of these sub-networks in the \net is precisely determined. An analysis of the error in the \net solution reveals that it contains contributions from the error in the training data, the training error, the quadrature error in sampling input and output functions, and a ``covering error'' that measures the distance between the test input functions and the nearest functions in the training dataset. It also depends on the stability constants for the exact solution operator and its \net approximation. The application of the \net to a canonical elliptic PDE and a nonlinear PDE reveals that for approximately the same number of network parameters, on average the \net incurs smaller errors than a standard DeepONet and a recently proposed multiple-input operator network (MIONet). Further, its performance is more robust to variations in input functions, the techniques used to sample the input and output functions, the techniques used to construct the basis functions, and the number of input functions. 
    Moreover, it consistently outperforms baseline methods at various dataset sizes.
\end{abstract}
\section{Introduction}

Over the last decade there has been significant interest in the use of Deep Learning (DL) based methods for solving problems that emerge from physical principles. This includes the so-called Physics-Informed Neural Networks (PINNs) that were introduced in the 1990s \cite{lagaris2000neural} (though not under that name) and then further developed more recently \cite{raissi2019physics}. There have been several interesting extensions of PINNs
and theoretical investigations into their approximations properties, a selective list of which includes \cite{zhang2019quantifying,krishnapriyan2021characterizing,mao2020,yang2021,jagtap2020,mishra2021,mishra2022,shukla2021,deryck2022,cuomo2022}. However, one shortcoming of these techniques is that the fully trained network represents the solution of a single instance of the problem. If a slightly different problem is to be solved, for example, by altering the forcing function for a given partial differential equation (PDE), the network has to be retrained, which can be an arduous task.


Operator Networks address this challenge directly. Rather than focusing on a single instance of the solution to a given system, they approximate the operator that maps the input functions to the solution. For example, for a physical problem governed by a system of PDEs (which is the focus of this work), the input would consist of functions that specify the initial conditions, the boundary conditions, the forcing terms and the material properties, while the output function would be the solution of the system of PDEs. The benefit of the operator point of view is obvious. Once the network is trained it can be used to solve multiple instances of a problem corresponding to different input functions. This makes it an ideal surrogate model for many-query tasks like uncertainty quantification and optimization that require solutions corresponding to multiple instances of input functions. 


Broadly speaking, two distinct types of Operator Networks have been proposed and implemented. 
One of the first operator network architecture was first introduced in  \cite{chen1995universal} 
along with universal approximation guarantees. This result was then adapted to deep neural networks in \cite{lu2021learning} and the resulting model was called a DeepONet. In a DeepONet, the overall network is split into a branch and a trunk sub-network, and the solution is represented as a dot product between their outputs. The branch network accepts as input, samples of the input functions at discrete points and maps these to a latent vector of a fixed dimension. On the other hand, the trunk network maps the spatial coordinates of the problem to a latent vector of the same fixed dimension. The solution is obtained by performing a dot-product between these vectors. Given this, the trunk network can be viewed as a network that constructs a basis for representing the solution, and the branch network can be viewed as a network that constructs the coefficients for the basis functions. 


There has been significant recent work on extending and improving DeepONets. This includes reducing the amount of labeled data (the number of spatial points where the solution is required) needed to train the network \cite{wang2021learning} by adding terms to the loss function that are driven by the residual of the original system of equations. For physical problems that are derived from variational principles, the same benefit can also be achieved by encoding the variational principle into the loss function \cite{goswami2022physics}. For example, if the governing equations are obtained from the minimization of the total potential energy, then the potential energy of the system can be added to the overall loss function for the network. Other interesting work on DeepONets includes their extension to approximating stochastic operators \cite{yang2022scalable}, to operators associated with multiphysics problems \cite{mao2021,cai2021}, to operators with multiple inputs and outputs \cite{Jin2022, tan2022enhanced}, and to instances where the sensor points, which are the points where the input functions are queried, are allowed to vary from one sample to another \cite{prasthofer2022}.

There has also been significant work on analyzing the approximation properties of DeepONets. This includes the original universal approximation result for a shallow network \cite{chen1995universal}, a variant using radial basis functions (RBFs) \cite{chen95rbf} and the recent extension to deep networks \cite{lanthaler2022error}. In the same work \cite{lanthaler2022error}, the authors have also established error bounds for DeepONets used to approximate a broad class of operators. These bounds quantify the rate of growth in the network parameters as a function of the desired error in its solution. For a broad class of operators it is shown that this rate is exponential, however for specific operators driven by solutions to PDEs it is much milder.

Distinct from DeepONets is the work on Neural Operators \cite{kovachki2021}, which includes Graph Neural Operators, Fourier Neural Operators and their variants \cite{li2020fourier}. Neural Operators can be thought of as  direct extensions of typical deep neural networks to infinite dimensional functions. Like a typical neural network, they are composed of multiple layers, where each layer performs a linear and a nonlinear operation. However, in contrast to a typical deep neural network, the operations are performed on functions rather than vectors. The non-linear operation is performed using standard activation functions that are applied point-wise, while the linear operation is performed by local (affine) and non-local operators. Different choices of implementing the non-local operators lead to different types of Neural Operators. These include Fourier Neural Operators \cite{li2020fourier}, Multipole Neural Operators \cite{li2020multipole} and Graph Neural Operators \cite{li2020neural}. A physics-informed variant of Neural Operators which combined both data and PDE constraints was proposed in \cite{pino}. Recently, analytical estimates that quantify the performance of Fourier Neural Operators \cite{kovachki2021universal} have been developed. These include a universal approximation theorem and growth rates in complexity as a function of desired error when these networks are used to solve canonical problems like the steady-state heat conduction equation and the incompressible Navier Stokes equations.

In this manuscript we ask the question whether operator networks can be informed by the weak, or variational, formulations that are used to approximate solutions to PDEs. These formulations have traditionally been used to develop finite element methods, spectral methods, and projection-based reduced order models. Specifically, we consider the solution of an elliptic PDE and a nonlinear advection-diffusion-reaction PDE with spatially varying coefficients, forcing functions, and Neumann boundary data. The application of a standard method, like the Galerkin method, to this problem leads to a numerical solution with a very specific structure. Motivated by this structure, we propose a Variationally Mimetic Operator Network (\net). Thereafter, we demonstrate that this network has several special features. These include:
\begin{enumerate}
    \item A computationally efficient network structure. For canonical linear and nonlinear problems we demonstrate that for roughly the same number of network parameters, the variationally mimetic network produces solutions that are significantly more accurate than the solutions produced by a generic DeepONet or MIONet \cite{Jin2022}. This improved performance holds for a network (a) with multiple input functions, (b) input sensors that are distributed randomly, or distributed in a spatially uniform grid, and (c) for a trunk network comprised of basis functions that employ ReLU activations, or those constructed from a span of radial basis functions. We also demonstrate that the accuracy of the \net is more robust. That is, when considering solutions corresponding to a distribution of input functions, the distribution of error has lower variance than a conventional DeepONet. Moreover, the VarMiON consistently outperforms both the DeepONet and the MIONet at varying cardinality of training set demonstrating its usefulness in low-data regimes.
    \item Motivated by the error analysis for methods like the finite element method, we derive an \textit{a-priori} error estimate for the solutions of the \net.  This analysis  reveals that the overall network error can be reduced by (a) training it with more accurate solutions, (b) training it with a large data set, (c) by designing branch networks that are stable with respect to perturbations in their input, and (b) by using larger number points for sampling the input and output functions. Further, the variationally-mimetic structure of the network enables a precise definition of the constants that appear in the error estimate. These constants are obtained from the discrete operators (matrices) that appear in the network, which can be easily evaluated for a trained network. 
    \item The development of a workflow that couples an efficient numerical solver for a given class of problems and a \net that is trained using solutions produced by the solver. The solver is used to train the \net, and to guide the design of the higher level structure of the network, such as selecting the nature of branches (linear/non-linear), determining the format of the branch output (matrix or vector), specifying how branches interact with each other (via matrix-vector product, dot product, or Hadamard product), and interpreting  each constituent of the network (such as the trunk network serving as a surrogate for the basis functions of the numerical solver). Further, the error analysis for the \net includes the error in the solver. 

\end{enumerate}




The format of the remainder of this paper is as follows. In Section \ref{sec:pde}, we introduce the model linear PDE and present its variational discretization. This motivates the architecture of the variationally mimetic operator network for canonical linear problems which is introduced in Section \ref{sec:net}. In the following section (Section \ref{sec:err_anal}) we present an analysis of the generalization error of the network.  Thereafter, in Section \ref{sec:nonlinear}, we present the extension of the \net to nonlinear PDEs. In Section \ref{sec:num} we present numerical results that demonstrate the benefit of the variationally mimetic architecture for both linear and non-linear problems. In particular, we demonstrate that for approximately the same number of network parameters, the \net produces more accurate solutions as compared to a conventional DeepONet and MIONet. It can also be used to introduce ``optimal'' basis functions, which further improve its performance. We end with conclusions in Section \ref{sec:con}.


\section{PDE model and discretization} 
\label{sec:pde}
Let $\Omega \in \Ro^d$ be an open, bounded domain with piecewise smooth boundary $\Gamma$. The boundary is further split into the Dirichlet boundary $\Gamma_g$ and natural boundary $\Gamma_\eta$, with $\Gamma = \Gamma_g \cup \Gamma_\eta$. Define the space $H^1_g = \{u \in H^1(\Omega) \ : \ u\big|_{\Gamma_g} = 0\}$. We consider the following scalar elliptic boundary value problem
\begin{equation}\label{eqn:pde}
\begin{aligned}
\mathcal{L}(u(\x);\theta(\x)) &=  f(\x), \quad &&\forall \ \x \in  \Omega, \\
\mathcal{B}(u(\x);\theta(\x)) &= \eta (\x), \quad &&\forall \ \x \in  \Gamma_\eta,\\
u(\x) &= 0, \quad &&\forall \ \x  \in  \Gamma_g,
\end{aligned}
\end{equation}
where $\mathcal{L}$ is assumed to be a second-order elliptic operator, $\mathcal{B}$ is the natural boundary operator, $f \in \dbF \subset L^2(\Omega)$ is the source term, $\eta \in \dbN \subset L^2(\Gamma_\eta)$ is the flux data, and $\theta \in \mathcal{T} \subset L^\infty(\Omega)$ is some spatially varying material parameter, such as thermal conductivity or permeability. 

The variational formulation of \eqref{eqn:pde} is given by: find $u \in H^1_g$ such that $\forall \ w \in H^1_g$
\begin{equation}\label{eqn:weak}
a(w,u;\theta) = (w,f) + (w,\eta)_{\Gamma_\eta},
\end{equation}
where $(.,.)$ and $(.,.)_{\Gamma_\eta}$ are respectively the $L^2(\Omega)$ and $L^2(\Gamma_\eta)$ inner-products, while $a(.,.;\theta)$ is the associated bilinear form. We assume that we are working with operators $\mathcal{L}$ possessing requisite properties, such as uniform ellipticity and continuity, to ensure that \eqref{eqn:weak} is well-posed.
We refer interested users to \cite{gilbarg2001elliptic,brenner2002mathematical} for additional details.

Consider the solution operator 
\begin{equation}\label{eqn:soln_op}
    \Sop:\dbX = \dbF \times \mathcal{T} \times \dbN  \longrightarrow \dbV \subset H^1_g, \quad \Sop(f,\theta,\eta) = u(.;f,\theta,\eta)
\end{equation}
which maps the data $(f,\theta,\eta)$ to the unique solution $u(,;f,\theta,\eta)$ of \eqref{eqn:weak}. Our goal is to approximate the operator $\Sop$ using a \net. 

For everything that follows, we assume that $\Omega$, $\Gamma_\eta$ and $\Gamma_g$ are fixed for a given PDE model, while the various quantities appearing in \eqref{eqn:pde} are dimensionless.

\subsection{Numerical solution}\label{sec:fem}
Training the \net requires solutions of \eqref{eqn:weak} for varying $(f,\theta,\eta) \in \dbX$. In the absence of analytical expressions, \eqref{eqn:weak} can be approximately solved using a suitable numerical solver. We consider the class of solvers that approximate $\dbV$ by the space $\dbV^h$ spanned by a set of continuous basis functions $\{\phi_i(\x)\}_{i=1}^q$. For instance, we could use a FEM or POD basis approximating the solutions of \eqref{eqn:weak} \cite{hughes2012finite,berkooz1993proper}. Then any function $v^h \in \dbV^{h}$ can be be expressed as a linear combination of the finite basis
\[
v^{h}(\x) = v_i \phi_i(\x) = \V^\top \bPhi(\x), \quad \V=(v_1,\cdots,v_q)^\top, \quad \bPhi(\x) = (\phi_1(\x),\cdots,\phi_q(\x))^\top.
\]
We also define the restricted space $\dbV^h\big|_{\Gamma_\eta} = \{v\big|_{\Gamma_\eta} : v \in \dbV^h\}$ for the boundary data.
Consider the projector
\begin{equation}\label{eqn:data_projector}
\mathcal{P} : \dbX \rightarrow \dbX^h :=\dbF^h \times \dbT^h \times \dbN^h \subset \dbV^h \times \dbV^h \times \dbV^h\big|_{\Gamma_\eta}, \qquad \mathcal{P}(f,\theta,\eta) = (f^h,\theta^h,\eta^h) 
\end{equation}
where the given PDE data is approximated (projected) as 
\begin{equation}\label{eqn:data_disc_app}
f^h(\x) =  \F^\top \bPhi(\x), \quad \theta^h(\x) = \Tb^\top \bPhi(\x), \quad 
\eta^h(\x) = \N^\top \bPhi(\x)\big|_{\Gamma_\eta}.
\end{equation}
The coefficients $\F,\Tb,\N$ will depend on the choice of the basis functions (also see Remark \ref{rem:coef}). If the approximate solution in $\dbV^h$ is represented as $u^h(\x) = \U^\top \bPhi(\x)$, then using \eqref{eqn:data_disc_app} in \eqref{eqn:weak} gives us the discrete weak formulation (see \cite{hughes2012finite}, for example)
\begin{equation}\label{eqn:d_weak}	
\K(\theta^h) \U = \M \F + \Mt \N
\end{equation}
where the matrices are given by
\begin{equation}\label{eqn:mat}	
K_{ij}(\theta^h) =a(\phi_i,\phi_j;\theta^h), \quad M_{ij} = (\phi_i,\phi_j), \quad \widetilde{M}_{ij} = (\phi_i,\phi_j)_{\Gamma_\eta} \quad 1 \leq i,j \leq q.
\end{equation}
Note that the matrix $\M$ will be invertible by virtue of its positive-definiteness. In order to recover the solution coefficients $\U$ from \eqref{eqn:d_weak}, we make the following assumption.
\begin{assumption}\label{ass:invert_K}
The basis $\{\phi_i(\x)\}_{i=1}^q$ is chosen such that the matrix $\K(\theta^h)$ is invertible for all $\theta^h \in \dbT^h$. Note that this is equivalent to requiring the spectrum of $\K(\theta^h)$ to be bounded away from 0.
\end{assumption}
We can now define the discrete solution operator 
\begin{equation}\label{eqn:d_soln_op}
\Sop^h: \dbV^h \times \dbV^h \times \dbV^h\big|_{\Gamma_\eta} \rightarrow \dbV^h, \quad \Sop^h(f^h,\theta^h,\eta^h) := u^h(.;f^h,\theta^h,\eta^h) = (\B(f^h,\theta^h) + \Bt(\eta^h,\theta^h))^\top \bPhi
\end{equation}
where
\begin{equation}\label{eqn:Bs}
\B(f^h,\theta^h) = \K^{-1}(\theta^h) \M \F, \quad \Bt(\eta^h,\theta^h) = \K^{-1}(\theta^h) \Mt \N.    
\end{equation}

\begin{remark}
\label{rem:coef}
Given $(f,\theta,\eta) \in \dbX$, the coefficients in \eqref{eqn:data_disc_app} can be evaluated as $\F = \M^{-1} \ol{\F}$ and $\Tb = \M^{-1} \ol{\Tb}$, where $\ol{F}_i = (f,\phi_i)$ and $\ol{\Theta}_i = (\theta,\phi_i)$ for $1\leq i \leq q$. Let us assume that $q^\prime \leq q$ basis functions do not vanish (almost everywhere) on $\Gamma_\eta$, which we enumerate as $\{\phi_{R(j)} \}_{j=1}^{q^\prime}$ with $R(j)$ being the global index. Then the $q^\prime \times q^\prime$ matrix $\mathring{\M}$ given by $\mathring{M}_{ij} = (\phi_{R(i)},\phi_{R(j)})_{\Gamma_\eta}$ will be a maximal invertible sub-matrix of $\Mt$. This is used to uniquely define the coefficients of $\eta^h(\x)$ as 
\begin{equation*}
\mathring{\N} = \mathring{\M}^{-1}\ol{\N}, \quad \ol{H}_j = (\eta,\phi_{R(j)})_{\Gamma_\eta} \ \forall \ 1 \leq j \leq q^\prime, \quad H_i = \begin{cases} \mathring{H}_j & \quad \text{if } i = R(j) \text{ for some } 1 \leq j \leq q^\prime\\ 0& \quad \text{otherwise}\end{cases}.
\end{equation*}
\end{remark}

Note that Remark \ref{rem:coef} implies that the projector $\mathcal{P}:\dbX \rightarrow \dbX^h$ is a linear operator.

\section{Variationally mimetic operator network}\label{sec:net}
The \net architecture is motivated by the discrete weak variational form \eqref{eqn:d_weak}. The PDE data $(f,\theta,\eta) \in \dbX$ is fed into the \net, with $f$ and $\theta$ sampled at the sensor nodes $\{\xh_i\}_{i=1}^k$ while $\eta$ is sampled at the boundary sensor nodes $\{\xh_i^b\}_{i=1}^{k^\prime}$ on $\Gamma_h$. We define the input vectors
\begin{equation}\label{eqn:input}
\Fh = (f(\xh_1),\cdots,f(\xh_k))^\top, \quad \Th = (\theta(\xh_1),\cdots,\theta(\xh_k))^\top, \quad \Nh = (\eta(\xh^b_1),\cdots,\eta(\xh^b_{k^\prime}))^\top,
\end{equation}
given by the data sensing operator $\Ph$ as
\begin{equation}\label{eqn:sensing_op}
    \Ph: \dbX \rightarrow \Ro^k \times \Ro^k \times \Ro^{k^\prime}, \qquad \Ph(f,\theta,\eta) = (\Fh,\Th,\Nh).
\end{equation}

\begin{figure}[htbp]
\begin{center}
\includegraphics[width=\textwidth]{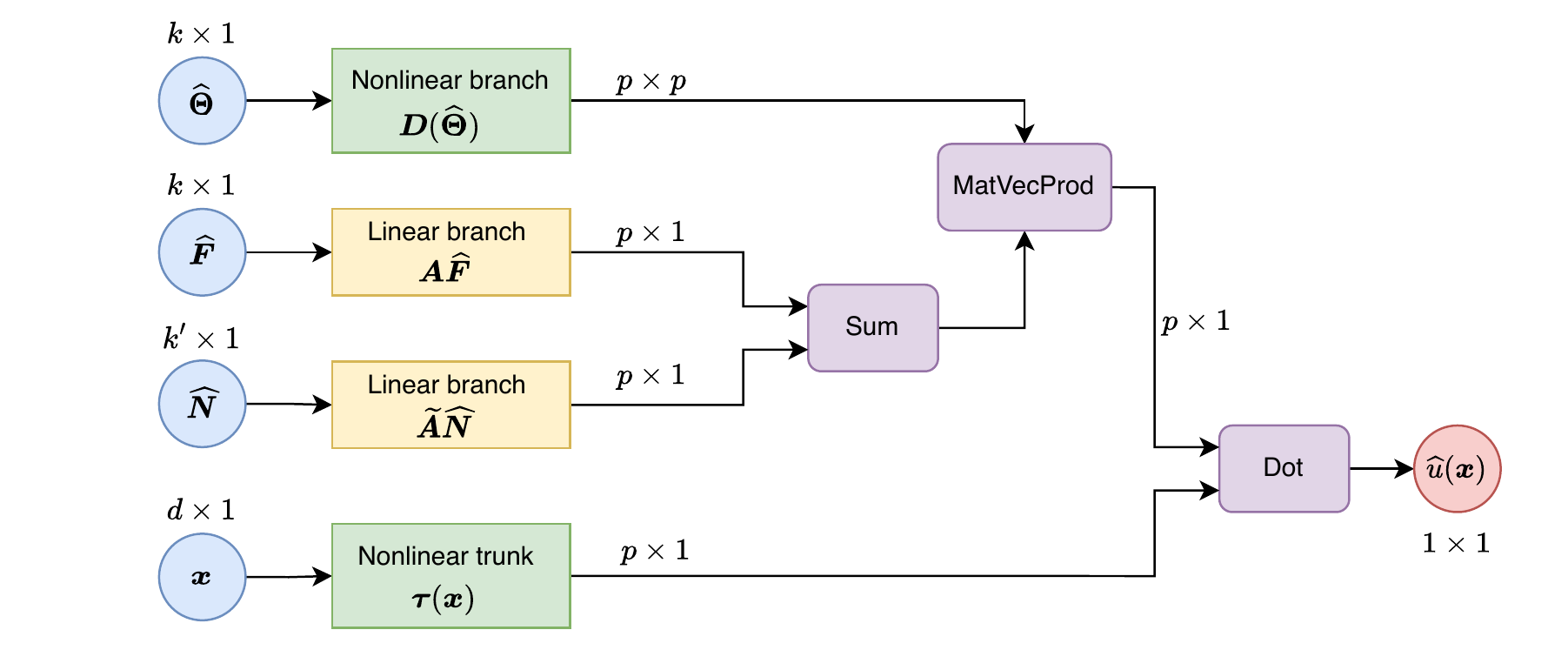}
\caption{\net architecture.}
\label{fig:varmion}
\end{center}
\end{figure}

We consider the \net shown in Figure \ref{fig:varmion}, which comprises:
\begin{itemize}
\item a \textbf{non-linear branch} taking the input $\Th \in \Ro^k$, which is transformed into a matrix output $\D(\Th) \in \Ro^{p \times p}$. Here $p$ is the latent dimension of the \net.

\item a \textbf{linear branch} taking the input $\Fh \in \Ro^k$ and transforming it as $\A \Fh$, where $\A \in \Ro^{p \times k}$ is a learnable matrix. The output of this branch is acted upon by the matrix $\D$ to give $\bbeta(\Fh,\Th) = \D(\Th) \A \Fh \in \Ro^p$.

\item a \textbf{linear branch} taking the input $\Nh \in \Ro^{k^\prime}$ and transforming it as $\At \Nh$, where $\At \in \Ro^{p \times k^\prime}$ is a learnable matrix. The output of this branch is acted upon by the matrix $\D$ to give $\bbetat(\Nh,\Th) = \D(\Th) \At \Nh \in \Ro^p$.

\item a \textbf{non-linear trunk} taking the input $\x \in \Ro^d$, which gives the output $\btau(\x) = (\tau_1(\x), \cdots, \tau_p(\x))^\top$, where each $\tau_i : \Ro^d \rightarrow \Ro$ is a trainable network. 
\end{itemize}

Let $\dbV^\tau$ be the space spanned by the trained trunk functions $\tau$. Then the final \net operator is given by
\begin{equation}\label{eqn:net}
    \widehat{\Sop} : \Ro^k \times \Ro^k \times   \Ro^{k^\prime} \rightarrow \dbV^\tau, \qquad \widehat{\Sop}(\Fh,\Th,\Nh) = \widehat{u}(.;\Fh,\Th,\Nh) = \left(\bbeta(\Fh,\Th) + \bbetat(\Nh,\Th) \right)^\top \btau.
\end{equation}
Notice the similarity between \eqref{eqn:d_soln_op} and \eqref{eqn:net}. In particular, $\B$ (respectively $\Bt$) have the same structure as $\bbeta$ (respectively $\bbetat$). We further demonstrate in Section \ref{sec:Kapprox} that $\D$ can be seen as a coarse scale approximation of $\K^{-1}$.

The discrete operator $\Sop^h$ is used to generate samples to train the \net. This is done in the following manner:
\begin{enumerate}
    \item For $1\leq j \leq J$, consider distinct samples $(f_j,\theta_j,\eta_j) \in \dbX$. 
    \item Use the projector \eqref{eqn:data_projector} to obtain the discrete approximations $(f^h_j,\theta^h_j,\eta^h_j) \in \dbX^h$. 
    \item Find the discrete numerical solution $u^h_j = \Sop^h(f^h_j,\theta^h_j,\eta^h_j)$.
    \item Use \eqref{eqn:sensing_op} to generate the \net input vectors $(\Fh_j,\Th_j,\Nh_j) = \Ph(f_j,\theta_j,\eta_j)$.
    \item Select a set of output nodes $\{\x_l\}_{l=1}^{L}$. For each $1 \leq j\leq J$, sample the numerical solution at these nodes as $u^h_{jl} = u^h_j(\x_l)$. 
    \item Finally, collect all the inputs and output labels to form the training set containing $JL$ samples
    \begin{equation}\label{eqn:tset}
        \mathbb{S} = \{(\Fh_j,\Th_j,\Nh_j,\x_l,u_{jl}^h) : 1 \leq j \leq J, \ 1 \leq l \leq L\},
    \end{equation}
    where $(\Fh_j,\Th_j,\Nh_j)$ is the input to the branch subnets, $\x_l$ is the trunk input, and $u_{jl}^h$ is the target output. 
\end{enumerate}

\begin{remark}\label{rem:pointwise}
To allow a point-wise evaluation of the input and output to the \net, we need to assume $\dbF \subset C(\Omega)$, $\dbN \subset C(\Gamma_\eta)$, $\mathcal{T} \subset C(\ol{\Omega})$ and $\mathcal{V} \subset C(\ol{\Omega})$.
\end{remark}

We also need to define a suitable loss/objective function which needs to be optimized to train the \net. First, let $\{w_l\}_{l=1}^L$ be the weights corresponding to the output nodes $\{\x_l\}_{l=1}^L \in \Omega$ used to describe the following quadrature for the square of function $g : \Omega \rightarrow \Ro$
\begin{equation}\label{eqn:quad_out}
\left | \sum_{l=1}^L w_l g^2(\x_l) -  \int_\Omega g^2(\x) \ud \x \right| \leq  \frac{\cons_Q(g)}{L^{\gamma}},
\end{equation}
where the $\cons_Q$ may depend on the derivatives of $g$. The rate of convergence $\gamma$ will depend on the quadrature rule used. Let $\bpsi$ be the (vector of) trainable parameters of the \net, which we make explicit by representing the \net by $\Soph_\bpsi$. We define the loss function as
\begin{equation}
    \Pi(\bpsi) = \frac{1}{J} \sum_{j=1}^{J} \Pi_j(\bpsi) , \qquad \Pi_j(\bpsi) = \sum_{l=1}^{L} w_{l} \left(u^h_{jl} - \Soph_\bpsi(\Fh_j,\Th_j,\Nh_j)[\x_l] \right)^2. \label{eqn:loss}
\end{equation}
Note that $\Pi_j(\bpsi)  \approx  \|\Sop^h\circ \mathcal{P} (f_j,\theta_j,\eta_j) - \Soph_\bpsi \circ \Ph (f_j,\theta_j,\eta_j) \|^2_{L^2(\Omega)}$ by virtue of \eqref{eqn:quad_out}.
Then training the \net corresponds to solving the following optimization problem
\begin{equation}\label{eqn:train}
    \bpsi^* = \argmin{\bpsi}\Pi(\bpsi).
\end{equation}
In practice, \eqref{eqn:train} is typically solved using iterative algorithms such as stochastic gradient descent or Adam \cite{kingma2017adam}. Further, in the numerical results presented in Section \ref{sec:num}, we choose the output sensor nodes as random Monte Carlo nodes in the domain, with the corresponding quadrature weights set as $w_l = 1/L$.

\section{Error analysis}\label{sec:err_anal}
Our objective is to train a \net such that the true solution operator $\Sop$ is well-approximated by $\Soph \circ \Ph:\dbX \rightarrow \dbV^\tau$. We denote the corresponding error for any $(f, \theta,\eta)\in \dbX$ as
\begin{eqnarray}\label{eqn:err_true_net}
 \Er(f,\theta,\eta) &:=& \|\Sop(f,\theta,\eta) - \Soph\circ\Ph(f,\theta,\eta)\|_{L^2(\Omega)}.
\end{eqnarray}
which we also call the \textit{generalization error}. Additionally, we define the error between $\Sop$ and $\Sop^h$, i.e., the numerical approximation error, as
\begin{eqnarray}\label{eqn:err_true_num}
 \Er_h(f,\theta,\eta) &:=& \|\Sop(f,\theta,\eta) - \Sop^h\circ\mathcal{P}(f,\theta,\eta)\|_{L^2(\Omega)},
\end{eqnarray}
and that between $\Sop^h$ and the $\Soph$ as
\begin{eqnarray}\label{eqn:err_num_net}
 \Erh(f,\theta,\eta) &:=& \|\Sop^h\circ\mathcal{P}(f,\theta,\eta) - \Soph\circ\Ph(f,\eta,\theta )\|_{L^2(\Omega)} \notag \\
 &=& \|\Sop^h(f^h,\theta^h,\eta^h) - \Soph(\Fh,\Th,\Nh )\|_{L^2(\Omega)}.
\end{eqnarray}
Further, for any $(f,\theta,\eta),(f^\prime,\theta^\prime,\eta^\prime) \in \dbX$, we define the corresponding perturbation in the solution given by $\Sop$ as
\begin{eqnarray}\label{eqn:sop_pert}
 \Er_\text{stab}[(f,\theta,\eta),(f^\prime,\theta^\prime,\eta^\prime)] &:=&  \|\Sop(f,\theta,\eta) - \Sop(f^\prime,\theta^\prime,\eta^\prime)\|_{L^2(\Omega)},
\end{eqnarray}
and the \net operator $\Soph$ as
\begin{eqnarray}\label{eqn:soph_pert}
 \Erh_\text{stab}[(f,\theta,\eta),(f^\prime,\theta^\prime,\eta^\prime)] &:=& \|\Soph \circ \mathcal{P}(f,\theta,\eta) - \Soph\circ \mathcal{P}(f^\prime,\theta^\prime,\eta^\prime)\|_{L^2(\Omega)},
\end{eqnarray}
Then, for any training sample index $1 \leq j \leq J$, we obtain the estimate
\begin{eqnarray}\label{eqn:gen_err_breakup}
 \Er(f,\theta,\eta) &=& \|\Sop(f,\theta,\eta) - \Sop(f_j,\theta_j,\eta_j) + \Sop(f_j,\theta_j,\eta_j) - \Sop^h\circ \mathcal{P}(f_j,\theta_j,\eta_j) \notag \\
 && \ \ + \Sop^h\circ\mathcal{P}(f_j,\theta_j,\eta_j)- \Soph\circ\Ph(f_j,\theta_j,\eta_j) + \Soph\circ\Ph(f_j,\theta_j,\eta_j) - \Soph\circ\Ph(f,\theta,\eta)\|_{L^2(\Omega)} \notag\\
 &\leq& \|\Sop(f,\theta,\eta) - \Sop(f_j,\theta_j,\eta_j)\|_{L^2(\Omega)} + \|\Sop(f_j,\theta_j,\eta_j) - \Sop^h\circ \mathcal{P}(f_j,\theta_j,\eta_j)\|_{L^2(\Omega)} \notag \\
 && + \|\Sop^h\circ\mathcal{P}(f_j,\theta_j,\eta_j)- \Soph\circ\Ph(f_j,\theta_j,\eta_j)\|_{L^2(\Omega)}  +  \|\Soph\circ\Ph(f_j,\theta_j,\eta_j) - \Soph\circ\Ph(f,\theta,\eta)\|_{L^2(\Omega)} \notag \\
 &=& \Er_\text{stab}[(f,\theta,\eta),(f_j,\theta_j,\eta_j)] + \Er_h(f_j,\theta_j,\eta_j) + \Erh(f_j,\theta_j,\eta_j) \notag \\
 && + \Erh_\text{stab}[(f_j,\theta_j,\eta_j),(f,\theta,\eta)].
\end{eqnarray}
Thus, we need to find suitable bounds to each of the four terms on the right of the above expression, to bound the generalization error. These estimates are investigated in the following sections.

\subsection{Stability of $\Sop$}\label{sec:stable_sop}
The term \eqref{eqn:sop_pert} captures how much the output of $\Sop$ varies as the PDE-data is perturbed. Thus, bounding this term requires the solution operator of the PDE to be stable. This leads us to the following assumption.

\begin{assumption}\label{ass:stable_sop}
The solution operator $\Sop$ is stable with respect to the PDE data. In other words, for any  $(f,\theta,\eta),(f^\prime,\theta^\prime,\eta^\prime) \in \dbX$
\begin{align}\label{eqn:stable_sop}
\Er_\text{stab}[(f,\theta,\eta),(f^\prime,\theta^\prime,\eta^\prime)]\leq& \cons_\text{stab} \left( \|f - f^\prime \|_{L^2(\Omega)} + \|\theta - \theta^\prime \|_{L^2(\Omega)} + \|\eta - \eta^\prime \|_{L^2(\Gamma_\eta)}\right),
\end{align}
where $\cons_\text{stab}$ depends may depend on $\Omega$, $\Gamma_\eta$ and $\Gamma_g$.
\end{assumption}
A stability estimate of the type \eqref{eqn:stable_sop} is guaranteed if the underlying PDE model is well-defined, and have been explored for elliptic PDEs in \cite{bonito2013,iglesias2016}.

\subsection{Numerical approximation error}\label{sec:num_approx}
The term \eqref{eqn:err_true_num} characterizes the error introduced while generating the training set using the numerical solver to the PDE. This motivates the use of high-order solvers in the data generation process. For the purposes of our discussions, we make the following assumption about the numerical error.

\begin{assumption}\label{ass:num_approx}
Given the projection function \eqref{eqn:data_projector}, there exists an $\epsilon_h > 0$ such that the following error bound holds
\begin{equation}\label{eqn:num_approx}
    \Er_h(f,\theta,\eta) < \cons_h(f,\theta,\eta)\epsilon_h \quad \forall \ (f,\theta,\eta) \in \dbX
\end{equation}
where $\cons_h$ may depend on the given PDE data.
\end{assumption}

Moving forward, we assume that the PDE data is bounded which is also typical for the data used to train neural networks.

\begin{assumption}\label{ass:data_bnd}
Assume that the PDE data space $\dbX$ is compact. We also assume that the data is point-wise bounded. In other words, there exist constants $C_\dbF, C_\dbT, C_\dbN < \infty$ such that
\begin{equation}\label{eqn:data_bnd}
\begin{aligned}
    \|f\|_{L^\infty(\Omega)} \leq C_\dbF \ \forall \ f \in \dbF,\quad \|\theta\|_{L^\infty(\Omega)} \leq  C_\mathcal{T} \ \forall \ \theta \in \mathcal{T},\quad  \|\eta\|_{L^\infty(\Gamma_\eta)} \leq C_\dbN \ \forall \ \eta \in \dbN.
\end{aligned}
\end{equation}
\end{assumption}

We make a few remarks here:
\begin{enumerate}
    \item Estimates of the form \eqref{eqn:num_approx} are typically available depending on the underlying numerical solver used to obtain $u^h$. For instance, if a Galerkin finite element method with degree $r$ basis functions is used to solve \eqref{eqn:pde}, then under sufficient regularity of the solution we get $\epsilon_h \sim h^{r+1}$ where $h$ represents the size of the elements.
    \item We have assumed $\cons_h$ to be as general as possible by allowing it to depend on the triplet $(f,\theta,\eta)$. However, if the underlying method is of Galerkin-type, then by Galerkin orthogonality, $\cons_h$ will not depend on $f$ and $\eta$, i.e., $\cons_h = \cons_h(\theta)$. 
    \item In addition to the bounds assumed in Assumption \ref{ass:data_bnd}, we may require $\theta > 0$ to ensure the well-definedness of the underlying PDE, for instance the steady-state heat conduction problem considered in Section \ref{sec:num}. 
    \item With the assumption of compactness of $\dbX$, we can define the following finite-valued constants
    \begin{equation}\label{eqn:num_err_constants}
    \cons_{h,\mathbb{S}} = \max_{1 \leq j \leq J} \cons_h(f_j,\theta_j,\eta_j), \quad \cons_{h,\dbX} = \sup_{\dbX} \cons_h(f,\theta,\eta),
\end{equation}
with $\cons_{h,\mathbb{S}} \leq \cons_{h,\dbX}$.
\end{enumerate}

\subsection{Training error}
We now wish to find a bound for \eqref{eqn:err_num_net} when evaluated on the training set. We demonstrate that this is closely associated with the final training error (which need not be zero) after solving the optimization problem \eqref{eqn:train}.

We begin by introducing the notion of when the \net is said to be well-trained.
\begin{defn}[$\epsilon_t$-trained] \label{def:well_trained}
Let a \net with operator $\Soph$ be trained on a dataset $\mathbb{S}$ generated by the set $\{f_j,\theta_j,\eta_j\}_{j=1}^J$ and defined by \eqref{eqn:tset}. Let $\{\x_l\}_{l=1}^L$ and $\{w_l\}_{l=1}^L$ be the quadrature nodes and weights asociated with the \eqref{eqn:quad_out}. Then the \net is said to be $\epsilon_t$-trained for some $\epsilon_t > 0$ if the following estimate holds for the trained network
\begin{equation}\label{eqn:err_training_samples}
\Pi_j(\bpsi^*) < \epsilon_t \quad \forall \ 1 \leq j \leq J
\end{equation}
where $\Pi(\phi^*)_j$ is defined according to \eqref{eqn:loss}.
\end{defn}

Note that the satisfaction of the estimate \eqref{eqn:err_training_samples} can be monitored in practice while training the \net. The following result is a consequence of Definition \ref{def:well_trained}, the proof of which follows trivially from \eqref{eqn:loss}, \eqref{eqn:quad_out} and the relation $\sqrt{(a + b)} \leq (\sqrt{a} + \sqrt{b})$ for $a,b\geq 0$.

\begin{lemma}\label{lem:welltrained}
Let us define the (point-wise) error in approximating the numerical solution using the \net $\Soph$ for the $j-th$ training sample as 
\[
e_j = \Sop^h\circ \mathcal{P} (f_j,\theta_j,\eta_j) - \Soph_\bpsi \circ \Ph (f_j,\theta_j,\eta_j). 
\]
If $\Soph$ is $\epsilon_t$-trained, then the training error $\Pi(\bpsi^*) < \epsilon_t$ and 
\begin{equation}\label{eqn:lossapprox}
  \Erh(f_j,\theta_j,\eta_j) \leq \sqrt{\Pi_j(\bpsi^*)} + \frac{\sqrt{\cons_{Qj}}}{L^{\gamma/2}} \leq \sqrt{\epsilon_t} + \frac{\sqrt{\cons_{Q,\mathbb{S}}}}{L^{\gamma/2}}
\end{equation}
where $\cons_{Qj} = \cons_Q(e_j)$ is defined as in \eqref{eqn:quad_out} and $\cons_{Q,\mathbb{S}} = \max \limits_{1\leq j \leq J} \cons_{Qj}$.
\end{lemma}

\subsection{Stability of $\Soph$}
Controlling the term \eqref{eqn:soph_pert} requires an estimate of the stability of \net operator $\Soph$. First, the assumption \ref{ass:data_bnd} leads to the following bounds for all $(f,\theta,\eta) \in \dbX$
\begin{equation}\label{eqn:l2bnd}
\begin{aligned}
\|\Fh\|_2 \leq \sqrt{k}\cons_\dbF, \quad
\|\Th\|_2 \leq  \sqrt{k} \cons_\mathcal{T}, \quad
\|\Nh\|_2 \leq \sqrt{k^\prime} \cons_\dbN.
\end{aligned}
\end{equation}

Next, to ensure the \net output is bounded, we make the following assumption.
\begin{assumption}\label{ass:bound_net}
The various sub-components of the \net are bounded. In other words, there exist positive constants $\cons_\A,\cons_\At,\cons_\tau, < \infty$ such that
\begin{equation}\label{eqn:net_bnds}
    \|\A\|_2 \leq \cons_\A,  \quad
    \|\At\|_2 \leq \cons_\At, \quad
    \|\btau\|_{L^2(\Omega)} \leq \cons_\tau,
\end{equation}
\end{assumption}
where it is understood that the $L^2$ norm of a vector-valued function is given by
\[
\|\btau\|_{L^2(\Omega)} = \left( \sum_{i=1}^p \|\tau_i\|^2_{L^2(\Omega)}\right)^{1/2}.
\]
The bounds in the above assumption are typically ensured by the fact the the input data to the \net is bounded, and the use of regularization techniques to control the magnitude of the trainable parameters of the \net.

We also rely on the following simple result to obtain several estimates moving forward.
\begin{lemma}\label{lem:ineq}
Let $\bm{f}\in L^2(\Omega;\Ro^n)$, $\bm{v}\in\Ro^n$ and define $\ell = \bm{v}^\top \bm{f}$. Then, $\ell \in L^2(\Omega)$ and
\[
\|\ell\|_{L^2(\Omega)} \leq  \|\bm{v}\|_2 \|\bm{f}\|_{L^2(\Omega)}.
\]
\end{lemma}
\begin{proof}
The result can obtained by a simple application of Cauchy-Schwarz
\begin{eqnarray*}
\|\ell\|^2_{L^2(\Omega)} = \int_\Omega \left(\bm{v}^\top \bm{f}(\x) \right)^2 \ud \x = \int_\Omega \left( \sum_{i=1}^n v_i f_i(\x) \right)^2 \ud \x &\leq& \int_\Omega \left( \sum_{i=1}^n v_i^2 \right) \left( \sum_{i=1}^n f_i^2(\x) \right) \ud \x \\
&=&\|\bm{v}\|^2_2 \sum_{i=1}^n \|f_i\|^2_{L^2(\Omega)} = \|\bm{v}\|^2_2 \|\bm{f}\|^2_{L^2(\Omega)}.
\end{eqnarray*}
\end{proof}

We can now prove the following stability estimate for the \net.
\begin{theorem}\label{thm:net_stability}
Assume the data is bounded in accordance to Assumption \ref{ass:data_bnd}. Define the set $\widehat{\dbT} = \{\Th : \Th = \{\theta(\xh_i)\}_{i=1}^k \ \forall \ \theta \in \dbT \} \subset \Ro^k$, where $\{\xh_i\}_{i=1}^k$ are the \net sensor nodes for $\theta$. Assume the non-linear branch $\D:\widehat{\dbT} \rightarrow \Ro^{p \times p}$ of the trained \net $\Soph$ is Lipschitz
    \begin{equation}\label{eqn:D_lip}
    \|\D(\widehat{\Tb}) - \D(\widehat{\Tb}^\prime) \|_2 \leq L_D \|\widehat{\Tb} - \widehat{\Tb}^\prime \|_2 \quad \forall \ \Th, \Th^\prime \in \widehat{\dbT},
\end{equation}
with Lipschitz constant $L_D$. Let $\|\D(\Th_0)\|_2$ be bounded for some $\Th_0 \in \widehat{\dbT}$. We also assume that the trained \net is bounded in the sense of Assumption \ref{ass:bound_net}.
 Then for any  $(f,\theta,\eta),(f^\prime, \theta^\prime,\eta^\prime) \in \dbX$, the following stability estimate holds
\begin{align}\label{eqn:net_stable}
\Erh_\text{stab}[(f,\theta,\eta),(f^\prime,\theta^\prime,\eta^\prime)] \leq& \widehat{\cons}_\text{stab} \left( \|\Fh - \Fh^\prime \|_2 + \|\Th - \Th^\prime \|_2 + \|\Nh - \Nh^\prime \|_2\right),
\end{align}
where $\widehat{\cons}_\text{stab}$ depends on the bounds in \eqref{eqn:data_bnd} and \eqref{eqn:net_bnds}, the number of sensor nodes, and $L_D$.
\end{theorem}
\begin{proof}
For any $(f, \theta,\eta),(f^\prime, \theta^\prime,\eta^\prime) \in \dbX$, we have using \eqref{eqn:net}and \eqref{eqn:soph_pert}
\begin{eqnarray}\label{eqn:net_stable1}
  \Erh_\text{stab}[(f,\theta,\eta),(f^\prime,\theta^\prime,\eta^\prime)] &=& \|\Soph(\Fh, \Th,\Nh) - \Soph(\Fh^\prime, \Th^\prime,\Nh^\prime)\|_{L^2(\Omega)} \notag \\
  &\leq& \underbrace{\|\big(\bbeta(\Fh,\Th) - \bbeta(\Fh^\prime,\Th^\prime)\big)^\top \btau\|_{L^2(\Omega)}}_{\Er_{\bbeta}} \notag\\
  &&+ \underbrace{\|\big(\bbetat(\Nh,\Th) - \bbetat(\Nh^\prime,\Th^\prime)\big)^\top \btau\|_{L^2(\Omega)}}_{\Er_{\bbetat}}.
\end{eqnarray}
We can use the Lipschitz assumption on $\D$ and \eqref{eqn:l2bnd} to prove that $\D$ is bounded on $\widehat{\dbT}$. To see this, we consider $\Th_0 \in \widehat{\dbT}$ defined in the theorem statement to get
\begin{align}\label{eqn:D_bnd}
   \|\D(\Th)\|_2 - \|\D(\Th_0)\|_2 &\leq  \|\D(\Th) - \D(\Th_0)\|_2
    \leq L_D \|\Th - \Th_0 \|_2 \quad \forall \  \Th \in \widehat{\dbT}, \notag \\
    \implies \|\D(\Th)\|_2  &\leq \underbrace{\|\D(\Th_0)\|_2 + 2 L_D \sqrt{k} \cons_\mathcal{T}^2 }_{ \cons_D} \quad \forall \  \Th \in \widehat{\dbT}.
\end{align}

Next, using \eqref{eqn:l2bnd}, \eqref{eqn:D_lip}, \eqref{eqn:D_bnd}, \eqref{eqn:net_bnds} and Lemma \ref{lem:ineq}, we get the estimate
\begin{eqnarray}\label{eqn:err_beta}
 \Er_{\bbeta} &=& \|\big(\bbeta(\Fh,\Th) - \bbeta(\Fh,\Th^\prime) + \bbeta(\Fh,\Th^\prime) - \bbeta(\Fh^\prime,\Th^\prime)\big)^\top \btau\|_{L^2(\Omega)} \notag \\
&=&\| \Big( \D(\Th) - \D(\Th^\prime) \big)\A \Fh + \D(\Th^\prime) \A \big(\Fh - \Fh^\prime \big)  \Big)^\top \btau \|_{L^2(\Omega)} \notag \\
&\leq&   \|\A\|_2\|\btau\|_{L^2(\Omega)} \left(L_D\| \Th - \Th^\prime \|_2 \|\Fh\|_2+ \|\D(\Th^\prime)\|_2 \|\Fh - \Fh^\prime \|_2\right) \notag \\
&\leq&   \|\A\|_2\|\btau\|_{L^2(\Omega)} \left(L_D \sqrt{k} \cons_\dbF \| \Th - \Th^\prime \|_2 + \cons_D \|\Fh - \Fh^\prime \|_2\right) \notag \\
&\leq&\wh{\cons}_1 \left( \| \Th - \Th^\prime \|_2 + \|\Fh - \Fh^\prime \|_2\right),
\end{eqnarray}
where
\[
\wh{\cons}_1 = \cons_\A \cons_\tau  \max \{L_D \sqrt{k} \cons_\dbF ,\cons_D\}.
\]
Similarly
\begin{eqnarray}\label{eqn:err_betat}
 \Er_{\bbetat} &=& \|\big(\bbetat(\Nh,\Th) - \bbetat(\Nh,\Th^\prime) + \bbetat(\Nh,\Th^\prime) - \bbetat(\Nh^\prime,\Th^\prime)\big)^\top \btau\|_{L^2(\Omega)} \notag \\
&=&\| \Big( \D(\Th) - \D(\Th^\prime) \big)\At \Nh + \D(\Th^\prime) \At \big(\Nh - \Nh^\prime \big)  \Big)^\top \btau \|_{L^2(\Omega)} \notag \\
&\leq&   \|\At\|_2\|\btau\|_{L^2(\Omega)} \left(L_D\| \Th - \Th^\prime \|_2 \|\Nh\|_2+ \|\D(\Th^\prime)\|_2 \|\Nh - \Nh^\prime \|_2\right) \notag \\
&\leq&   \|\At\|_2\|\btau\|_{L^2(\Omega)} \left(L_D \sqrt{k^\prime} C_\dbN\| \Th - \Th^\prime \|_2 + \cons_D \|\Nh - \Nh^\prime \|_2\right) \notag \\
&\leq&\wh{\cons}_2 \left( \| \Th - \Th^\prime \|_2 + \|\Nh - \Nh^\prime \|_2\right),
\end{eqnarray}
where
\[
\wh{\cons}_2 = \cons_\At \cons_\tau \max \{L_D \sqrt{k^\prime} \cons_\dbN ,\cons_D\}.
\]
Combining \eqref{eqn:net_stable1}, \eqref{eqn:err_beta} and \eqref{eqn:err_betat}, and setting $\wh{\cons}_{stab} = 2\max\{\wh{\cons}_1,\wh{\cons}_2\}$ proves the result. 
\end{proof}

\begin{remark}
In \cite{bonito2013}, it was shown for the homogeneous Dirichlet problem that if $\nabla u \in L^r(\Omega)$ for some $r \geq 2$, then the solution operator (measured in $H^1_0(\Omega)$) is stable with respect to perturbations in $\theta$ measured in $L^s$ for $s = 2r/(r-2)$. An analogue of such a result for the general elliptic mixed-boundary value problem considered in the present manuscript might change the norms for measuring $\theta$ in \eqref{eqn:stable_sop}, as well as the discrete norms in \eqref{eqn:D_lip}. This norm consistency might be useful in ensuring that the Lipschitz constant $L_D$ is stable to variations in the latent dimension $p$. These concepts will be explored in future work.
\end{remark}

\subsection{Covering estimates and quadrature errors}\label{sec:covering}

In addition to the data being bounded, we also require the training set to sufficiently cover the entire data space. As a consequence of compactness of $\dbX$ in Assumption \ref{ass:data_bnd}, given any $\epsilon_s >0$, we can find a positive integer $J :=J(\epsilon_s)$ such that $\dbX$ possesses a $\epsilon_s$-net with $J$ elements. In particular, there exists a set $W:=\{(f_j,\theta_j,\eta_j)\}_{j=1}^{J} \subset \dbX$ such that for any $(f,\theta,\eta)\in\dbX$
\begin{equation}\label{eqn:covering_L2}
    \|f-f_j\|_{L^2(\Omega)} < \epsilon_s, \quad \|\theta - \theta_j\|_{L^2(\Omega)} < \epsilon_s, \quad \|\eta - \eta_j\|_{L^2(\Gamma_\eta)} < \epsilon_s \quad \text{for some } 1 \leq j \leq J.
\end{equation}
We make the following assumption about the training dataset.
\begin{assumption}\label{ass:covering}
Assuming $\dbX$ to be compact and given $\epsilon_s >0$, let $W \subset \dbX$ be with $J$ elements such that \eqref{eqn:covering_L2} holds. We assume that the training set $\mathbb{S}$ is generated using such an $\epsilon_s$-net $W$.
\end{assumption}

Next, we need suitable conditions on the sensor nodes to allow us to bound the norm of the input vectors of the \net with the corresponding $L^2$ norms.

\begin{assumption}\label{ass:sensor_quad}
The sensor nodes $\{\xh_i\}_{i=1}^k$ and $\{\xh^b_i\}_{i=1}^{k^\prime}$ are chosen to be suitable quadrature nodes to approximate the square integrals of $g:\Omega \rightarrow \Ro$ and $g_b:\Gamma_\eta \rightarrow \Ro$ as
\begin{equation}\label{eqn:sensor_quad}
    \left | \sum_{i=1}^k \wh{w}_i g^2(\xh_i) -  \int_\Omega g^2(\x) \ud \x \right| \leq  \frac{\wh{\cons}_Q(g)}{k^{\alpha}}, \quad \left | \sum_{i=1}^{k^\prime} \wh{w}^b_i g_b^2(\xh^b_i) -  \int_{\Gamma_\eta} g_b^2(\x) \ud \x \right| \leq  \frac{\wh{\cons}_{b,Q}(g_b)} {{(k^\prime)}^{\alpha^\prime}},
\end{equation}
where $\{\wh{w}_i\}_{i=1}^k$ and $\{\wh{w}^b_i\}_{i=1}^{k^\prime}$ are the quadrature weights. Further, the weights are assumed to be positive.
\end{assumption}

As a consequence of Assumption \ref{ass:sensor_quad} the following estimates hold:
\begin{lemma}\label{lem:L2l2}
Let $g:\Omega \rightarrow \Ro$, $g_b:\Gamma_\eta \rightarrow \Ro$, and the sensors be chosen according to Assumption \ref{ass:sensor_quad}. Define the vector $\wh{\bm{G}} = (g(\xh_1),\cdots g(\xh_k))^\top \in \Ro^k$, $\wh{\bm{G}_b} = (g_b(\xh^b_1),\cdots g_b(\xh^b_{k^\prime}))^\top \in \Ro^{k^\prime}$, and the diagonal matrices $\wh{\W}_{\frac{1}{2}}  = \text{diag}(\sqrt{\wh{w}_1},\cdots \sqrt{\wh{w}_k}) \in \Ro^{k \times k}$, $\wh{\W}^b_{\frac{1}{2}}  = \text{diag}(\sqrt{\wh{w}^b_1},\cdots \sqrt{\wh{w}^b_{k^\prime}}) \in \Ro^{k^\prime \times k^\prime}$. Then
\begin{equation}
    \|\wh{\bm{G}}\|_2 \leq \|\wh{\W}^{-1}_{\frac{1}{2}}\|_2\left( \|g\|_{L^2(\Omega)}+ \frac{\sqrt{\wh{\cons}_Q(g) }}{ k^{\alpha/2}} \right), \quad \|\wh{\bm{G}}_b\|_2 \leq \|(\wh{\W}^{b}_{\frac{1}{2}})^{-1}\|_2\left( \|g_b\|_{L^2(\Gamma_\eta)}+ \frac{\sqrt{\wh{\cons}_{b,Q}(g_b) }}{ (k^\prime)^{\alpha^\prime/2}} \right),
\end{equation}
where $\wh{\cons}_Q, \wh{\cons}_{b,Q}, \alpha, \alpha^\prime$ are as defined in Assumption \ref{ass:sensor_quad}.
\end{lemma}
\begin{proof}
Let us first focus on the square integral of $g$. Note that
\begin{eqnarray*}
    \sum_{i=1}^k \wh{w}_i g^2(\xh_i) = \sum_{i=1}^k \left(\sqrt{\wh{w}_i} g(\xh_i) \right)^2 = \|\wh{\W}_{\frac{1}{2}} \wh{\bm{G}}\|_2^2.
\end{eqnarray*}
Then, using \eqref{eqn:sensor_quad} we have
\begin{eqnarray*}
    \|\wh{\bm{G}}\|_2^2 \leq \| \wh{\W}^{-1}_{\frac{1}{2}}\|^2_2\|\wh{\W}_{\frac{1}{2}} \wh{\bm{G}}\|_2^2 \leq \| \wh{\W}^{-1}_{\frac{1}{2}}\|^2_2 \left(
    \|g\|^2_{L^2(\Omega)}  + \frac{\wh{\cons}_Q(g)}{ k^{\alpha}} \right),
\end{eqnarray*}
which implies
\[
\|\wh{\bm{G}}\|_2 \leq \|\wh{\W}^{-1}_{\frac{1}{2}}\|_2\left(\|g\|^2_{L^2(\Omega)} + \frac{\wh{\cons}_Q(g)}{ k^{\alpha}}\right)^{1/2} \leq \|\wh{\W}^{-1}_{\frac{1}{2}}\|_2\left( \|g\|_{L^2(\Omega)}+ \frac{\sqrt{\wh{\cons}_Q(g) }}{ k^{\alpha/2}} \right).
\]
The estimate for $\|\wh{\bm{G}}_b\|_2 $ can be obtained in a similar manner.
\end{proof}

\subsection{Generalization error}\label{sec:gen_err}
We can finally combine all the assumptions and results from the previous sections to bound the generalization error \eqref{eqn:err_true_net}.

\begin{theorem}\label{thm:gen_err}
Consider a \net trained on a dataset $\mathbb{S}$ to approximate $\Sop$. Let us assume:
\begin{enumerate}
    \item Assumption \ref{ass:stable_sop} holds, which ensures that $\Sop$ is stable and \eqref{eqn:stable_sop} holds.
    \item Assumption \ref{ass:num_approx} holds, which expresses how well $\Sop$ is approximated by $\Sop^h$.
    \item Assumption \ref{ass:covering} holds, which ensures that given $\epsilon_s > 0$, \eqref{eqn:covering_L2} holds for the set $W:=\{(f_j,\eta_j,\theta_j)\}_{j=1}^{J}$ used to generate the training set $\mathbb{S}$, where $J=J(\epsilon_s)$.
    \item $\Soph$ is $\epsilon_t$-trained in the sense of Definition \ref{def:well_trained}.
    \item The assumptions in the statement of Theorem \ref{thm:net_stability} hold, which ensures that $\Soph$ is stable.
    \item The sensor nodes are chosen in accordance to Assumption \ref{ass:sensor_quad}. Furthermore, the constants 
    \begin{subequations}\label{eqn:sensor_quad_constants}
\begin{align}
    \wh{\cons}_{Q,\dbF} = \sup_{f,f^\prime \in \dbF} \wh{\cons}_Q(f-f^\prime), \quad \wh{\cons}_{Q,\dbT} = \sup_{\theta,\theta^\prime \in \dbT} \wh{\cons}_Q(\theta-\theta^\prime), \quad \wh{\cons}_{b,Q,\dbN} = \sup_{\eta,\eta^\prime \in \dbN} \wh{\cons}_{b,Q}(\eta-\eta^\prime),
\end{align}
\end{subequations}
    are finite, where $\wh{\cons}_Q,\wh{\cons}_{b,Q}$ are as defined in \eqref{eqn:sensor_quad}.
\end{enumerate}
Then we have the following estimates for the generalizaition error
\begin{equation}\label{eqn:gen_err}
    \Er(f,\theta,\eta) \leq  \cons_{h,\mathbb{S}}\epsilon_h + \ol{\cons}_S \epsilon_s + \sqrt{\epsilon_t} + \ol{\cons}_Q \left( \frac{1}{k^{\alpha/2}} + \frac{1}{(k^\prime)^{\alpha^\prime/2}} + \frac{1}{L^{\gamma/2}}\right),
\end{equation}
where $L,\gamma$ are as defined in Lemma \ref{lem:welltrained}, $k,k^\prime,\alpha,\alpha^\prime$ are as defined in \eqref{eqn:sensor_quad} and  $\cons_{h,\mathbb{S}}$ is defined in \eqref{eqn:num_err_constants}. Further, $\ol{\cons}_S, \ol{\cons}_Q$ depend on the stability constants in \eqref{eqn:stable_sop} and \eqref{eqn:net_stable}, the quadrature weights of \eqref{eqn:sensor_quad}, and the constants in \eqref{eqn:sensor_quad_constants}. 
\end{theorem}

\begin{proof}
Given a $(f, \eta,\theta) \in \dbX$ and $\epsilon_s > 0$, the compactness of $\dbX$ and Assumption \ref{ass:covering} ensures we can find an index $j := j(f,\theta,\eta) \leq J$ of the $\mathbb{S}$ such that \eqref{eqn:covering_L2} holds. Furthermore, by assuming that the sensor nodes are chosen according to Lemma \ref{lem:L2l2} and assuming the finiteness of the constants defined in \eqref{eqn:sensor_quad_constants}, we can bound
\begin{subequations}\label{eqn:L2l2diff}
\begin{align}
\|\Fh - \Fh_j\|_2 &\leq \|\wh{\W}_{\frac{1}{2}}^{-1}\|_2 \left(\|f - f_j\|_{L^2(\Omega)} + \frac{\sqrt{\wh{\cons}_Q(f - f_j) }}{ k^{\alpha/2}} \right) <  \|\wh{\W}_{\frac{1}{2}}^{-1}\|_2 \left(\epsilon_s + \frac{\sqrt{\wh{\cons}_{Q,\dbF}}}{ k^{\alpha/2}} \right),\\
\|\Th - \Th_j\|_2 &\leq \|\wh{\W}_{\frac{1}{2}}^{-1}\|_2 \left(\|\theta - \theta_j\|_{L^2(\Omega)} + \frac{\sqrt{\wh{\cons}_Q(\theta - \theta_j) }}{ k^{\alpha/2}} \right) <  \|\wh{\W}_{\frac{1}{2}}^{-1}\|_2 \left(\epsilon_s + \frac{\sqrt{\wh{\cons}_{Q,\dbT}}}{ k^{\alpha/2}} \right),\\
\|\Nh - \Nh_j\|_2 &\leq \|(\wh{\W}^b_{\frac{1}{2}})^{-1}\|_2 \left(\|\eta - \eta_j\|_{L^2(\Gamma_\eta)} + \frac{\sqrt{\wh{\cons}_{b,Q}(\eta - \eta_j) }}{ k^{\alpha/2}} \right) <  \|(\wh{\W}^b_{\frac{1}{2}})^{-1}\|_2 \left(\epsilon_s + \frac{\sqrt{\wh{\cons}_{b,Q,\dbN}}}{ (k^\prime)^{\alpha^\prime/2}} \right).
\end{align}
\end{subequations}
Using \eqref{eqn:gen_err_breakup}, \eqref{eqn:stable_sop}, \eqref{eqn:num_approx}, \eqref{eqn:lossapprox}, \eqref{eqn:net_stable}, \eqref{eqn:num_err_constants}, \eqref{eqn:covering_L2} and \eqref{eqn:L2l2diff}, we obtain the estimate
\begin{eqnarray*}
 \Er(f,\theta,\eta) &<& \cons_\text{stab} \left( \|f - f_j \|_{L^2(\Omega)} + \|\theta - \theta_j \|_{L^2(\Omega)} + \|\eta - \eta_j \|_{L^2(\Gamma_\eta)}\right) + \cons_h(f_j,\theta_j,\eta_j)\epsilon_h \notag \\
 &&+ \sqrt{\epsilon_t} + \frac{\sqrt{\cons_{Q,\mathbb{S}}}}{L^{\gamma/2}}  + \widehat{\cons}_\text{stab} \left( \|\Fh - \Fh^\prime \|_2 + \|\Th - \Th^\prime \|_2 + \|\Nh - \Nh^\prime \|_2\right)\notag \\
 &\leq& 3\cons_\text{stab} \epsilon_s + \cons_{h,\mathbb{S}}\epsilon_h + \sqrt{\epsilon_t}  +  \frac{\sqrt{\cons_{Q,\mathbb{S}}}}{L^{\gamma/2}} + \widehat{\cons}_\text{stab} \left( 2\|\wh{\W}_{\frac{1}{2}}^{-1}\|_2 + \|(\wh{\W}^b_{\frac{1}{2}})^{-1}\|_2 \right) \epsilon_s \notag \\
 && +  \wh{\cons}_\text{stab} \left( \|\wh{\W}_{\frac{1}{2}}^{-1}\|_2\frac{\left(\sqrt{\wh{\cons}_{Q,\dbF}} + \sqrt{\wh{\cons}_{Q,\dbT}}\right)}{ k^{\alpha/2}} + \|(\wh{\W}^b_{\frac{1}{2}})^{-1}\|_2\frac{\sqrt{\wh{\cons}_{b,Q,\dbN} }}{ (k^\prime)^{\alpha^\prime/2}}\right) \\
 &\leq& \cons_{h,\mathbb{S}}\epsilon_h + \ol{\cons}_S \epsilon_s + \sqrt{\epsilon_t} + \ol{\cons}_Q \left( \frac{1}{k^{\alpha/2}} + \frac{1}{(k^\prime)^{\alpha^\prime/2}} + \frac{1}{L^{\gamma/2}}\right)
\end{eqnarray*}
where
\begin{align*}
    \ol{\cons}_S &= 3 \cons_\text{stab} + \widehat{\cons}_\text{stab}\left( 2\|\wh{\W}_{\frac{1}{2}}^{-1}\|_2 + \|(\wh{\W}^b_{\frac{1}{2}})^{-1}\|_2 \right),\\
    \ol{\cons}_Q &= \max\{\wh{\cons}_\text{stab}\|\wh{\W}_{\frac{1}{2}}^{-1}\|_2\left(\sqrt{\wh{\cons}_{Q,\dbF}} + \sqrt{\wh{\cons}_{Q,\dbT}}\right) \ , \ \wh{\cons}_\text{stab} \|(\wh{\W}^b_{\frac{1}{2}})^{-1}\|_2 \sqrt{\wh{\cons}_{b,Q,\dbN} } \ , \ \sqrt{\cons_{Q,\mathbb{S}}} \}.
\end{align*}

\end{proof}

As a consequence of Theorem \ref{thm:gen_err}, we also obtain an estimate for the error \eqref{eqn:err_num_net} for any $(f,\theta,\eta) \in \dbX$.

\begin{corollary}\label{cor:gen_err_num}
Let the assumptions of Theorem \ref{thm:gen_err} hold. Then
\begin{equation}\label{eqn:gen_err_num}
\wh{\Er}(f,\theta,\eta) \leq  2\cons_{h,\dbX}\epsilon_h + \ol{\cons}_S \epsilon_s + \sqrt{\epsilon_t} + \ol{\cons}_Q \left( \frac{1}{k^{\alpha/2}} + \frac{1}{(k^\prime)^{\alpha^\prime/2}} + \frac{1}{L^{\gamma/2}}\right),
\end{equation}
where $\cons_{h,\dbX}$ is defined in \eqref{eqn:num_err_constants}.
\end{corollary}
\begin{proof}
Noting that $C_{h,\mathbb{S}} \leq \cons_{h,\dbX}$, we use \eqref{eqn:num_approx} and \eqref{eqn:gen_err} to get
\begin{eqnarray*}
 \wh{\Er}(f,\theta,\eta) &=& \|\Sop^h\circ\mathcal{P}(f,\theta,\eta) - \Soph\circ\Ph(f,\eta,\theta )\|_{L^2(\Omega)} \notag \\
 &\leq& \|\Sop^h\circ\mathcal{P}(f,\theta,\eta) - \Sop(f,\theta,\eta)\|_{L^2(\Omega)} + \|\Sop(f,\theta,\eta) - \Soph\circ\Ph(f,\eta,\theta )\|_{L^2(\Omega)} \notag \\
 &\leq&  C_h(f,\theta,\eta)\epsilon_h + \Er(f,\theta,\eta) \notag\\
 &\leq& \cons_{h,\dbX} \epsilon_h + \Er(f,\theta,\eta) \notag\\
 &\leq& 2\cons_{h,\dbX}\epsilon_h + \ol{\cons}_S \epsilon_s + \sqrt{\epsilon_t} + \ol{\cons}_Q \left( \frac{1}{k^{\alpha/2}} + \frac{1}{(k^\prime)^{\alpha^\prime/2}} + \frac{1}{L^{\gamma/2}}\right).
\end{eqnarray*}
\end{proof}

We make some remarks here:
\begin{itemize}
\setlength\itemsep{0pt}
\item The expression \eqref{eqn:gen_err} makes it clear that the \net when applied to a test data in $\mathcal{X}$ leads to four distinct error terms. The first corresponds to the error in approximating the true analytical (weak) solution of the PDE by the numerical solution to \eqref{eqn:d_weak}, which is used to generate the training data for the \net. The second term is governed by how well the training data covers the PDE data space. The third is the training error of the \net, which depends on the capacity/size of the network and the algorithms used to train it. Finally, the fourth term is the accumulated quadrature error corresponding to the input evaluated at the finite sensor nodes, and the loss term in \eqref{eqn:loss} approximating the true $L^2$ error. The constants $\ol{\cons}_S, \ol{\cons}_Q$ depend on the stability constants of the PDE and the \net.  
\item If we judiciously increase the number of training samples in $\mathbb{S}$, i.e., $J \uparrow$, we can reduce the radius of the cover, i.e., $\epsilon_s \downarrow$. In order to continue satisfying \eqref{eqn:err_training_samples}, we may need to increase the capacity (size) of the \net. Note that for a fixed \net latent dimension $p$, the architecture only allows changing the size of the non-linear branch and the trunk.
\item The quadrature errors can be reduced by either increasing the number of sensor and output nodes, or by using the nodes (and weights) corresponding to a higher-order quadrature.
\item We cannot in general hope for the \net error to be smaller than the error corresponding to the training samples in approximating the exact solution. This puts a practical lower bound on the monitored validation error achievable.
\item We have assumed that $f$ and $\theta$ are sampled at the same sensor nodes for simplicity of the discussion. However, it is possible to sample the functions at different sets of sensor nodes, which would lead to an error estimate similar to \eqref{eqn:gen_err} but now including separate quadrature error terms for $f$ and $\theta$.

\end{itemize}

\subsection{Structural estimates}\label{sec:Kapprox}
We now demonstrate another advantage of the \net architecture, in that we can approximate additional structures arising in the discrete weak formulation \eqref{eqn:d_weak}. To derive the estimates in this section, we will work with the particular setup where
\begin{itemize}
    \item The number of output nodes of the \net equals the dimension of the space $\mathcal{V}^h$ defined in Section \ref{sec:fem}, i.e., $L=q$.
    \item We restrict the PDE-data to the projection space $\dbX^h$ defined in \eqref{eqn:data_projector}. Note that $\mathcal{P}$ restricted to $\dbX^h$ is the identity map. 
\end{itemize}

We define the following matrices associated with the sensor nodes $\{\xh_i\}_{i=1}^k$ and the output nodes $\{\x_l\}_{l=1}^q$
\begin{equation}\label{eqn:Vmats}
    \Vh \in \Ro^{k \times q}, \quad \wh{V}_{ij} = \phi_j(\xh_i), \quad \V \in \Ro^{q \times q}, \quad V_{ij} = \phi_j(\x_i), \quad \T \in \Ro^{q \times p}, \quad T_{ij} = \tau_j(\x_i),
\end{equation}
where $\{\phi_i\}_{i=1}^q$ are the basis functions of $\mathcal{V}^h$ while $\btau = \{\tau_i\}_{i=1}^p$ are the \net trunk basis functions. For $f \in \dbF^h \subset \dbF$, the vector of nodal values can be evaluated as by $\Fh = \wh{\V} \F$.

We now prove the following result that shows that we can approximate the operator $\K^{-1}(\theta)$ arising in \eqref{eqn:d_weak} using the trained \net components.

\begin{theorem}\label{thm:Kinv_approx}
Consider a \net $\Soph$ trained on a dataset $\mathbb{S}$ to approximate $\Sop$. Furthermore let:
\begin{enumerate}
\item The assumptions in the statement of Corollary \ref{cor:gen_err_num} hold true, which ensures \eqref{eqn:gen_err_num} is satisfied for all $(f,\theta,\eta) \in \mathcal{X}$, and in particular on $\dbX^h$. 
\item The output nodes be chosen such that the matrix $\V$ is invertible.
\item The quadrature weights $\{w_l\}_{l=1}^q$ defined in \eqref{eqn:quad_out} corresponding to the output nodes are positive. 
\item For the set  
\[
\mathcal{G} = \{(u,u^h) \in \dbV \times \dbV^h : u = \Sop(f,\theta,\eta), \ u^h = \Sop \circ \mathcal{P}(f,\theta,\eta), \quad (f,\theta,\eta) \in \dbX^h \},
\]
the constant
\[
C_{Q,h} = \sup_{(u^h,u) \in \mathcal{G}}C_Q(u^h-u),
\]
is finite, where $C_Q$ is the quadrature constant in \eqref{eqn:quad_out}.
\item $\{f \in L^2(\Omega) \cap C(\Omega) : \|\F\|_2 = 1 \} \subset \dbF^h$.
\end{enumerate}
Then we have the estimate
\begin{eqnarray}\label{eqn:Kinv_approx}
\|\K^{-1}(\theta^h) - \Q(\Th) \|_2 &\leq& \Bigg[2\cons_{h,\dbX}\epsilon_h + \ol{\cons}_S \epsilon_s + \sqrt{\epsilon_t} \notag \\
    && + \widetilde{C}_Q \left( \frac{1}{k^{\alpha/2}} + \frac{1}{(k^\prime)^{\alpha^\prime/2}} + \frac{1}{q^{\gamma/2}}\right) \Bigg]\|\W_{\frac{1}{2}}^{-1}\|_2\|\V^{-1}\|_2\|\M^{-1}\|_2,
\end{eqnarray}
where $\Q(\Th) = \V^{-1}\T \D(\Th) \A \wh{\V}\M^{-1} \in \Ro^{q\times q}$, and $\W_{\frac{1}{2}} = \text{diag}(\sqrt{w_1},\cdots \sqrt{w_q}) \in \Ro^{q \times q}$ with $\{w_l\}_{l=1}^q$ being the quadrature weights defined in \eqref{eqn:quad_out}.
\end{theorem}
\begin{proof}
The vector of nodal values $\U^{nodal}$ and $\wh{\U}^{nodal}$ of the discrete weak solution satisfying \eqref{eqn:d_weak} and the \net solution, respectively, at the output nodes are given by
\begin{equation}\label{eqn:nodal_vals}
\U^{nodal} = \V \K^{-1}(\theta^h)\M \F , \qquad  \wh{\U}^{nodal} = \T \D(\Th)\A \Fh = \T \D(\Th)\A \Vh \F,   
\end{equation}
where we have assumed $\eta=0$ (which implies $\N=\bm{0}$). Thus, we have
\begin{eqnarray}\label{eqn:err_quad_1}
    \left(\sum_{i=1}^q w_i \left( (\Sop^h \circ \mathcal{P} - \Soph \circ \Ph)(f,\theta,0)[\x_i] \right)^2\right)^{1/2}
    &=& \left(\sum_{i=1}^q w_i \left( U^{nodal}_i - \wh{U}^{nodal}_i \right)^2\right)^{1/2}\notag\\
    &=& \|\W_{\frac{1}{2}} (\U^{nodal} -\wh{\U}^{nodal})\|_2\notag\\
    &=& \|\W_{\frac{1}{2}} \left(\V \K^{-1}(\theta^h) \M \F -\T \D(\Th) \A \Vh \F\right)\|_2 \notag \\
    &=&\|\W_{\frac{1}{2}} \V \left( \K^{-1}(\theta^h) -\V^{-1}\T \D(\Th) \A \wh{\V} \M^{-1}\right)\M \F \|_2 \notag \\
    &=& \|\W_{\frac{1}{2}}  \V \left(\K^{-1}(\theta^h) -\Q(\Th) \right)\M \F\|_2.
\end{eqnarray}

Next, we combining \eqref{eqn:quad_out} and \eqref{eqn:gen_err_num}, we get
\begin{eqnarray}\label{eqn:err_quad_2}
    \left(\sum_{i=1}^q w_i \left( (\Sop^h \circ \mathcal{P} - \Soph \circ \Ph)(f,\theta,0)[\x_i] \right)^2\right)^{1/2} &\leq& \wh{\Er}(f,\theta,0) + \frac{\sqrt{C_Q(u^h-\wh{u})}}{ q^{\gamma/2}} \notag \\
    &\leq& 2\cons_{h,\dbX}\epsilon_h + \ol{\cons}_S \epsilon_s + \sqrt{\epsilon_t} \notag \\
    && + \ol{\cons}_Q \left( \frac{1}{k^{\alpha/2}} + \frac{1}{(k^\prime)^{\alpha^\prime/2}} + \frac{1}{q^{\gamma/2}}\right)  + \frac{\sqrt{C_{Q,h}}}{q^{\gamma/2}} \notag \\
    &\leq& 2\cons_{h,\dbX}\epsilon_h + \ol{\cons}_S \epsilon_s + \sqrt{\epsilon_t} \notag \\
    && + \widetilde{C}_Q \left( \frac{1}{k^{\alpha/2}} + \frac{1}{(k^\prime)^{\alpha^\prime/2}} + \frac{1}{q^{\gamma/2}}\right),
\end{eqnarray}
where $\widetilde{C}_Q = \ol{\cons}_Q + \sqrt{C_{Q,h}}$. Finally, combining \eqref{eqn:err_quad_1} and \eqref{eqn:err_quad_2}, we get required estimate
\begin{eqnarray*}
  &&\|\W_{\frac{1}{2}}  \V \left(\K^{-1}(\theta^h) -\Q(\Th) \right)\M \F\|_2 \leq 2\ol{\cons}_{h}\epsilon_h + \ol{\cons}_S \epsilon_s + \sqrt{\epsilon_t} \notag \\
    && \hspace{6cm}+ \widetilde{C}_Q \left( \frac{1}{k^{\alpha/2}} + \frac{1}{(k^\prime)^{\alpha^\prime/2}} + \frac{1}{q^{\gamma/2}}\right) \\
  &\implies& \|\W_{\frac{1}{2}}  \V \left(\K^{-1}(\theta^h) -\Q(\Th) \right)\M \|_2 \leq 2\ol{\cons}_{h}\epsilon_h + \ol{\cons}_S \epsilon_s + \sqrt{\epsilon_t} \notag \\
    && \hspace{6cm}+ \widetilde{C}_Q \left( \frac{1}{k^{\alpha/2}} + \frac{1}{(k^\prime)^{\alpha^\prime/2}} + \frac{1}{q^{\gamma/2}}\right) \\
  &\implies& \|\K^{-1}(\theta^h) - \Q(\Th) \|_2 \leq \Bigg[2\ol{\cons}_{h}\epsilon_h + \ol{\cons}_S \epsilon_s + \sqrt{\epsilon_t} \notag \\
    && \hspace{4cm}+ \widetilde{C}_Q \left( \frac{1}{k^{\alpha/2}} + \frac{1}{(k^\prime)^{\alpha^\prime/2}} + \frac{1}{q^{\gamma/2}}\right) \Bigg]\|\W_{\frac{1}{2}}^{-1}\|_2\|\V^{-1}\|_2\|\M^{-1}\|_2.
\end{eqnarray*}
\end{proof}

\section{Extension to nonlinear problems}\label{sec:nonlinear}
In this section we describe the extension of the \net to nonlinear partial differential equations. To fix ideas, we consider a time-independent nonlinear advection-diffusion-reaction equation with inhomogeneous source term and Dirichlet and Neumann boundary conditions. The problem is given by,
\begin{eqnarray}
- \nabla \cdot (\theta \nabla u) + \bm{a} \cdot \nabla u + \rho u - f &=& 0, \mbox{ in } \Omega, \label{eqn:adr1}\\
u &=& g, \mbox{ on } \Gamma_g, \label{eqn:adr2}\\
-\theta \nabla u \cdot \bm{n} &=& \eta, \mbox{ on } \Gamma_\eta. \label{eqn:adr3}
\end{eqnarray}
Here $\theta, \bm{a}$, and $\rho$ are the diffusion, advection and reaction coefficients and are allowed to be functions of $u$ and the spatial coordinates. Further $f$ is the forcing function, 
\clrr{$\eta$} is the prescribed Neumann data and $g$ is the prescribed Dirichlet data. Note that $\overline{\Gamma_\eta \bigcup \Gamma_g} = \partial \Omega$.

The weak formulation of this PDE is given by: find $u \in \mathcal{V} \equiv H^1 (\Omega)$, such that $\forall w \in \mathcal{V}$,
\begin{eqnarray}
a(w,u) &=& (w,f) + (w,\eta)_{\Gamma_\eta} + ( -  \theta \nabla w \cdot \bm{n}+ \beta w,g)_{\Gamma_g}. \label{eq:weak_nl}
\end{eqnarray}
Here $a(\cdot,\cdot)$ is linear in its first argument and nonlinear in the second, and is defined as 
\begin{eqnarray}
    a(w,u) \equiv (\nabla w, \theta \nabla u) + (w, \bm{a} \cdot \nabla{u} + \rho u) - (w, \theta \nabla u \cdot \bm{n})_{\Gamma_g} + (- \theta \nabla w \cdot \bm{n}+ \beta w,u)_{\Gamma_g}. \label{eq:defa}
\end{eqnarray}
Further,  $\beta$ is a numerical parameter. It is easily shown that the formulation above is consistent, and when the weighting function and trial solution spaces are approximated by their finite dimensional counterpart, it leads to a finite element solution that converges to the exact solution \cite{bazilevs2007weak}. This formulation is often referred to as the weak imposition of Dirichlet boundary conditions.

Next, as in Section \ref{sec:fem}, we approximate the space $\mathcal{V}$ with its finite dimensional counterpart $\mathcal{V}^h$, which is spanned by a set of continuous basis functions $\{\phi_i(\x)\}_{i=1}^q$. Further expanding the weighting functions, the approximate solution and the approximate PDE data 
($f,\eta,g$) using this basis, and substituting this in (\ref{eq:weak_nl}), we arrive at,
\begin{equation}
    \bm{r} (\bm{U}) = \bm{M} \bm{F} +  \widetilde{\bm{M}} \bm{N} + \breve{\bm{M}} \bm{G}. \label{eq:defr_nl}
\end{equation}
Here $\bm{U}, \bm{F}, \bm{N}$ and $\bm{G} \in \mathbb{R}^q$ are the coefficients in the expansion of $u^h, f^h, \eta^h$ and $g^h$, respectively. Further, $\bm{r}: \mathbb{R}^q \to \mathbb{R}^q$ is given by,
\begin{eqnarray}
    r_i (\bm{U}) \equiv a(\phi_, U_j \phi_j), \qquad 1 \le i,j \le q.
\end{eqnarray}
The matrices $\bm{M}$ and $\widetilde{\bm{M}}$ are as defined in \eqref{eqn:mat}, and the matrix $\breve{\bm{M}}$ is given by
\begin{eqnarray}
    \breve{M}_{ij}  \equiv ( -  \theta \nabla \phi_i \cdot \bm{n}+ \beta \phi_i ,\phi_j)_{\Gamma_g}, \qquad 1 \le i,j \le q.
\end{eqnarray}
Note that due to dependence of $\theta$ on $u$, $\breve{\bm{M}}$ is a function of $\bm{U}$; however, for simplicity we ignore this dependence.

Assuming that (\eqref{eq:defr_nl}) is solvable for $\bm{U}$, we may write 
\begin{equation}\label{eqn:nonlinear_fem}
    u^h( \x) = (\bm{R}^{-1}( \bm{M} \bm{F} +  \widetilde{\bm{M}} \bm{N} + \breve{\bm{M}} \bm{G} ))^T \bPhi, 
\end{equation}
where 
$\bm{R}^{-1}: \Ro^q \rightarrow \Ro^q$ solves for $\U$ in \eqref{eq:defr_nl}  
This suggests the form of the \net shown in Figure \ref{fig:NLvarmion}(a). In this we observe that $(f,\eta,g)$ are evaluated at sensor nodes to construct the input vectors $\Fh \in \Ro^k$, $\Nh \in \Ro^{k^\prime}$ and $\Gh \in \Ro^{k^{\prime\prime}}$. These input vectors are transformed by linear operators (matrices) $\A, \At, \breve{\A}$ respectively to given vectors whose dimension is the same as the latent dimension $p$. Mimicking \eqref{eqn:nonlinear_fem}, these three vectors are added and then passed though a nonlinear network $\mathcal{N}$ and transformed into a the vector $\bbeta(\Fh,\Nh,\Gh) \in \Ro^p$, where
\begin{equation}\label{eqn:varmionNL_branch}
\bbeta(\Fh,\Nh,\Gh) = \mathcal{N}(\A\Fh + \At \Nh + \breve{\A} \Gh).
\end{equation}
The network $\mathcal{N}$ is the analog of the discrete operator $\bm{R}^{-1}$. 
A dot product with the trunk vector $\btau(\x) \in \Ro^p$ yields the final solution predicted at $\x$. The corresponding VarMiON operator is given by
\begin{equation}\label{eqn:NLnet}
    \widehat{\Sop}^{NL} : \Ro^k \times \Ro^{k^\prime} \times   \Ro^{k^{\prime\prime}} \rightarrow \dbV^\tau, \qquad \widehat{\Sop}^{NL}(\Fh,\Nh,\Gh) = \widehat{u}(.;\Fh,\Nh,\Gh) = \bbeta(\Fh,\Nh,\Gh)^\top \btau.
\end{equation}

We also observe that the exact solution operator is homogeneous. That is, when all the input functions are zero, the corresponding solution is zero. At the discrete level, we can interpret this as $\bm{R}^{-1}(\bm{0}) = \bm{0}$. The \net operator can also be endowed by this property, for instance by using the following modified branch instead of \eqref{eqn:varmionNL_branch} 
\begin{equation}\label{eqn:varmioncNL_branch}
\bbeta_c(\Fh,\Nh,\Gh) = \widehat{\Z} \odot \mathcal{N} (\widehat{\Z}), \quad \widehat{\Z} = \A\Fh + \At \Nh + \breve{\A} \Gh
\end{equation}
where $[\bm{a} \odot \bm{b}]_i = a_i b_i$ represents the Hadamard product of the vectors $\bm{a},\bm{b}$ with equal dimensions. We call the corresponding constrained operator network as VarMiON-c, which is depicted in Figure \ref{fig:NLvarmion}(b). In the numerical section \ref{sec:eikonal} we consider both architectures, one where this constraint is imposed and another where it is not imposed. Not surprisingly, we find that the architecture where this constraint is imposed performs better.

\begin{figure}[htbp]
\begin{center}
\subfigure[VarMiON]{
\includegraphics[width=0.48\textwidth]{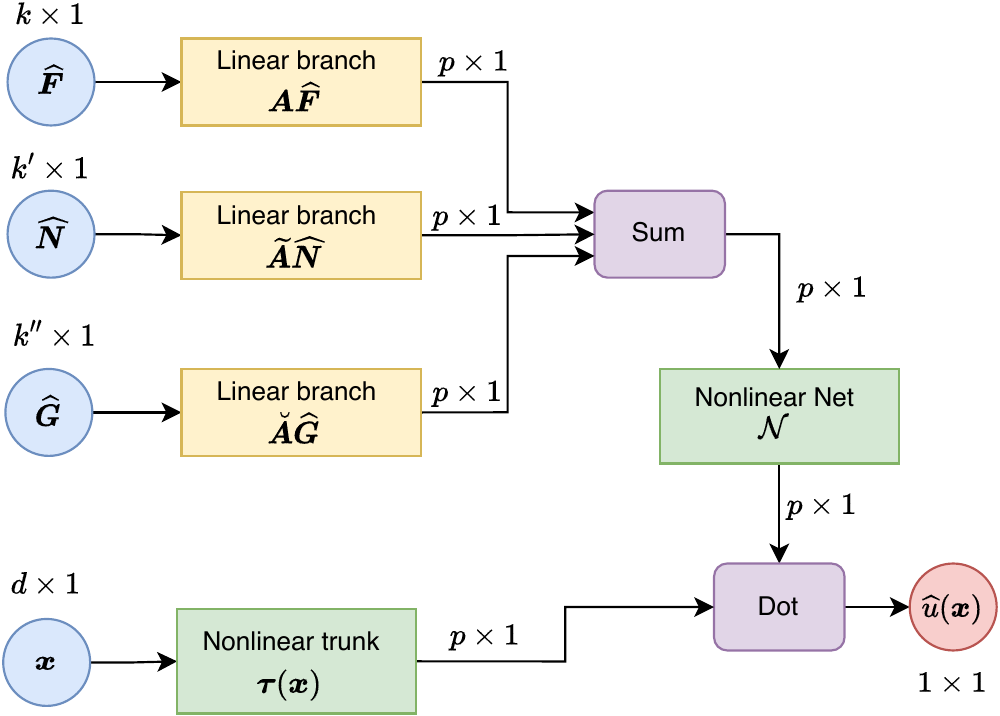}}\hfill
\subfigure[VarMiON-c]{
\includegraphics[width=0.48\textwidth]{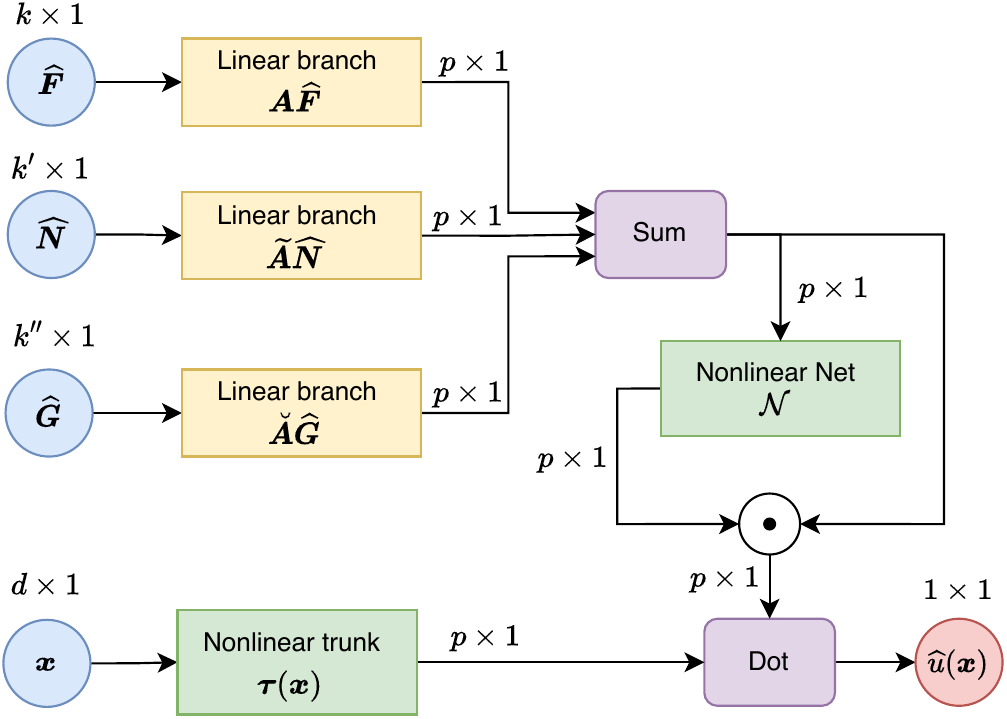}}
\caption{\net architectures for nonlinear advection-diffusion-reaction equation.}
\label{fig:NLvarmion}
\end{center}
\end{figure}

\section{Numerical results}\label{sec:num}

In this section we compare the performance of the \net with a vanilla DeepONet which consists of a branch and a trunk network. We note that the vanilla DeepONet only works with a single input function and so for the benchmark comparison we simply concatenate all the values of the input functions ($\theta$, $f$, $\eta$, etc) through a common fully-connected branch network. In contrast, as shown in Figure \ref{fig:varmion}, the architecture of the \net is more nuanced. The precise architectures of the networks are shown in the Appendix. When comparing the performance of these networks we ensure that the number of network parameters in each is approximately the same (see Table \ref{tab:result_summary} and  \ref{tab:result_summary_NL} for precise values). We also ensure that the architecture of the network used for constructing the basis functions (often referred to as the trunk network) is identical. 

We first consider the steady-state heat conduction problem as the prototypical linear elliptic PDE model. We demonstrate that the conclusions drawn from the numerical examples are broadly applicable by considering different sampling strategies for the input functions (spatially random and uniform sampling), different methods for generating basis functions (ReLU networks and radial basis functions (RBFs)), and operators with varying number of input functions (two in Section \ref{sec:two_input_linear} and three in Section \ref{sec:three_input_linear}). In each case, we observe that the \net yields more accurate results than the vanilla DeepONet.

For the experiments using an RBF trunk, in Section \ref{sec:mionet_data_coverage} we also compare the results with those obtained using a MIONet architecture which was proposed in \cite{Jin2022} to suitably handle multiple input functions. In a MIONet architecture, a separate branch is constructed for each input function. Hadamard products are performed between the outputs of each branch sub-network, which results in a final branch output vector of size equal to the latent dimension. The MIONet output is obtained by taking a dot product of this vector with the trunk output. Specific architectures of the MIONet are detailed in the Appendix. In numerous science and engineering applications, the amount of available training data is constrained. Given this, we also evaluate the performance of the networks as a function of the size of the training data set in Section \ref{sec:mionet_data_coverage}.

We end by designing a variationally mimetic architecture for the nonlinear regularized Eikonal equation with a varying source function in Section \ref{sec:eikonal}. Here, we compare the performance of the vanilla DeepONet with both the VarMiON and the constrained VarMiON-c architectures as discussed in Section \ref{sec:nonlinear}.

\subsection{Steady-state heat equation}
We consider the operator defined by the solution to the steady-state heat conduction problem given by,
\begin{equation}\label{eqn:heat_con}
\begin{aligned}
- \nabla \cdot ( \theta(\x) \nabla u (\x)) &=  f(\x), \quad &&\forall \ \x \in  \Omega, \\
0\theta (\x) \nabla u(\x) \cdot \bm{n} (\x) &= \eta (\x), \quad &&\forall \ \x \in  \Gamma_\eta,\\
u(\x) &= 0, \quad &&\forall \ \x  \in  \Gamma_g,
\end{aligned}
\end{equation}
where $u$ is the temperature field, $\theta$ is the thermal conductivity, $f$ represents volumetric heat sources and $\eta$ is the heat flux through a part of the boundary. The domain $\Omega$ is a unit square, the boundary $\Gamma_{\eta}$ represents the left and right edges, and the boundary $\Gamma_g$ represents the top and bottom edges. Our goal is construct approximations to the operator that maps $\theta, f$ and $\eta$ to the solution $u$.

For both training and test data, the functions $\theta$ and $f$ are selected to be Gaussian random fields defined on $\Omega$ with length scales of 0.4 and 0.2, respectively. Both $f$ and $\theta$ are scaled so that they assume values in the interval $(0.02, 0.99)$. The function $\eta$ is a Gaussian random field defined on $\Gamma_\eta$ with length scale of 0.3 and is scaled to be in the interval $(-1.,1.)$.

\subsubsection{Operators with two input functions}\label{sec:two_input_linear}


In this section we consider a \net and a vanilla DeepONet that map $\theta$ and $f$ to the temperature field, $u$. The training data is generated by creating 10,000 realizations of pairs of $\theta$ and $f$.
This input is used in FEniCS to solve the steady-state heat conduction problem with a mesh of $32 \times 32$ linear finite elements to yield 10,000 realizations of the triad $(\theta, f, u)$. Out of these, 9,000 are used for training and validation and 1,000 are used for testing. We select 224 output sensor nodes out of which 100 are distributed randomly inside the domain and 124 are distributed along all boundary nodes.

\begin{table}[ht]
\renewcommand{\arraystretch}{1.}
\centering
\caption{Summary of the \net and vanilla DeepONet performance.}
\begin{adjustbox}{width=0.85\linewidth}
\begin{tabular}{c c c c}
\toprule
Case & Model & Number of parameters & Relative $L_2$ error \\
\toprule
Two input functions with & DeepONet (w/ ReLU trunk) & 111,248 & 1.07 $\pm$ 0.39 \% \\
randomly sampled input & \net (w/ ReLU trunk) & {109,013} & {0.93 $\pm$ 0.28} \% \\
\midrule
Two input functions with & DeepONet (w/ ReLU trunk) & 49,928 & 1.98 $\pm$ 0.79 \% \\
uniformly sampled input & \net (w/ ReLU trunk) & {46,281} & {1.05 $\pm$ 0.42} \% \\
\midrule
Two input functions with  & DeepONet (w/ RBF trunk) & 17,911 & 1.39 $\pm$ 0.60  \% \\
uniformly sampled input & \net (w/ RBF trunk) &{17,345} & {0.64 $\pm$ 0.35} \% \\
\midrule
Three input functions with & DeepONet (w/ RBF trunk) &{24,543} & {5.99 $\pm$ 4.24} \% \\
uniformly sampled input & \net (w/ RBF trunk) & {23,065} &  {2.06 $\pm$ 0.90} \% \\ 
\bottomrule
\end{tabular}\label{tab:result_summary}
\end{adjustbox}
\end{table}

We test the performance of the operator networks under different scenarios. These include (a) sampling the input functions at randomly selected points (100 in total) or on a uniform $10 \times 10$ grid, and (b) constructing the basis functions using a ReLU network or a linear combination of radial basis functions whose centers and widths are determined during the training process. This yields a total of three cases for which we train the \net and the vanilla DeepONet. In each case, we train the network using the Adam optimizer \cite{kingma2017adam} until the training error saturates, and we then select the model corresponding to the lowest validation loss as the final trained network. Thereafter, we test the performance of this network on 1,000 test samples  and evaluate the normalized $L_2$ error between the network and the finite element solutions. The average and standard deviation of this error computed over all the test samples is presented in Table \ref{tab:result_summary}. From this table, we conclude that in each case the \net incurs lower average error than the vanilla DeepONet. Further, in each case the standard deviation of the error is lower with the \net as compared to the vanilla DeepONet indicating its robustness. Further, the drop in error when using uniform sampling is significant (a factor of $\approx 1/2$).
%
%
%


\begin{figure}[htbp]
\begin{center}
\includegraphics[width=0.99 \linewidth]{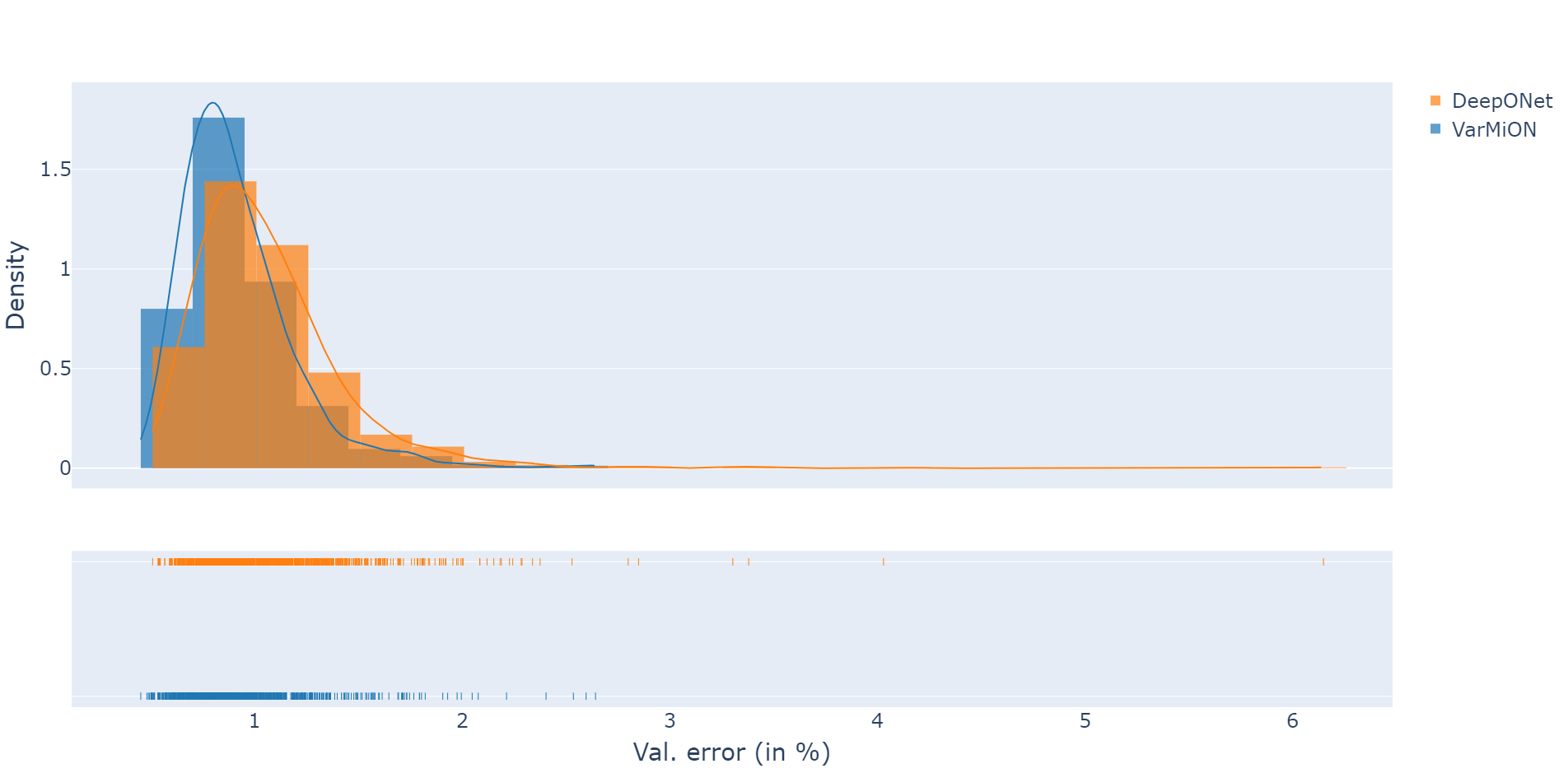}
\end{center}
\caption{\reviii{}{Probability density (top) and rug plot (bottom) of the test error for the case with two input functions sampled on a random grid with ReLU trunk for the \net (blue) and the vanilla DeepONet (orange). In the rug plot each tick indicates one sample.}}
\label{fig:2inputs_random_hist}
\end{figure}

\begin{figure}[htbp]
\begin{center}
\includegraphics[width=0.99 \linewidth]{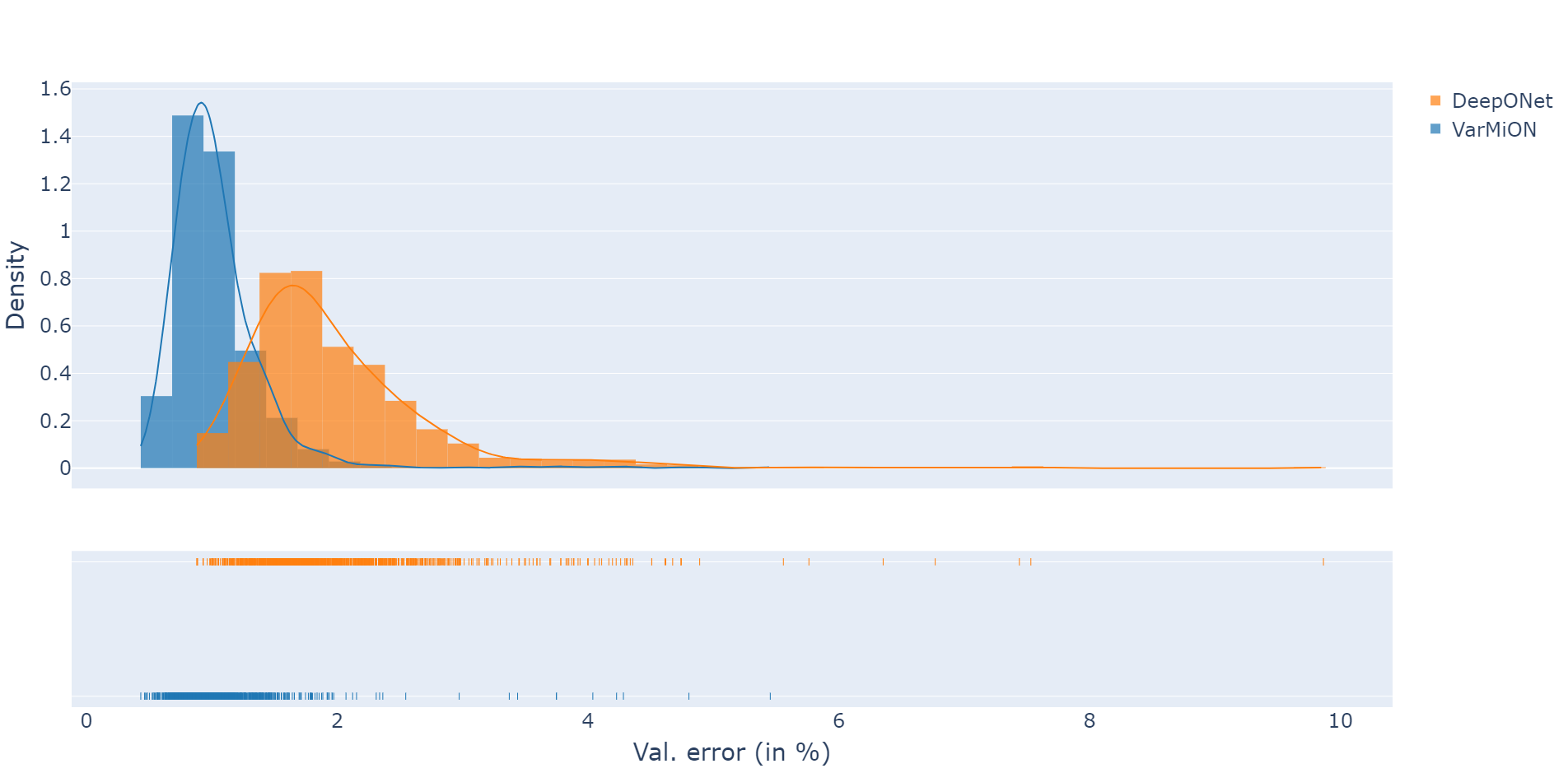}
\end{center}
\caption{\reviii{}{Probability density (top) and rug plot (bottom) of the test error for the case with two input functions sampled on a uniform grid with ReLU trunk for the \net (blue) and the vanilla DeepONet (orange). In the rug plot each tick indicates one sample.}}
\label{fig:2inputs_uniform_hist}
\end{figure}

\begin{figure}[htbp]
\begin{center}
\includegraphics[width=0.99 \linewidth]{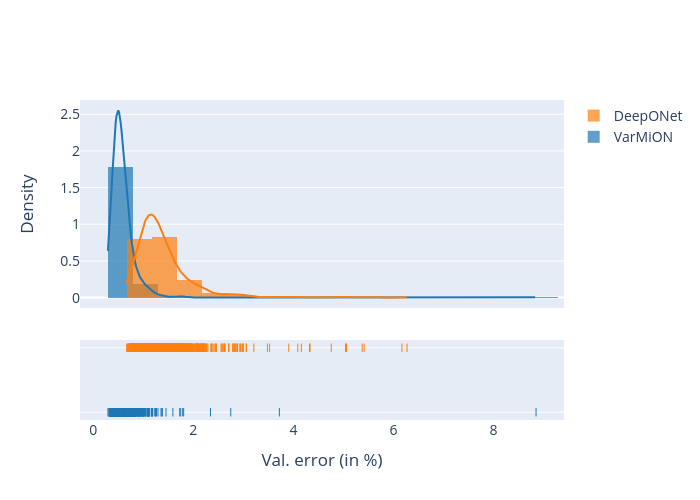}
\end{center}
\caption{\reviii{}{Probability density (top) and rug plot (bottom) of the test error for the case with two input functions sampled on a uniform grid with RBF trunk for the \net (blue) and the vanilla DeepONet (orange). In the rug plot each tick indicates one sample.}}
\label{fig:2inputs_uniform_rbf_hist}
\end{figure}

A more refined analysis of the error involves examining the discrete probability density of scaled $L_2$ error for the two network types (\net and the vanilla DeepONet) as shown in Figures \ref{fig:2inputs_random_hist},   \ref{fig:2inputs_uniform_hist}, and \ref{fig:2inputs_uniform_rbf_hist}. In these figures, the discrete probability density for the error is shown on the top and the distribution of individual errors is shown below in the rug plot, where each tick represents one sample. We observe that the error probability density for the \net is much tighter, with very few cases where the error is significantly larger than the mean error. This is not the case for the vanilla DeepONet, where for some test cases the network results have very large errors (around 10\%). Thus we conclude that the performance of the \net is more robust. 

In Figure \ref{fig:2inputs_random_predictions}, we present the true solution, the \net solution, and the vanilla DeepONet solution for five different instances of input functions. We consider operator networks with a ReLU trunk and a spatially random grid for sampling the input functions. These instances were selected from the 1,000 test samples to highlight the heterogeneity in the spatial variation of the solution. This heterogeneity can be observed by considering the true solution (3rd column in the figure) which displays significant differences among the five instances. From this figure we observe that for each instance, both the \net solution (4th column) and the DeepONet solution (6th column) capture the overall behavior of the solution. However, by considering corresponding errors (5th and 7th columns) we observe that the error in the DeepONet solution is significantly higher. Figures \ref{fig:2inputs_uniform_predictions} and \ref{fig:2inputs_uniform_predictions_rbf}, are the corresponding figures for operator networks with uniform spatial sampling of input functions, and a ReLU or an RBF trunk, respectively. From these figures also we observe that \net solution incurs smaller error.

\begin{figure}[htbp]
\begin{center}
\includegraphics[width=\textwidth]{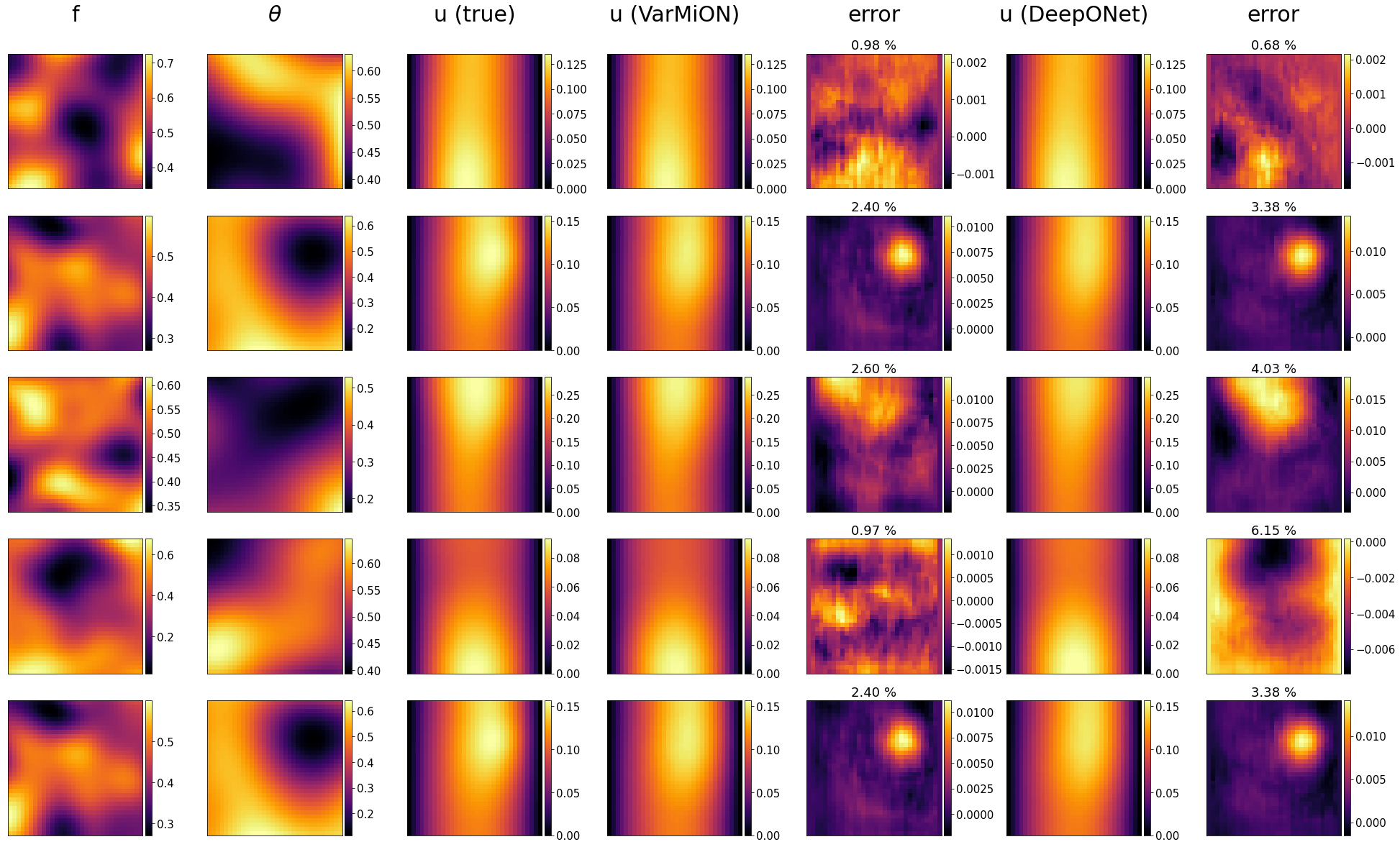}
\caption{\reviii{}{\textit{Temperature prediction for five representative samples from the test set for the case of two input functions sampled on a random grid (with ReLU trunk)}: (First column) source field, (second column) conductivity field (third column) true temperature field (fourth and sixth column) temperature field prediction by \net and DeepONet respectively, (fifth and seventh column) corresponding error (normalized $L_2$ error is shown at the top).}}
\label{fig:2inputs_random_predictions}
\end{center}
\end{figure}

\begin{figure}[htbp]
\begin{center}
\includegraphics[width=\textwidth]{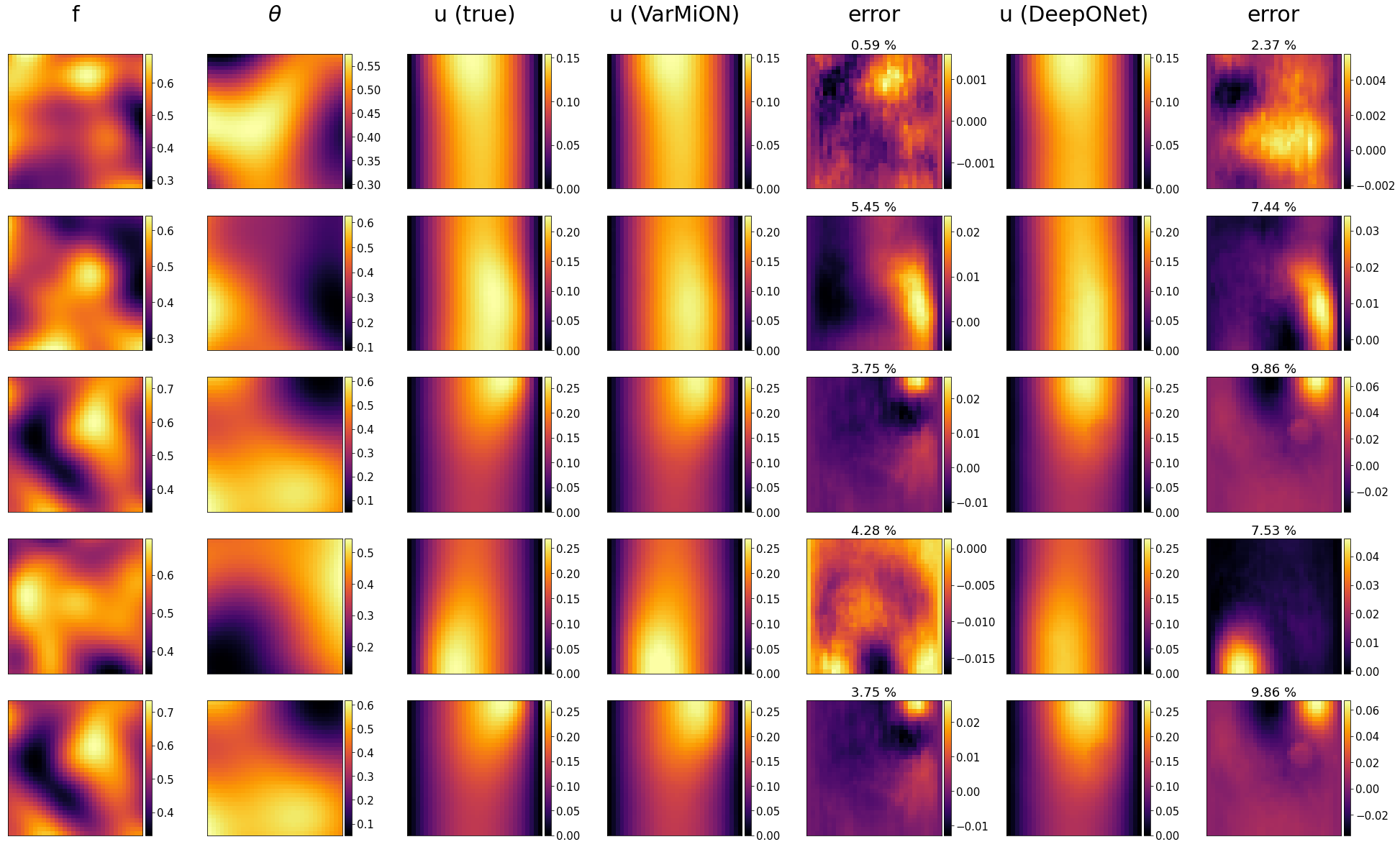}
\caption{\reviii{}{\textit{Temperature prediction for five representative samples from the test set for the case of two input functions sampled on a uniform grid (with ReLU trunk)}: (First column) source field, (second column) conductivity field (third column) true temperature field (fourth and sixth column) temperature field prediction by \net and DeepONet respectively, (fifth and seventh column) corresponding error (normalized $L_2$ error is shown at the top).}}
\label{fig:2inputs_uniform_predictions}
\end{center}
\end{figure}

\begin{figure}[htbp]
\begin{center}
\includegraphics[width=\textwidth]{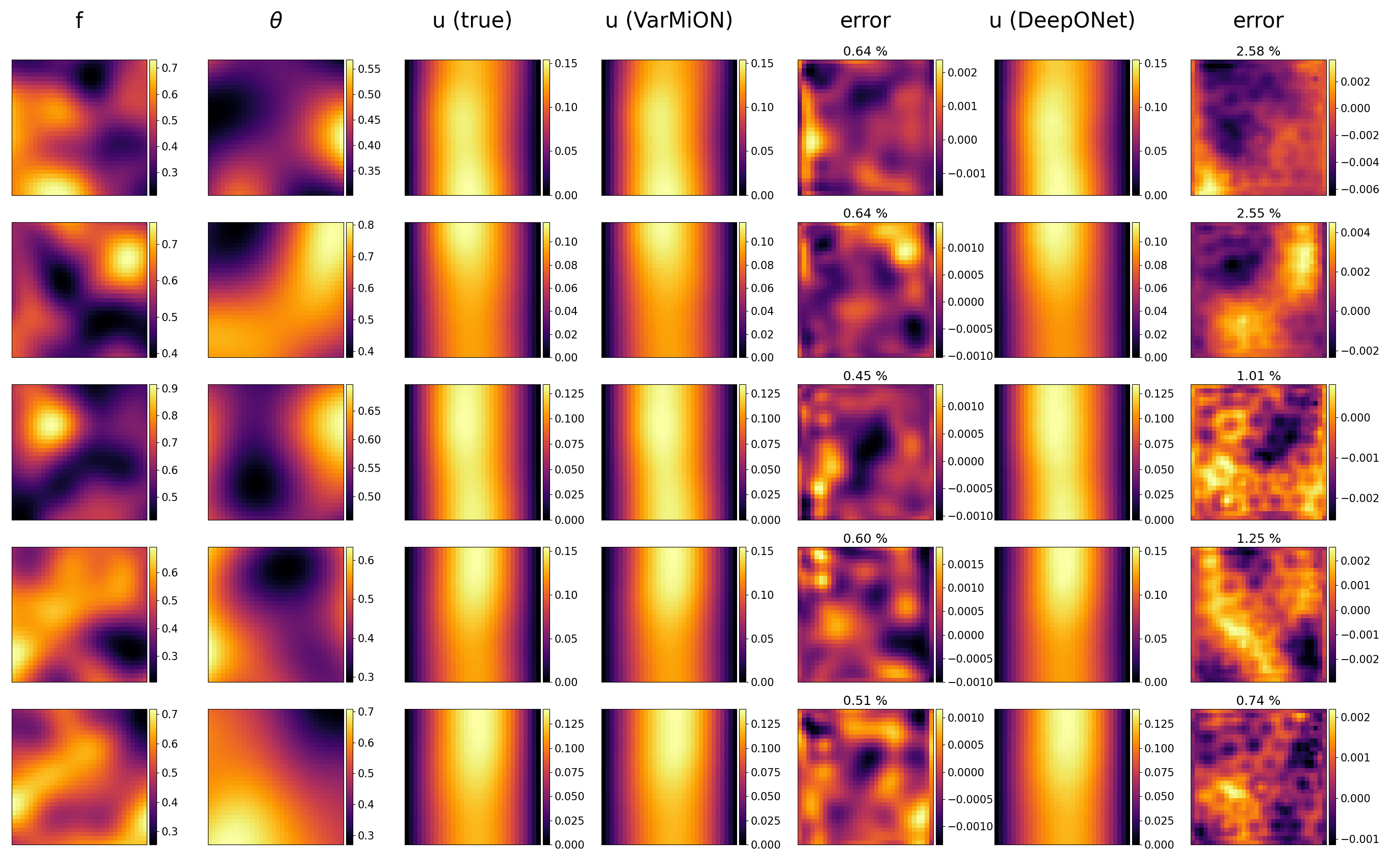}
\caption{\reviii{}{\textit{Temperature prediction for five representative samples from the test set for the case of two input functions sampled on a uniform grid (with RBF trunk)}: (First column) source field, (second column) conductivity field (third column) true temperature field (fourth and sixth column) temperature field prediction by \net and DeepONet respectively, (fifth and seventh column) corresponding error (normalized $L_2$ error is shown at the top).}}
\label{fig:2inputs_uniform_predictions_rbf}
\end{center}
\end{figure}

\subsubsection{Operators with three input functions}\label{sec:three_input_linear}

Next we consider \net and vanilla DeepONet operators that map $\theta$, $f$ and $h$ to the temperature field, $u$. The specification of the input fields remains unchanged from the previous section.

To the best of our knowledge this is the first instance of training and testing operator networks for PDEs with more than two input fields. The training data is generated by creating 10,000 realizations each of the triad $(\theta,f,\eta)$.
This  input is used in FEniCS to solve the steady-state heat conduction problem to yield 10,000 realizations of $(\theta, f,\eta, u)$. Of these, 9000 are used for training and validation, and 1000 are used for testing. For both the \net and the vanilla DeepONet uniform sampling points are used for the input functions, and radial basis functions are used in the trunk. 


In Table \ref{tab:result_summary}, we have reported the average scaled $L_2$ error for the \net and the vanilla DeepONet formulation. Once again, we observe that \net performs much better. Its error is approximately three times smaller than the DeepONet error. This is also observed in Figure \ref{fig:3inputs_uniform_rbf_hist}, we have plotted error histograms for the two networks. Once again we observe that the error distribution for the \net solution is much tighter and centered closer to the origin, thereby indicating that this operator is more robust and generalizes better to test data. The maximum error for the vanilla DeepONet is as high as $30\%$.

\begin{figure}[htbp]
\begin{center}
\includegraphics[width=0.99 \linewidth]{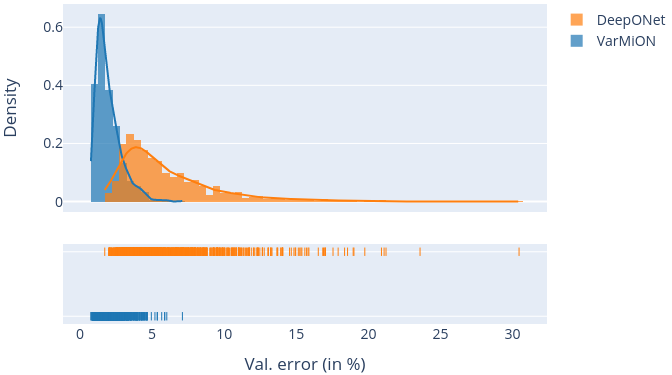}
\end{center}
\caption{\reviii{}{Probability density (top) and rug plot (bottom) of the test error for the case with three input functions sampled on a uniform grid with RBF trunk for the \net (blue) and the vanilla DeepONet (orange). In the rug plot each tick indicates one sample.}}
\label{fig:3inputs_uniform_rbf_hist}
\end{figure}

In Figure \ref{fig:3inputs_uniform_predictions_rbf}, we have plotted five randomly selected instances from the test set. In each case, we plot the forcing function (column 1), the thermal conductivity (column 2), the Neumann boundary condition (column 3), the true solution (column 4), the \net solution and its error (columns 5 \& 6), and the vanilla DeepONet solution and its error (columns 7 \& 8). In each instance, the \net solution is more accurate. 

\begin{figure}[htbp]
\begin{center}
\includegraphics[width=\textwidth]{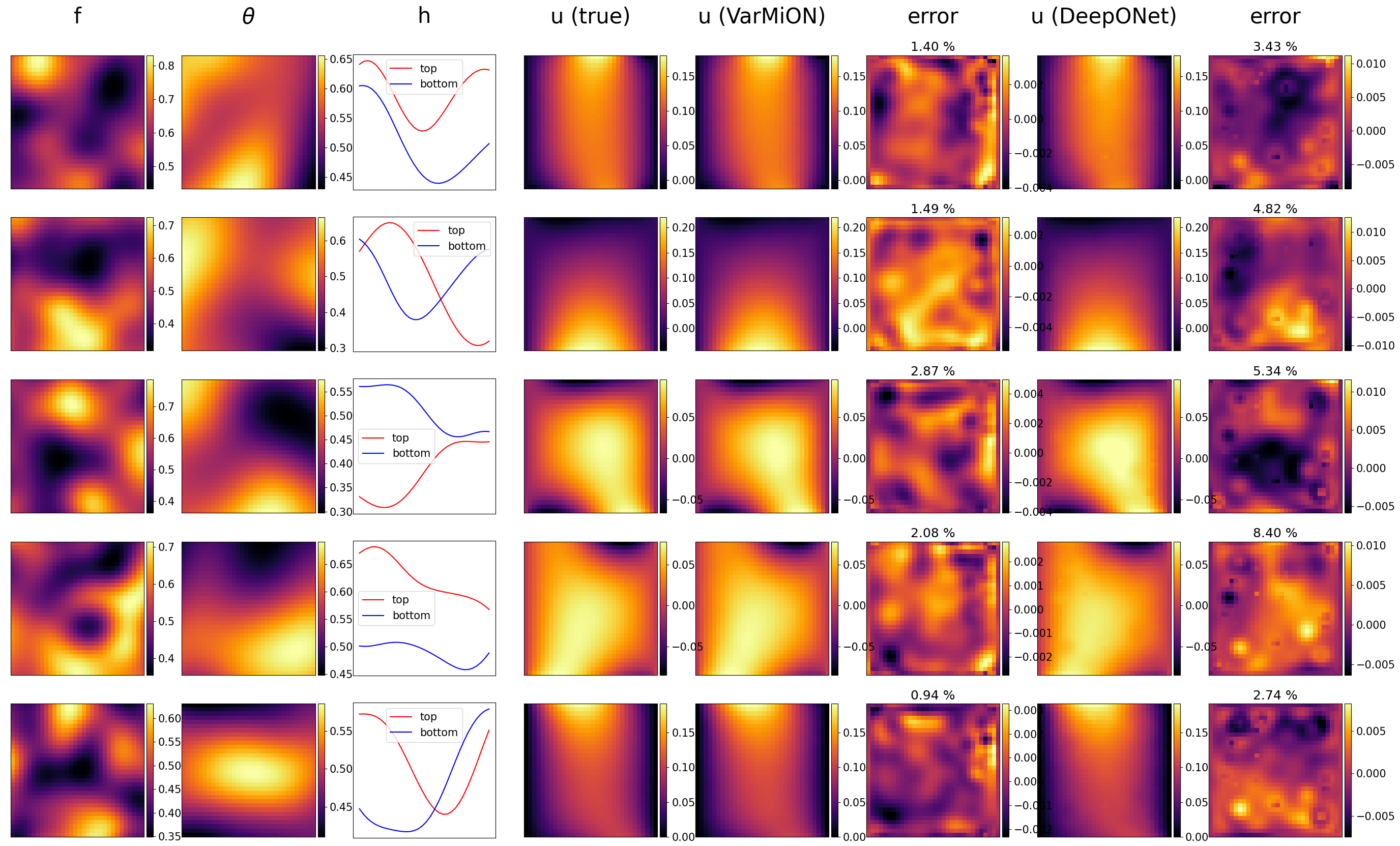}
\caption{\reviii{}{\textit{Temperature prediction for five representative samples from the test set for the case of three input functions sampled on a uniform grid (with RBF trunk)}: (First column) source field, (second column) conductivity field,  (third column) top and bottom Neumann data, (fourth column) true temperature field, (fifth and seventh column) temperature field prediction by \net and DeepONet respectively, (sixth and eighth column) corresponding error (normalized $L_2$ error is shown at the top).}}
\label{fig:3inputs_uniform_predictions_rbf}
\end{center}
\end{figure}


\revi{}{\subsubsection{Comparison with MIONet and the effect of the dataset coverage}\label{sec:mionet_data_coverage}}
\revii{}{One of the attractive features of VarMiON is its ability to work with multiple input functions. The MIONet \cite{Jin2022} is also designed to work with multiple input functions. Given this, in this section, we compare the VarMiON with the MIONet.  For this we consider uniformly sampled input data and RBF in the trunk network since these choices provided the most accurate results in the previous section. The precise details of MIONet architecture used can be found in the Appendix. For the two input case with uniformly sampled inputs and RBF trunk functions, the MIONet (with 17,216 parameters) has a relative $l_2$ error of 1.66 $\pm$ 0.78 \% which is more than double that of VarMiON error as reported in Table \ref{tab:result_summary}.}

\revi{}{We investigate the three inputs case while varying the training dataset size. Specifically, we train VarMiON, MIONet, and DeepONet with four different dataset sizes \{1000, 2000, 4000, 6000\} and test their performance on a separate test set with 1000 samples. The architectures of the VarMiON and DeepONet are kept the same as those considered in the previous section, while the corresponding MIONet architecture has 22,752 trainable parameters. The purpose of this study is to investigate how these different operator learning models perform when we have limited training data.}
\revi{}{We consider two variations of the training dataset. The first dataset constitutes an array of solutions to (\ref{eqn:heat_con}) that are generated from a randomized realization of input functions $(\theta, f,\eta)$. The second set constitutes an array of solutions to (\ref{eqn:heat_con}) that are generated from an ordered realization of the input functions $(\theta, f,\eta)$. Such an ordering follows from a nested loop where we iterate over different realizations of $f$ (for some fixed $\theta$ and $\eta$) in the inner--most loop, the middle loop iterates over $\theta$, and finally the outer--most loop iterates over $\eta$. For both data sets, we observe how the generalization error varies when we only use the first $n\in\{1000,2000,4000,6000\}$ components for training. By ordering the dataset with respect to the input parameters in such a nested--loop manner and restricting of the training set to the first $n$ solutions, we are effectively reducing the variety of the input parameters used for generating the training set. Therefore, such a test investigates the sensitivity of the generalization error to the reduction in coverage by the training data of the PDE data space.}

\revi{}{Table \ref{tab:sensitivity_summary} shows the relative $L_2$ errors in the output of \net, MIONet and vanilla DeepONet when a randomized and ordered dataset of size $n\in\{1000,2000,4000,6000\}$ are used for training. It can be seen that \net output exhibits consistently less error than both MIONet and vanilla DeepONet. The gap between \net's and both of MIONet's and DeepONet's performance is substantially more in the case when an ordered dataset is used. Such results suggest that VarMiON's conformity to the variational formulation is beneficial and leads to superior efficiency, with respect to the dataset coverage, over MIONet and vanilla DeepONet.}

\revii{}{
We note that the MIONet performed very poorly in the 3 input case. We believe this might be a result of the inconsistency introduced 
by the construction of the MIONet proposed in \cite{Jin2022} when applied to this particular problem. Let us denote the the individual MIONet branches for the input function $f$, $\theta$ and $\eta$ by $\bbeta^1(\Fh)$, $\bbeta^2(\Th)$ and $\bbeta^3(\Nh)$, respectively. Now the MIONet formulation requires taking a Hadamard product of the output of the three branches, followed by a dot product with the trunk to obtain the final prediction at the node $\x$
\begin{equation}\label{eqn:MIONet}
\hat{u}(\x) = \left( \bbeta^1(\Fh) \odot \bbeta^2(\Th) \odot \bbeta^3(\Nh)\right)^\top \btau(\x).
\end{equation}
Further, as has been suggested by the authors in \cite{Jin2022}, the linearity in $f$ and $\eta$ dictate that the branches $\bbeta^1$ and $\bbeta^3$ are linear operators, which is similar to what is proposed by us for the VarMiON. Now, for the current PDE model \eqref{eqn:pde}, if $f\equiv 0$ but $\eta \neq 0$ and $\theta > 0$, the true solution need not be zero. The same holds if $\eta\equiv 0$ but $f \neq 0$ and $\theta > 0$. However, based on the MIONet formulation \eqref{eqn:MIONet}, the predicted solution will be identically zero if either $f$ or $\eta$ are zero, which leads to the inconsistency. Thus, we feel the original MIONet formulation will not lead to a good network for this 3 input problem, and may need to be suitably modified. Note that the VarMiON does not suffer from such an inconsistency.
}

\begin{table}[ht]
\renewcommand{\arraystretch}{1.}
\centering
\caption{\revi{}{Summary of the three-input \net, MIONet, and vanilla DeepONet performance when the training dataset size is reduced. All operator networks use an RBF trunk network.}}
\begin{tabular}{c c c c}
\toprule
Training dataset size   & Model & Relative $L_2$ error & Relative $L_2$ error \\
       $n$              &       &(randomized dataset)  & (ordered dataset) \\
\toprule
\multirow{2}{*}{$n=1000$}   & DeepONet  & $10.23 \pm 5.20$ & $38.65 \pm 12.71$ \\
                            & MIONet      & $ 88.07 \pm 69.68$ &  $ 58.96 \pm 39.07$  \\
                            & \net      & \bm{$4.27 \pm 2.23$} &  \bm{$17.82 \pm 8.84$}  \\
\midrule                                                   
\multirow{2}{*}{$n=2000$}   & DeepONet  & $9.00 \pm 5.63$ & $40.96 \pm 35.08$ \\
                            & MIONet      & $ 85.03 \pm 123.77$ &  $ 50.81 \pm 4.69 $  \\
                            & \net      & \bm{$4.04 \pm 1.86$} & \bm{$8.32 \pm 3.61$} \\
\midrule                                                   
\multirow{2}{*}{$n=4000$}   & DeepONet  & $7.19 \pm 3.75$ & $61.34 \pm 26.82$\\
                            & MIONet      & $ 88.39 \pm 62.10$ &  $ 76.98 \pm 10.01$  \\
                            & \net      & \bm{$2.90 \pm 1.50$} & \bm{$8.42 \pm 2.76$} \\
\midrule
\multirow{2}{*}{$n=6000$}   & DeepONet  &  $6.28 \pm 3.55$ &  $35.62 \pm 14.28$ \\
                            & MIONet      & $ 82.89 \pm 34.19$ &  $ 68.39 \pm 8.71$  \\
                            & \net      &   \bm{$2.74 \pm 1.31$} &  \bm{$6.29\pm2.32$} \\
\bottomrule
\end{tabular}\label{tab:sensitivity_summary}
\end{table}

\revii{}{\subsection{Regularized eikonal equation}\label{sec:eikonal}}
\revii{}{
We consider the regularized eikonal equation which is useful in modeling wave propagation. It is given by
\begin{equation}\label{eqn:eikonal}
\begin{aligned}
- 0.01 \Delta u (\x) + |\nabla u(x)| &=  f(\x), \quad &&\forall \ \x \in  \Omega, \\
u(\x) &= 0, \quad &&\forall \ \x  \in  \partial \Omega,
\end{aligned}
\end{equation}
where $u(\x)$ is interpreted as the minimal time required to travel from $\x$ to the domain boundary $\partial \Omega$, while $1/f$ represents the speed of travel through the medium. Note that the eikonal equation \eqref{eqn:eikonal} can be reformulated as the advection-diffusion-reaction equation \eqref{eqn:adr1}-\eqref{eqn:adr3} by choosing $\theta \equiv 0.01$, $\rho = 0.0$, $g = 0.0$, $\bm{a} = \nabla u(x) /|\nabla u(x)|$ and $\Gamma_g = \partial \Omega$. Thus, the corresponding discrete weak solution in terms of some basis $\{\phi_i(\x)\}_{i=1}^q$ is given by
\begin{equation}
    u^h(\x) = (\bm{R}^{-1}(\M \F))^\top \bPhi.
\end{equation}
We are interested in approximation the solution operator that maps $f$ to $u$.}

\revii{}{The domain $\Omega$ is chosen as a unit square, while the source function $f$ is given by a Gaussian random field with length scale 0.4, whose values are scaled to lie in $(0.1,2.0)$. The training data is constructed by generating 10,000 realizations of $f$ and obtaining the corresponding solutions using FEniCS on a $32 \times 32$ mesh, which yield 10,000 pairs of $(f,u)$. We use 8,000 samples for training and validation and retain 2,000 for testing.}

\revii{}{To train the operator networks, $f$ is sampled on the same $32 \times 32$ mesh and fed as input to the branch. In other words, there are 1024 sensor nodes. For a given $f$, the target solution is evaluated at 140 output nodes randomly chosen from the $32 \times 32$ mesh, where the predicted solution values are matched in the training loss function. Motivated by the performance of the operators with the linear PDE model in the previous section, and the solution features of the eikonal problem, we continue to use an RBF trunk. All networks use a latent dimension $p=100$ but have different branch architectures. We recall the key properties of the weak form discussed in Section \ref{sec:nonlinear} which we are interested in mimicking:
\begin{itemize}
    \item The sampled input function needs to be compressed by a linear transform to a vector of dimension that equals the latent dimension $p$. Since the present problem has a single input function, i.e., the \textit{sum} block shown in Figure \ref{fig:NLvarmion} is redundant, we can consider the branch sub-network to be a single feed-forward network whose first hidden layer does not have a bias term and has a width $p=100$, irrespective of the input dimension (1024 in this case). One can argue that this leads to an architecture that is an instance of a DeepONet. However, in standard practice, one would construct the branch of a DeepONet to i) include a bias in every hidden layer, and ii) gradually decay the width of the hidden layers with increasing depth. This is what distinguishes a vanilla DeepONet from the proposed VarMiON architecture.
    \item We want to preserve the homogeneous nature of the solution operator. This is achieved be considering the constrained VarMiON-c architecture (see Figure \ref{fig:NLvarmion}(b)), but which also satisfies the above property of compressing the branch input vector to a $p$ dimensional vector.
\end{itemize} }

\revii{}{We consider a number of architectures for the eikonal problem, a summary of whose performance is given in Table \ref{tab:result_summary_NL}. Note that the numbers in the brackets denotes the widths of the hidden layers in the branch. The details of their architectures can be found in the Appendix. As is evident from Table \ref{tab:result_summary_NL} and the error histograms shown in Figure \ref{fig:nonlinear_hist}, both VarMiON architectures lead to the best performance in terms of the relative $L_2$ error on the test samples, with the constrained VarMiON-c architecture performing marginally better. We also consider DeepONet architectures with a branch that compresses the input vector of size 1024 more gradually. Keeping the number of trainable parameters similar to the VarMiONs affords us a DeepONet branch with a single hidden layer of width 130. The error with this operator network is more than twice of the VarMiON networks. Increasing the hidden layer width to be 200 pulls down the mean error but at the cost of having more trainable parameters. We also look at the extreme case where the DeepONet branch has as many layers as the VarMiONs but with a gradually tapering down architecture, i.e., we consider four hidden layers with widths 512, 256, 128 and 100. With this DeepONet architecture, the error is brought down to the range of the VarMiONs. However, the number of parameters is five times that of the VarMiONs, which leads to larger training and evaluation times. These results indicate that using a branch whose first hidden layer compresses the input to a vector of size of the latent dimension leads to the optimal results, while also controlling the size of the network. Further, enforcing the solution operator's homogeneous constraint leads to the expected $u=0$ predictions when $f=0$, as shown in Figure \ref{fig:nonlinear_zero_preds}. The predictions using the VarMiON and DeepONet architectures with similar sizes on a few test samples are shown in Figure \ref{fig:nonlinear_preds}.}  
\begin{table}[ht]
\renewcommand{\arraystretch}{1.}
\centering
\caption{\revii{}{Summary of operator network performance for the eikonal problem. The numbers in the brackets denite the widths of the hidden layers in the branch.}}
\begin{adjustbox}{width=0.85\linewidth}
\begin{tabular}{c c c}
\toprule
Model & Number of parameters & Relative $L_2$ error \\
\toprule
DeepONet (130) & 146,650 & 5.78 $\pm$ 1.50 \% \\
DeepONet (200) & 225,400 & 4.79 $\pm$ 1.46 \% \\
DeepONet (512,256,128,100) & 712,324 & 2.35 $\pm$ 0.36 \% \\
VarMiON (100,100,100,100) & 143,200 & 2.44 $\pm$ 0.41 \% \\
VarMiON-c (100,100,100,100) & 143,200 & 2.21 $\pm$ 0.43 \% \\
\bottomrule
\end{tabular}\label{tab:result_summary_NL}
\end{adjustbox}
\end{table}

\begin{figure}[htbp]
\begin{center}
\includegraphics[width=\textwidth]{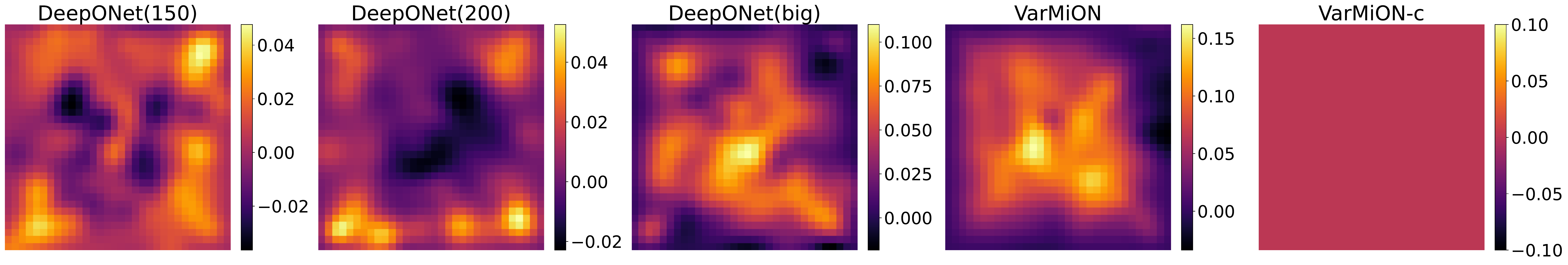}
\caption{\revii{}{Eikonal solution predictions for a zero source input. DeepONet(big) corresponds to DeepONet (512,256,128,100).}}
\label{fig:nonlinear_zero_preds}
\end{center}
\end{figure} 

\begin{figure}[htbp]
\begin{center}
\includegraphics[width=\textwidth]{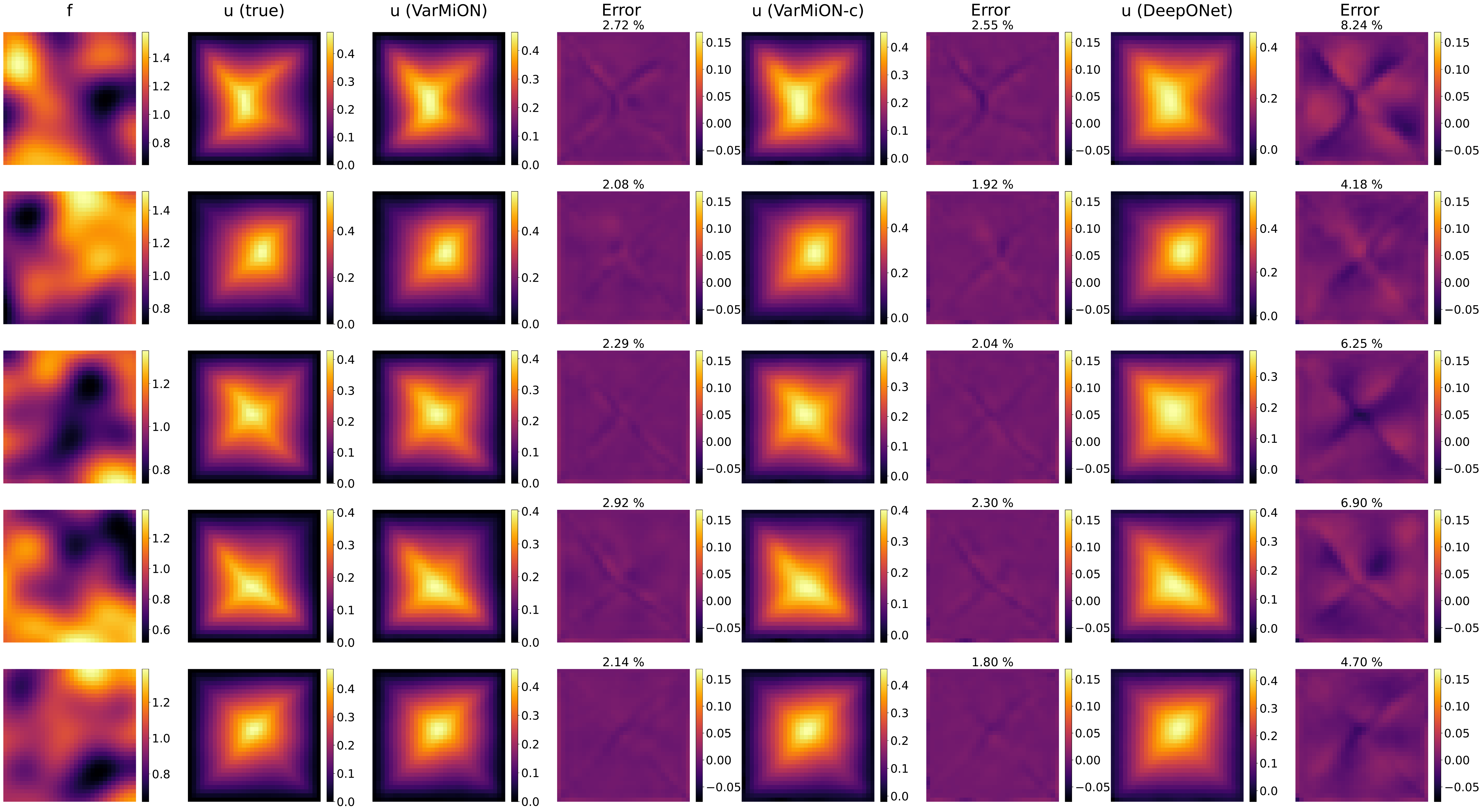}
\caption{\revii{}{\textit{Eikonal solution predictions for five representative samples from the test set}: (First column) source field, (second column) true solution field, (third, fifth and seventh column) solution field prediction by VarMiON (100,100,100,100), VarMiON-c (100,100,100,100) and DeepONet (130) respectively, (fourth, sixth and eighth column) corresponding error (normalized $L_2$ error is shown at the top).}}
\label{fig:nonlinear_preds}
\end{center}
\end{figure}

\begin{figure}[htbp]
\begin{center}
\includegraphics[width=0.99 \linewidth]{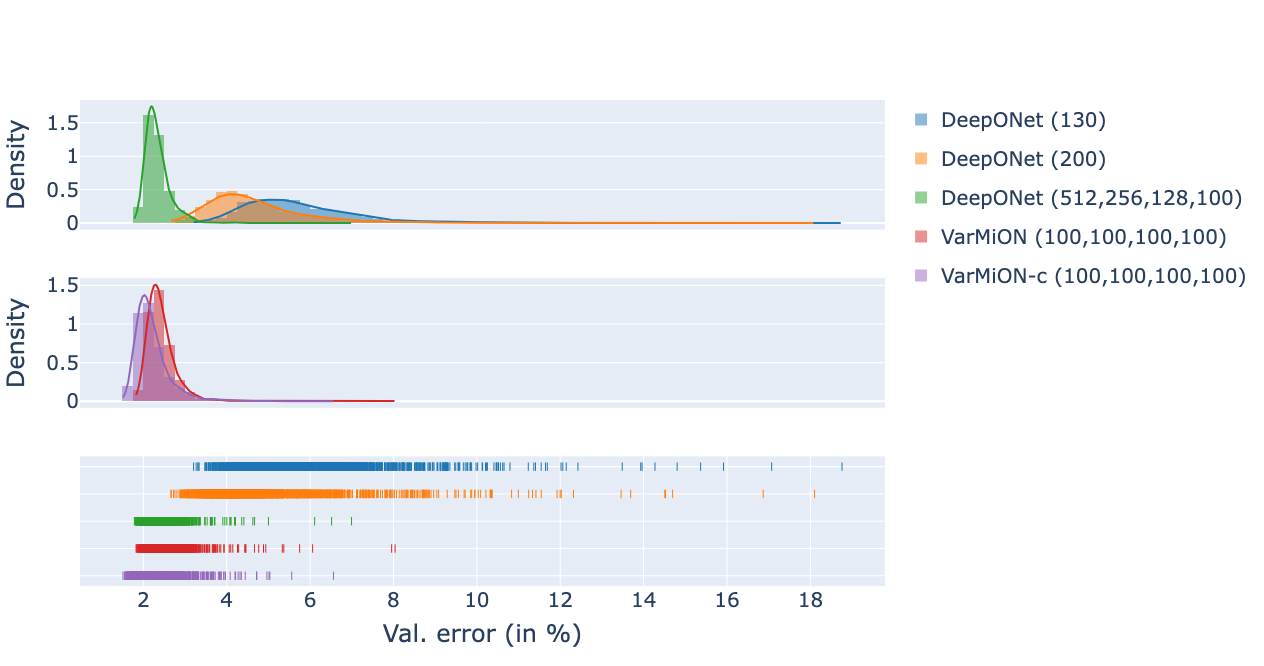}
\end{center}
\caption{\revii{}{Probability density for DeepONets (top), VarMiONs (middle) and rug plot (bottom) of the test error for the cases associated with the eikonal problem. In the rug plot each tick indicates one sample.}}
\label{fig:nonlinear_hist}
\end{figure}

\section{Conclusions}

\label{sec:con}

In this manuscript we have proposed a novel architecture for an operator network that maps input functions for a\reviii{n elliptic}{given} PDE, including forcing functions, boundary data and material property functions, to its solution. This architecture mimics the form of the numerical solution obtained by approximating the corresponding variational formulation of the PDE. For this reason we refer to this operator as the variationally mimetic operator network, or \net. Like the conventional Deep Operator Network (DeepONet), the \net can also be decomposed into branch and trunk networks, where the latter constructs the basis functions for representing the output, and the former generates the coefficients for this expansion. However, in contrast to a conventional DeepONet, the \net prescribes a precise architecture for the branch network. In the \reviii{}{linear PDE} case considered in this manuscript, this involves linear branches for the forcing function and Neumann data vectors and a nonlinear branch for the material property vector that leads to a matrix of small dimension (say, $100 \times 100$). The output of the branch network is obtained by summing the vectors of the linear branches and then computing the product of the material property matrix with the resulting vector. \reviii{}{For the nonlinear PDE problem, variational formulation leads to distinct architecture, where the sum of the output of the linear branches (corresponding to forcing function and Neumann data) is fed to a nonlinear network. Furthermore, an extension of this architecture is also proposed (called VarMiON-c), in which the homogeneity constraint of the underlying solution operator is encoded in the architecture.}

An analysis of the error in the solution generated by this network reveals several important contributions. These include, the training error, error in the solutions used to train the network, quadrature error in sampling the input and output functions during training, and the distance between the test input functions and the ``closest'' functions in the training dataset. This last component is multiplied by the sum of the stability constants for the true and the \net operators. This clear delineation of errors provides the user with a systematic approach to thinking about the performance of \net. Further, the special structure of the \net allows the identification of the precise operators responsible for these properties. 

The application of the \net to a canonical \reviii{}{linear} elliptic PDE \reviii{}{and generic nonlinear PDE} reveals several interesting results. First, for approximately the same number of network parameters, on average the \net incurs smaller error than a vanilla DeepONet \reviii{}{and MIONet}. Second, across large instances of test functions the distribution of this error for the \net is much tighter, thereby indicating that it is more robust to variations in input data. \reviii{Finally}{Third}, we recognize that this performance is robust to techniques used to sampling the input functions (random or regular), to different basis functions (ReLU or radial basis functions) and to the number of inputs functions (two or three). \reviii{}{Finally, VarMiON performs consistently better than baseline methods at various sizes of training datasets.}

There are several directions for future work that emerge from this work. These include the application of the \net as a surrogate model for problems in optimization and uncertainty quantification. 
In the solution to these problems many solutions of the forward problem are required, and the \net can be used to compute these quickly and accurately. 
\reviii{}{Extending the VarMiON architecture to solve other challenging nonlinear PDE models, and developing the associated theoretical framework to estimate the approximation error would be a fruitful endeavor. While we believe that the philosophy of mimicking the weak form will carry over, the precise architecture of the VarMiON may greatly differ depending on the weak formulation of the PDE model.}
In a similar vein, the ideas developed in this manuscript may be extended to time-dependent and hyperbolic PDEs, where the \net architecture could be applied to solve a series of linear problems. These, and related ideas, will be explored in future work. 

\section*{Acknowledgements}
AAO and DR acknowledge support from ARO grant W911NF2010050. DP acknowledges support from the Stephen Timoshenko Distinguished Postdoctoral Fellowship at Stanford University. 


\newpage 

\appendix

\revii{}{\section{DeepONet, VarMiON and MIONet architectures}}
We descibe the key network blocks used to construct the various DeepONets, \revii{}{MIONets} and VarMiONs considered in this work:  

\begin{itemize}
    
    \item {\tt Dense(k)} denotes a fully connected layer of width {\tt k}.

    \reviii{}{\item {\tt Linear(k)} denotes a fully connected dense layer of width {\tt k} but without a bias vector.}
    
    \item {\tt TrConv(k,n,s)} denotes a 2D transpose convolution with {\tt k} output filters of size {\tt (n,n)} and stride {\tt (s,s)}.
    
    \item {\tt ReLU} denotes the ReLU activation \reviii{}{while \tt{TanS} denotes the TanhShrink activation.}

    \item {\tt BN} denotes batch normalization.

    \item {\tt Reshape(q)} is used to reshape the incoming tensor into a shape specified by the tuple {\tt q}.
    
    
    \item {\tt RBF(n,m)} denotes a reduced basis function layer that evaluates $m$ scalar-valued parametrized functions on an $n$ dimensional input. In particular, for an input vector $\x \in \Ro^n$, the output $\y \in \Ro^m$ is given by
    \[
    y_i = \exp{\left(-\frac{\|\x - \bm{c}_i\|^2}{\sigma_i^2}\right)} \quad 1 \leq i \leq m,
    \]
    where vectors $\bm{c}_i \in \Ro^n$ and the scalars $\sigma_i$ are trainable parameters.

    \reviii{}{\item \tt{Input()} and \tt{Output()} denote the size of the input and output vectors, and are not computable layers.}
\end{itemize}

\subsection{Two input functions with randomly sampled input} 
\textbf{DeepONet branch:}\\
\texttt{
Input(200,1) \ra Dense(170) \ra ReLU \ra Dense(170) \ra ReLU \ra Dense(64) \ra  Output(64,1)
}

\textbf{VarMiON $\Th$ branch:}\\
\texttt{
Input(100,1) \ra Dense(100) \ra ReLU \ra Dense(512) \ra ReLU \ra Reshape(4,4,32) \ra TrConv(16,2,2) \ra ReLU \ra BN \ra TrConv(16,2,2) \ra ReLU \ra BN \ra TrConv(8,2,2) \ra ReLU \ra BN\ra TrConv(1,2,2) \ra \reviii{ReLU}{TanS} \ra Output(64,64)
}

\textbf{VarMiON $\Fh$ branch:}\\
\texttt{
Input(100,1) \ra \reviii{Dense}{Linear(64)} \ra Output(64,1)
}

\textbf{DeepONet/VarMiON trunk:}\\ 
\texttt{
Input(2,1) \ra Dense(100) \ra ReLU \ra Dense(100) \ra ReLU \ra Dense(100) \ra ReLU \ra Dense(100) \ra ReLU \ra Dense(64) \ra Output(64,1)
}

Also see schematics in Figure \ref{fig:arch_2input_rand}.

\begin{figure}[htbp]
\begin{center}
\subfigure[DeepONet]{
\includegraphics[width=0.35\textwidth]{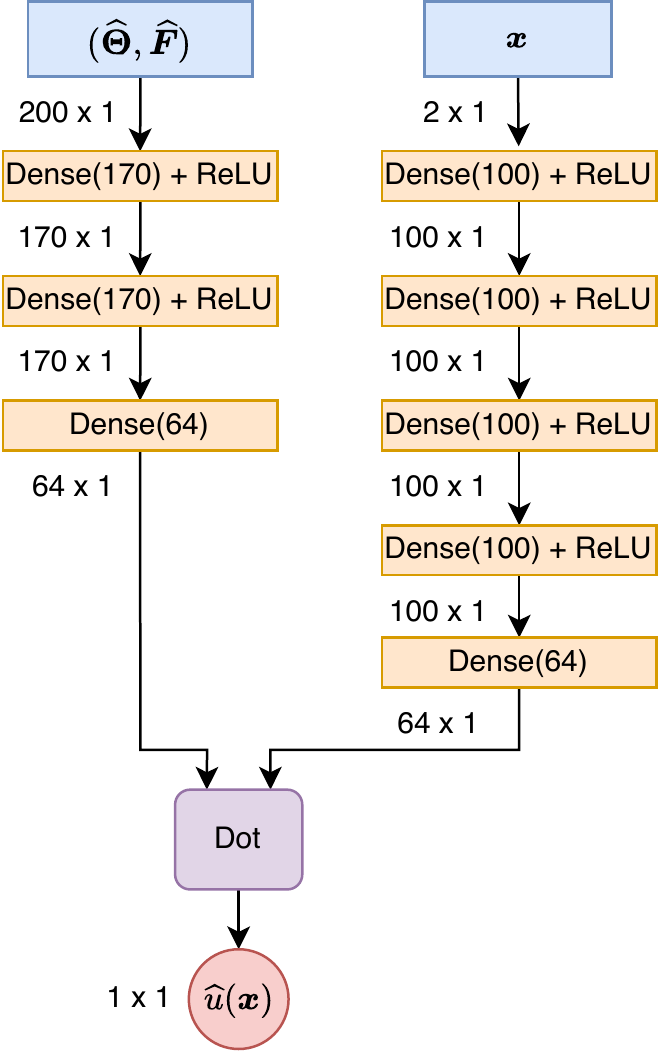}}\hfill
\subfigure[VarMiON]{
\includegraphics[width=0.55\textwidth]{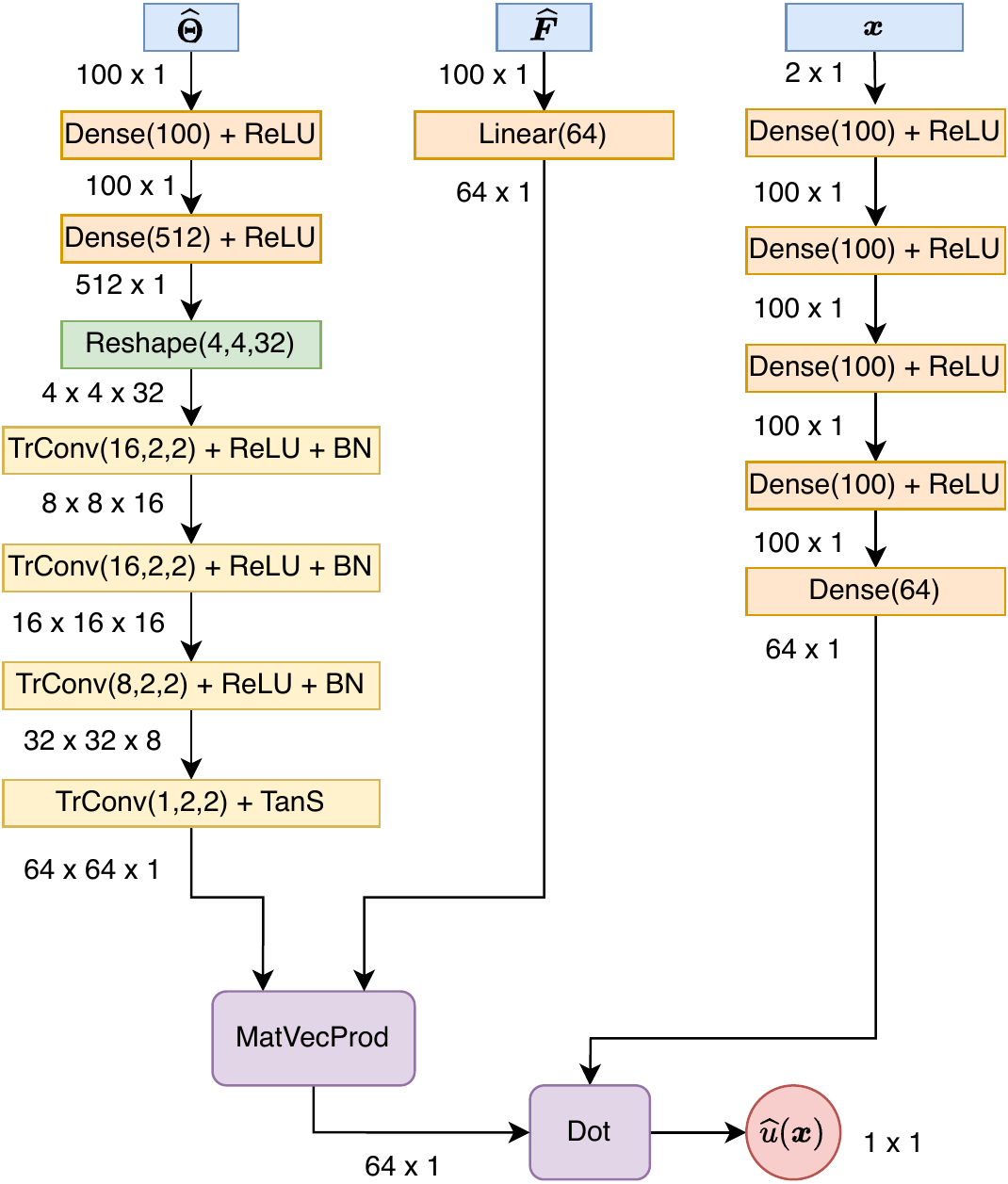}}
\caption{\reviii{}{Operator architectures for two input functions with randomly sampled input}}
\label{fig:arch_2input_rand}
\end{center}
\end{figure}

\subsection{Two input functions with uniformly sampled input and ReLU trunk} \label{app:a2}

\textbf{DeepONet branch:}\\
\texttt{
Input(200,1) \ra Dense(64) \ra ReLU \ra  Output(64,1)
}

\textbf{VarMiON $\Th$ branch:}\\
\texttt{
Input(10,10) \ra TrConv(8,4,1) \ra ReLU \ra BN \ra TrConv(16,4,1) \ra ReLU \ra BN \ra TrConv(8,2,2) \ra ReLU \ra BN\ra TrConv(1,2,2) \ra \reviii{ReLU}{TanS} \ra Output(64,64)
}

\textbf{VarMiON $\Fh$ branch:}\\
\texttt{
Input(100,1) \ra \reviii{Dense}{Linear}(64) \ra Output(64,1)
}

\textbf{DeepONet/VarMiON trunk:}\\ 
\texttt{
Input(2,1) \ra Dense(100) \ra ReLU \ra Dense(100) \ra ReLU \ra Dense(100) \ra ReLU \ra Dense(100) \ra ReLU \ra Dense(64) \ra Output(64,1)
}

Also see schematics in Figure \ref{fig:arch_2input_uniform_relu}.

\begin{figure}[htbp]
\begin{center}
\subfigure[DeepONet]{
\includegraphics[width=0.35\textwidth]{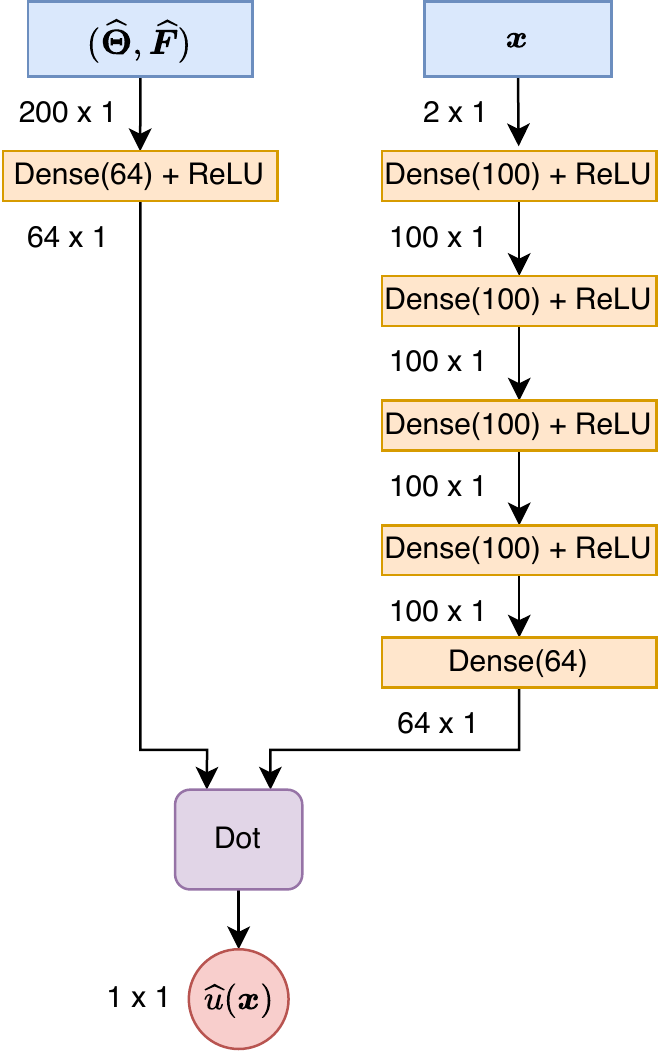}}\hfill
\subfigure[VarMiON]{
\includegraphics[width=0.55\textwidth]{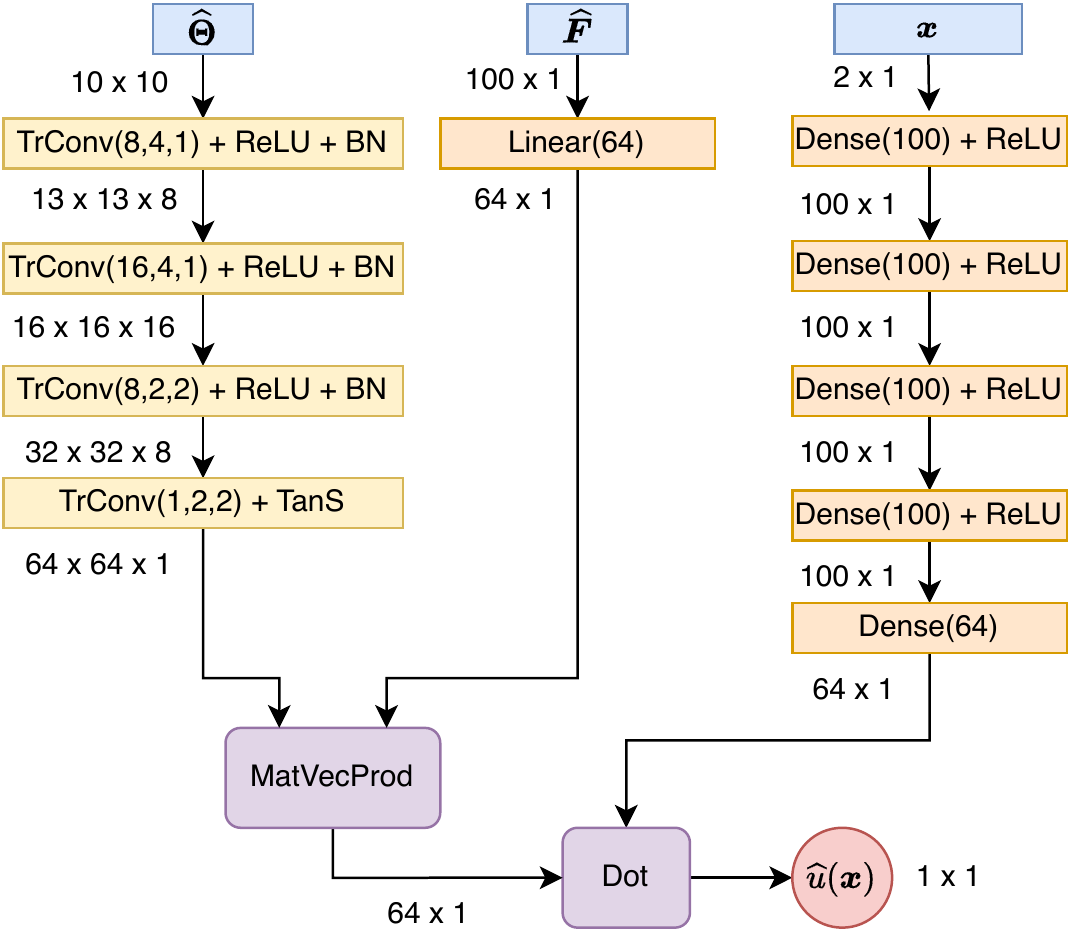}}
\caption{Operator architectures for two input functions with uniformly samples input and ReLU trunk}
\label{fig:arch_2input_uniform_relu}
\end{center}
\end{figure}

\subsection{Two input functions with uniformly sampled input and RBF trunk}

\textbf{DeepONet branch:}\\
\texttt{
Input(200,1) \ra Dense(55) \ra ReLU \ra Dense(55) \ra ReLU \ra Dense(64) \ra  Output(64,1)
}

\textbf{VarMiON $\Th$ branch:}\\
\texttt{
Input(10,10) \ra TrConv(16,4,1) \ra ReLU \ra BN \ra TrConv(32,4,1) \ra ReLU \ra BN \ra TrConv(16,2,2) \ra ReLU \ra BN\ra TrConv(1,2,2) \ra \reviii{ReLU}{TanS} \ra Output(64,64)
}

\textbf{VarMiON $\Fh$ branch:}\\
\texttt{
Input(100,1) \ra \reviii{Dense}{Linear}(64) \ra Output(64,1)
}

\revii{}{\textbf{MIONet $\Th$ branch:}\\
\texttt{
Input(100,1) \ra Dense(64) \ra ReLU \ra Dense(64) \ra Output(64,1)
}}

\revii{}{\textbf{MIONet $\Fh$ branch:}\\
\texttt{
Input(100,1) \ra Linear(64) \ra Output(64,1)
}}

\textbf{DeepONet/VarMiON\revii{}{/MIONet} trunk:}\\  
\texttt{
Input(2,1) \ra RBF(2,64) \ra Output(64,1)
}

Also see schematics in Figure \ref{fig:arch_2input_uniform_rbf}.

\begin{figure}[htbp]
\begin{center}
\subfigure[DeepONet]{
\includegraphics[width=0.35\textwidth]{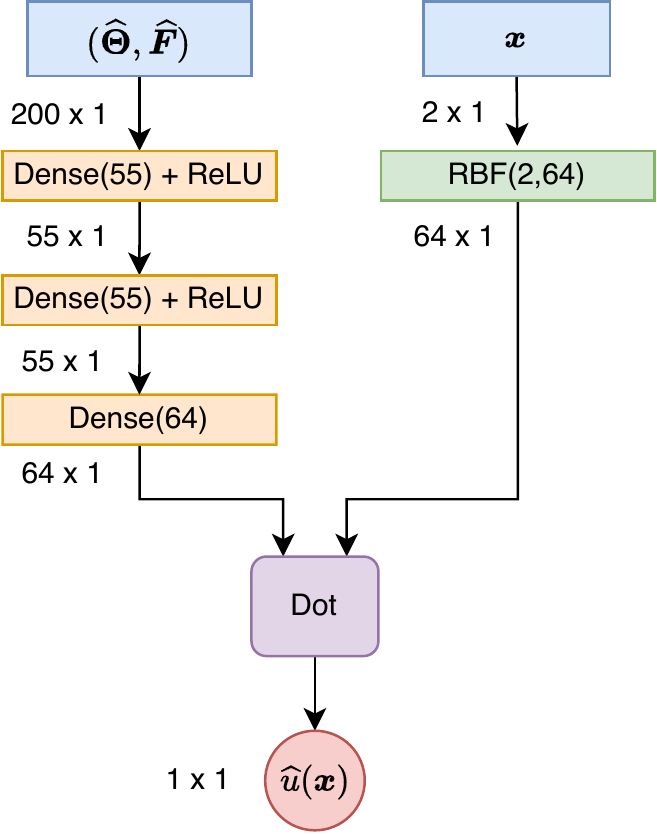}}\hfill
\subfigure[VarMiON]{
\includegraphics[width=0.55\textwidth]{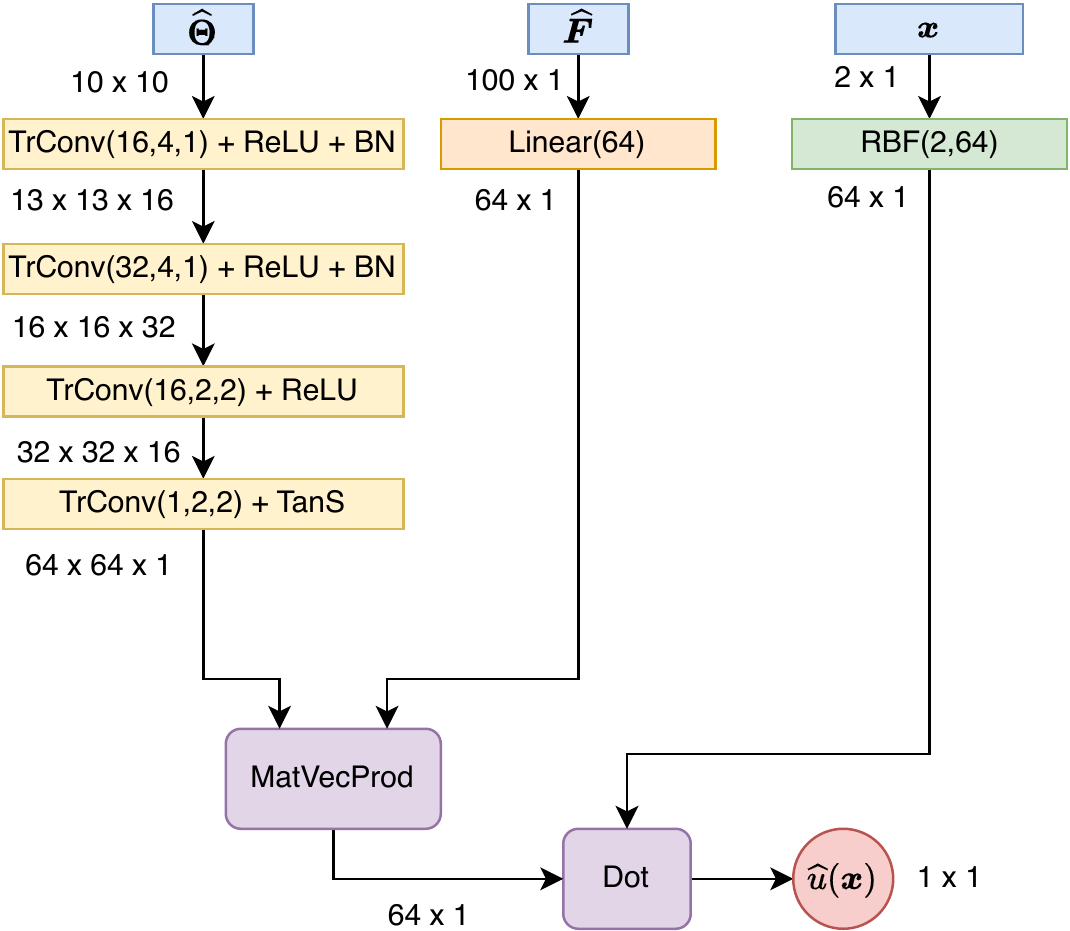}}
\subfigure[MIONet]{
\includegraphics[width=0.55\textwidth]{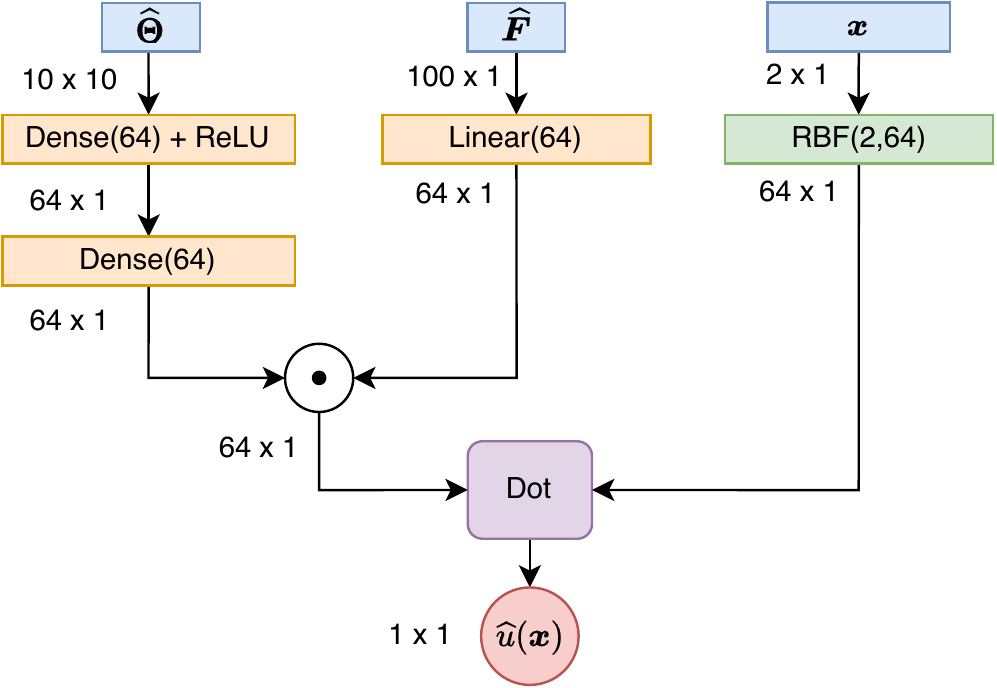}}
\caption{Operator architectures for two input functions with uniformly samples input and RBF trunk}
\label{fig:arch_2input_uniform_rbf}
\end{center}
\end{figure}

\subsection{Three input functions with uniformly sampled input} 
\reviii{}{The input $\Nh$ is built by evaluating the Neumann data at 12 uniformly placed nodes on each of the top and bottom boundaries that form $\Gamma_\eta$.}
\textbf{DeepONet branch:}\\
\texttt{
\reviii{Input(432,1)}{Input(312,1)} \ra Dense(55) \ra ReLU \ra Dense(55) \ra ReLU \ra Dense(72) \ra  Output(72,1)
}

\textbf{VarMiON $\Th$ branch:}\\
\texttt{
Input(12,12) \ra TrConv(16,4,1) \ra ReLU \ra BN \ra TrConv(32,4,1) \ra ReLU \ra BN \ra TrConv(16,2,2) \ra ReLU \ra BN\ra TrConv(1,2,2) \ra \reviii{ReLU}{TanS} \ra Output(72,72)
}

\textbf{VarMiON $\Fh$ branch:}\\
\texttt{
Input(144,1) \ra \reviii{Dense}{Linear}(72) \ra Output(72,1)
}

\textbf{VarMiON $\Nh$ branch:} \reviii{The input is evaluated on $12 \times 12$ grid including the domain boundary nodes. Then all the values at the nodes not on $\Gamma_\eta$ are set to zero.}{}\\
\texttt{
\reviii{Input(144,1)}{Input(24,1)} \ra \reviii{Dense}{Linear}(72) \ra Output(72,1)
}

\revii{}{\textbf{MIONet $\Th$ branch:}\\
\texttt{
Input(144,1) \ra Dense(72) \ra ReLU \ra Output(72,1)
}}

\revii{}{\textbf{MIONet $\Fh$ branch:}\\
\texttt{
Input(144,1) \ra Linear(72) \ra Output(72,1)
}}

\revii{}{\textbf{MIONet $\Nh$ branch:}\\
\texttt{
Input(24,1) \ra Linear(72) \ra Output(72,1)
}}

\textbf{DeepONet/VarMiON\revii{}{/MIONet} trunk:}\\ 
\texttt{
Input(2,1) \ra RBF(2,72) \ra Output(72,1)
}

Also see schematics in Figure \ref{fig:arch_3input_uniform_rbf}.

\begin{figure}[htbp]
\begin{center}
\subfigure[DeepONet]{
\includegraphics[width=0.32\textwidth]{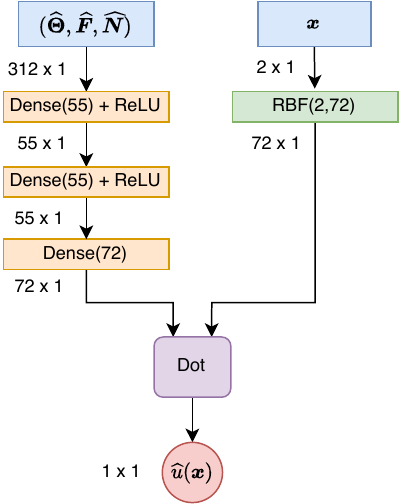}}\hfill
\subfigure[VarMiON]{
\includegraphics[width=0.65\textwidth]{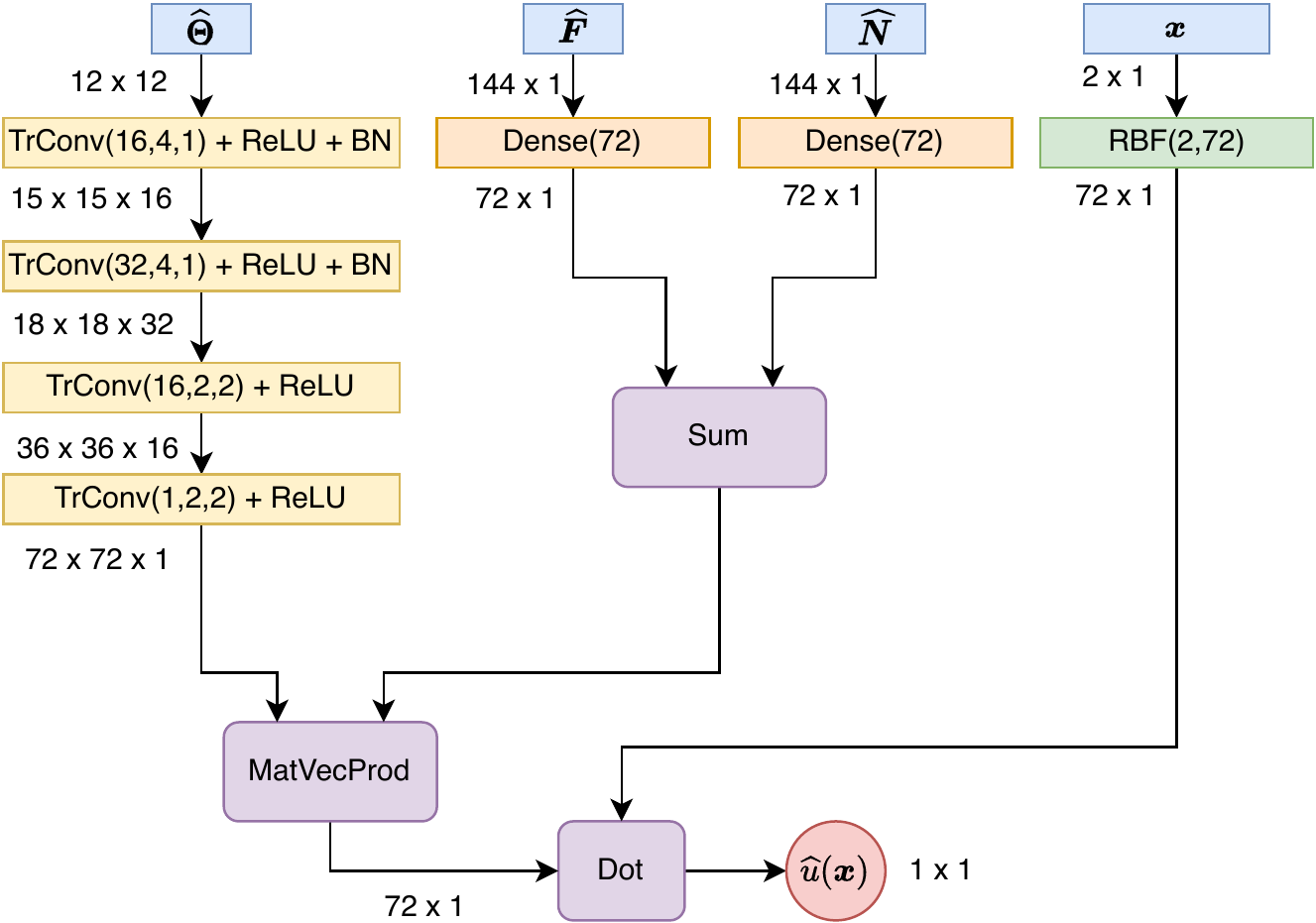}}
\subfigure[MIONet]{
\includegraphics[width=0.65\textwidth]{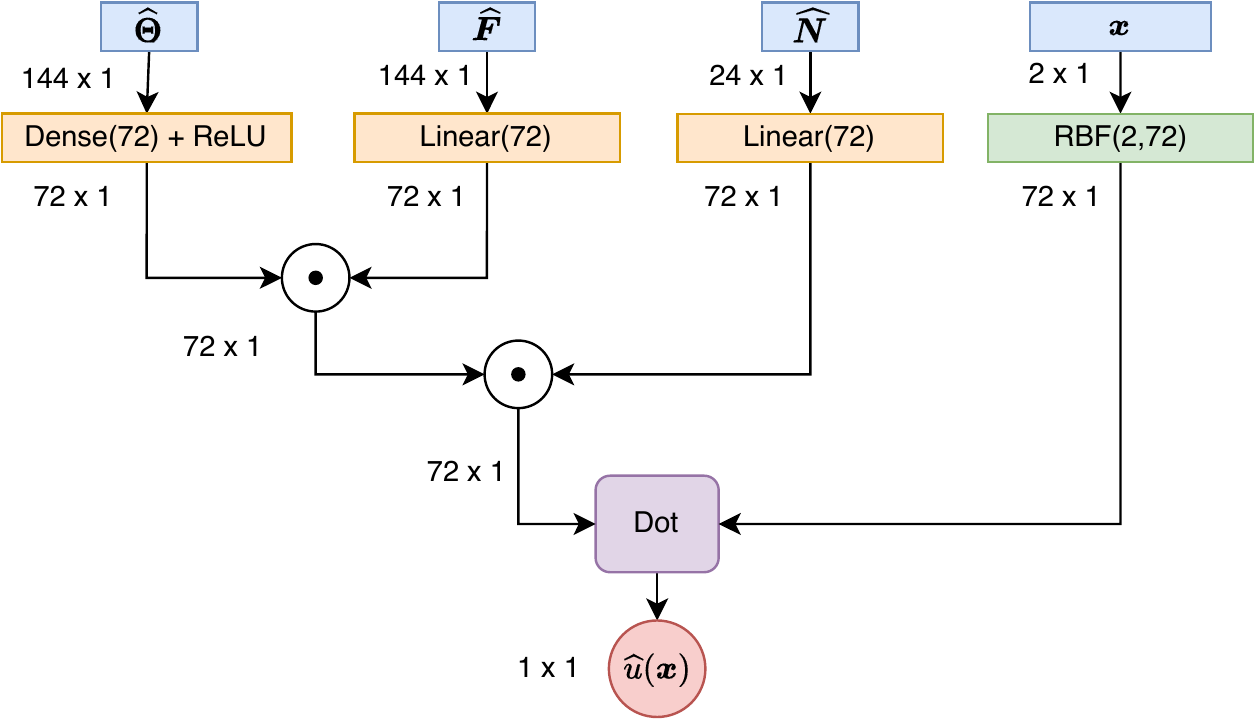}}
\caption{Operator architectures for three input functions with uniformly samples input and RBF trunk}
\label{fig:arch_3input_uniform_rbf}
\end{center}
\end{figure}

\revii{}{\subsection{Nonlinear problem}} 

\revii{}{\textbf{DeepONet (130) branch:}\\
\texttt{
Input(1024,1) \ra Dense(130) \ra ReLU \ra Dense(100) \ra  Output(100,1)
}}

\revii{}{\textbf{DeepONet (200) branch:}\\
\texttt{
Input(1024,1) \ra Dense(200) \ra ReLU \ra Dense(100) \ra  Output(100,1)
}}

\revii{}{\textbf{DeepONet (512,256,128,100) branch:}\\
\texttt{
Input(1024,1) \ra Dense(512) \ra ReLU \ra Dense(256) \ra ReLU \ra Dense(128) \ra ReLU \ra Dense(100) \ra ReLU \ra Dense(100) \ra  Output(100,1)
}}

\revii{}{\textbf{VarMiON (100,100,100,100) branch:}\\
\texttt{
Input(1024,1) \ra Linear(100) \ra ReLU \ra Dense(100) \ra ReLU \ra Dense(100) \ra ReLU \ra Dense(100) \ra ReLU \ra Dense(100) \ra  Output(100,1)
}}

\revii{}{\textbf{VarMiON-c (100,100,100,100) branch:}\\
\texttt{
Input(1024,1) \ra Z = Linear(100) \ra ReLU \ra Dense(100) \ra ReLU \ra Dense(100) \ra ReLU \ra Dense(100) \ra ReLU \ra Dense(100) $\odot$ Z \ra Output(100,1)
}}

\revii{}{\textbf{DeepONet/VarMiON/VarMiON-c trunk:}\\ 
\texttt{
Input(2,1) \ra RBF(2,100) \ra Output(100,1)
}}

Also see schematics in Figure \ref{fig:arch_NL}.

\begin{figure}[htbp]
\begin{center}
\subfigure[DeepONet (130)]{
\includegraphics[width=0.32\textwidth]{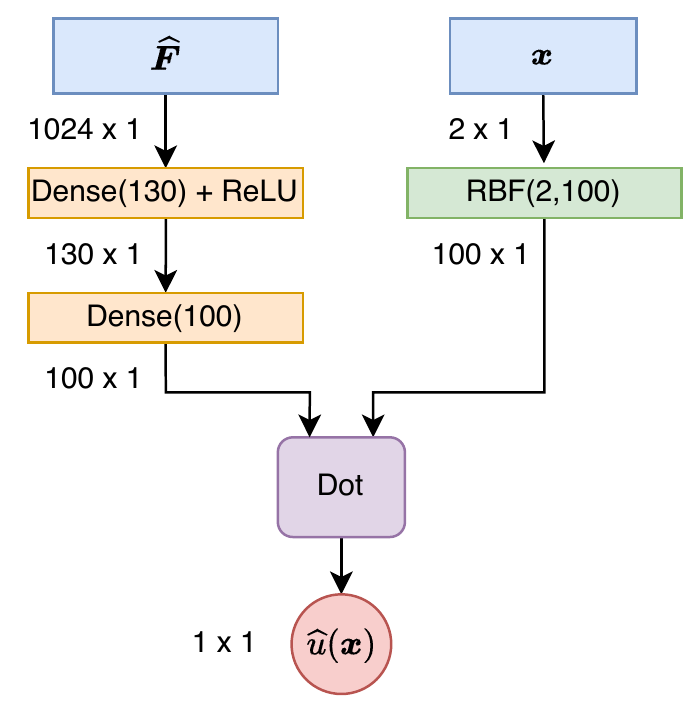}}\hfill
\subfigure[DeepONet (200)]{
\includegraphics[width=0.32\textwidth]{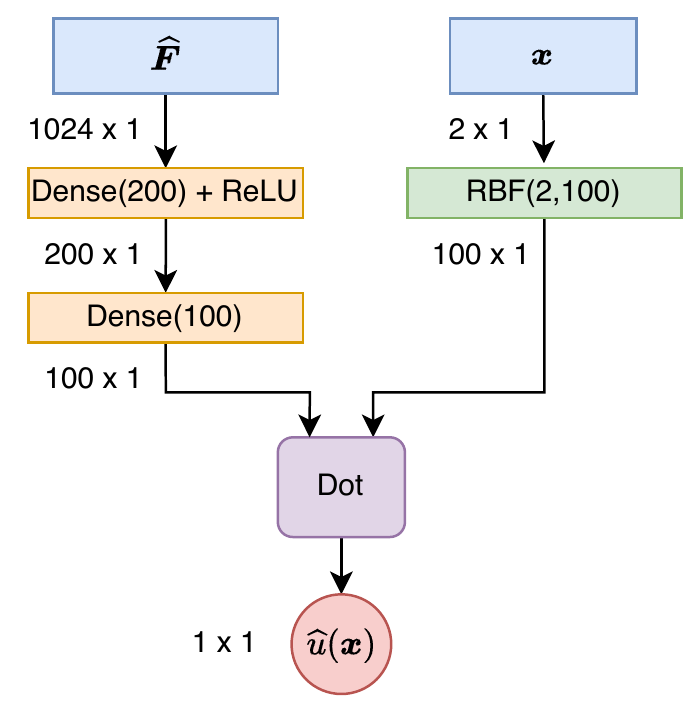}}\hfill
\subfigure[DeepONet (512,256,128,100)]{
\includegraphics[width=0.32\textwidth]{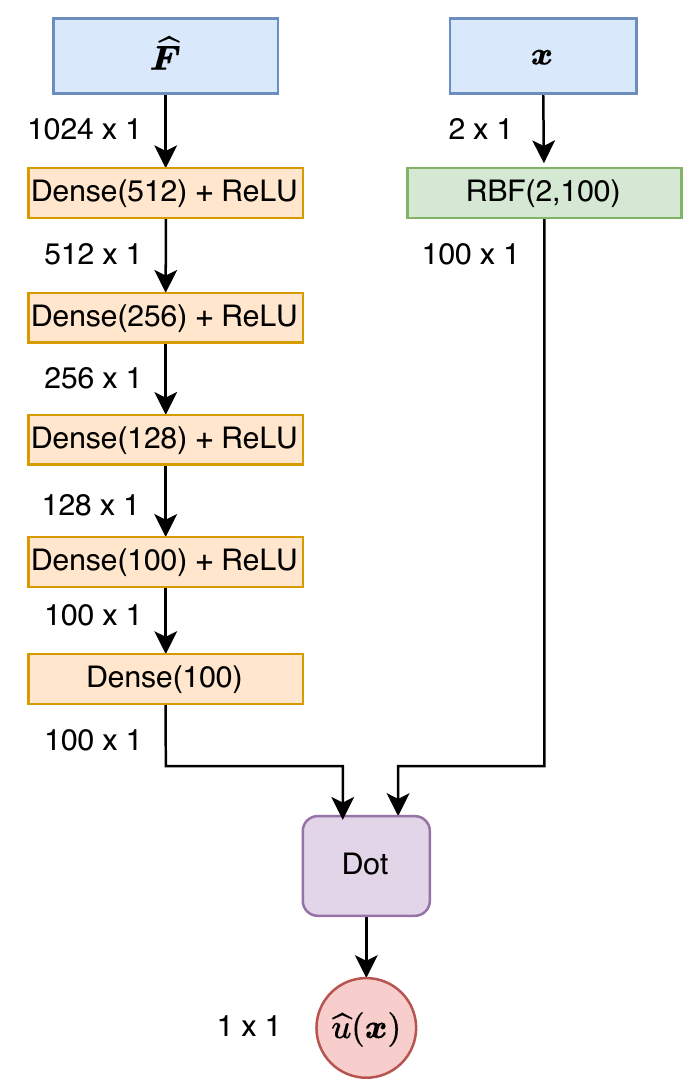}}\hfill
\subfigure[VarMiON (100,100,100,100)]{
\includegraphics[width=0.33\textwidth]{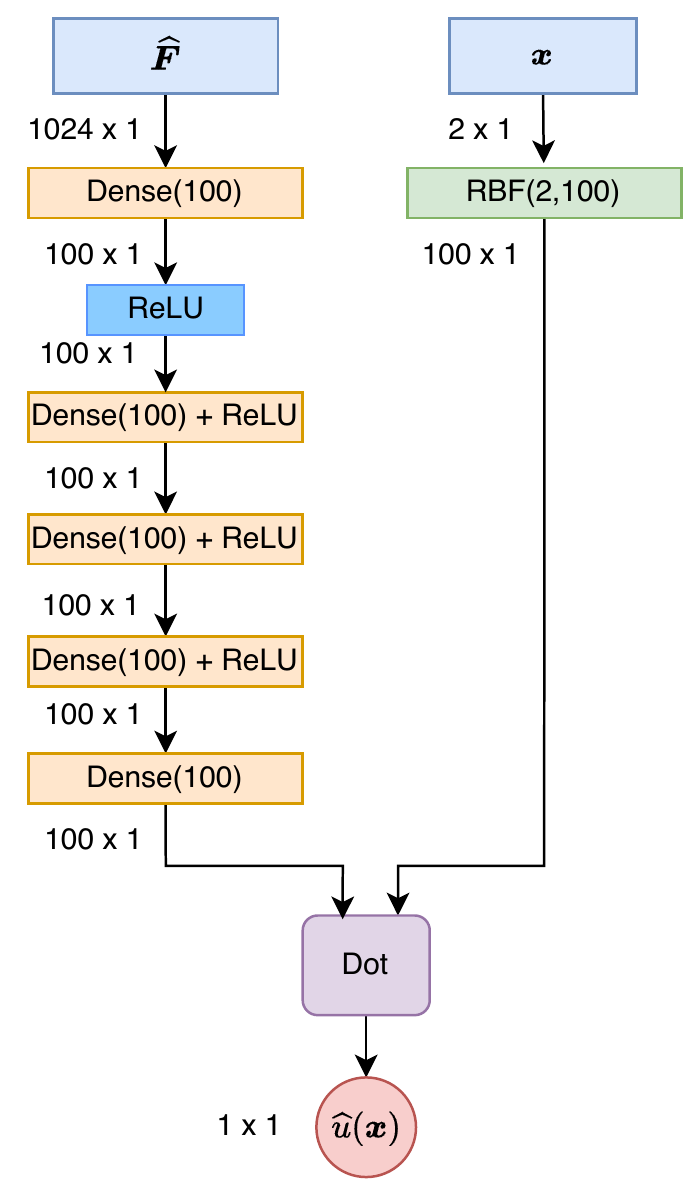}}\hfill
\subfigure[VarMiON-c (100,100,100,100)]{
\includegraphics[width=0.33\textwidth]{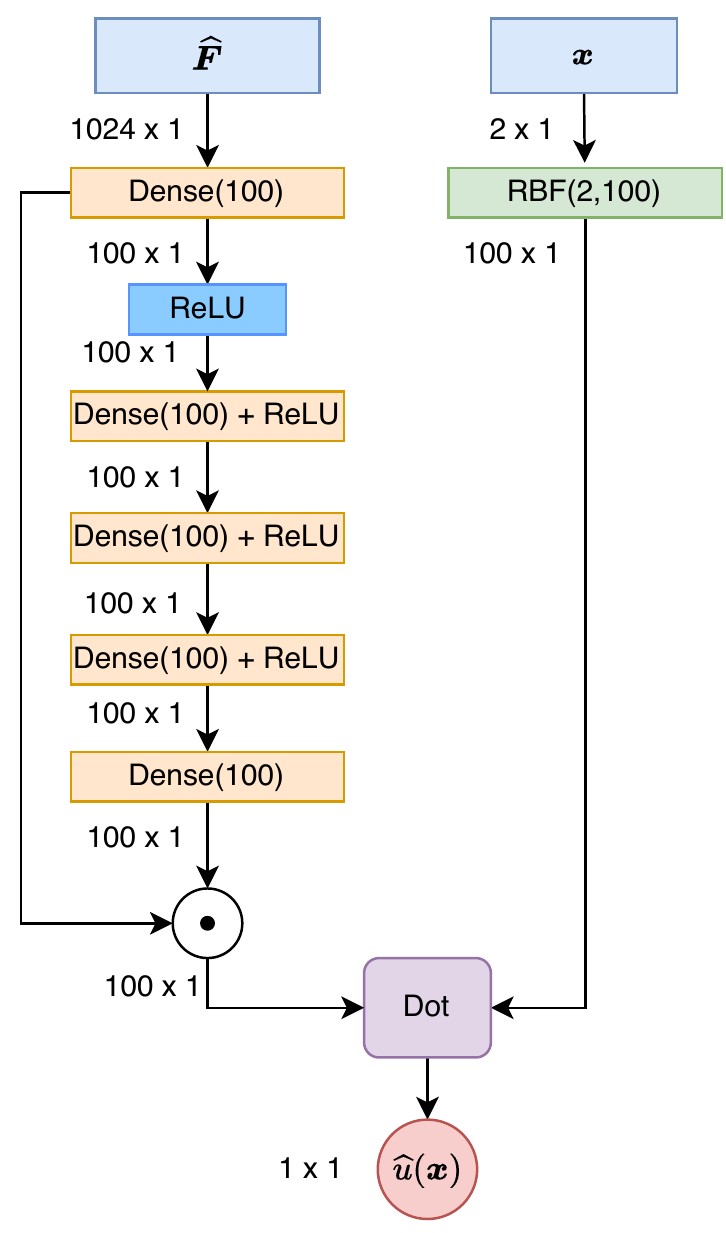}}
\caption{Operator architectures with RBF trunk for eikonal problem}
\label{fig:arch_NL}
\end{center}
\end{figure}

\newpage 

\bibliographystyle{plain} 
\bibliography{refs}

\begin{thebibliography}{10}

\bibitem{bazilevs2007weak}
Yuri Bazilevs and Thomas~JR Hughes.
\newblock Weak imposition of dirichlet boundary conditions in fluid mechanics.
\newblock {\em Computers \& fluids}, 36(1):12--26, 2007.

\bibitem{berkooz1993proper}
Gal Berkooz, Philip Holmes, and John~L Lumley.
\newblock The proper orthogonal decomposition in the analysis of turbulent
  flows.
\newblock {\em Annual review of fluid mechanics}, 25(1):539--575, 1993.

\bibitem{bonito2013}
Andrea Bonito, Ronald~A. DeVore, and Ricardo~H. Nochetto.
\newblock Adaptive finite element methods for elliptic problems with
  discontinuous coefficients.
\newblock {\em SIAM Journal on Numerical Analysis}, 51(6):3106--3134, 2013.

\bibitem{brenner2002mathematical}
S.~Brenner and L.R. Scott.
\newblock {\em The Mathematical Theory of Finite Element Methods}.
\newblock Texts in Applied Mathematics. Springer New York, 2002.

\bibitem{cai2021}
Shengze Cai, Zhicheng Wang, Lu~Lu, Tamer~A. Zaki, and George~Em Karniadakis.
\newblock Deepm\&mnet: Inferring the electroconvection multiphysics fields
  based on operator approximation by neural networks.
\newblock {\em Journal of Computational Physics}, 436:110296, 2021.

\bibitem{chen95rbf}
Tianping Chen and Hong Chen.
\newblock Approximation capability to functions of several variables, nonlinear
  functionals, and operators by radial basis function neural networks.
\newblock {\em IEEE Transactions on Neural Networks}, 6(4):904--910, 1995.

\bibitem{chen1995universal}
Tianping Chen and Hong Chen.
\newblock Universal approximation to nonlinear operators by neural networks
  with arbitrary activation functions and its application to dynamical systems.
\newblock {\em IEEE Transactions on Neural Networks}, 6(4):911--917, 1995.

\bibitem{cuomo2022}
Salvatore Cuomo, Vincenzo~Schiano Di~Cola, Fabio Giampaolo, Gianluigi Rozza,
  Maziar Raissi, and Francesco Piccialli.
\newblock Scientific machine learning through physics--informed neural
  networks: Where we are and what's next.
\newblock {\em Journal of Scientific Computing}, 92(3):88, 2022.

\bibitem{jagtap2020}
Ameya D.~Jagtap and George Em~Karniadakis.
\newblock Extended physics-informed neural networks (xpinns): A generalized
  space-time domain decomposition based deep learning framework for nonlinear
  partial differential equations.
\newblock {\em Communications in Computational Physics}, 28(5):2002--2041,
  2020.

\bibitem{deryck2022}
Tim De~Ryck, Siddhartha Mishra, and Roberto Molinaro.
\newblock wpinns: Weak physics informed neural networks for approximating
  entropy solutions of hyperbolic conservation laws, 2022.

\bibitem{gilbarg2001elliptic}
D.~Gilbarg and N.S. Trudinger.
\newblock {\em Elliptic Partial Differential Equations of Second Order}.
\newblock Classics in Mathematics. Springer Berlin Heidelberg, 2001.

\bibitem{goswami2022physics}
Somdatta Goswami, Minglang Yin, Yue Yu, and George~Em Karniadakis.
\newblock A physics-informed variational deeponet for predicting crack path in
  quasi-brittle materials.
\newblock {\em Computer Methods in Applied Mechanics and Engineering},
  391:114587, 2022.

\bibitem{hughes2012finite}
Thomas~JR Hughes.
\newblock {\em The finite element method: linear static and dynamic finite
  element analysis}.
\newblock Courier Corporation, 2012.

\bibitem{Jin2022}
Pengzhan Jin, Shuai Meng, and Lu~Lu.
\newblock Mionet: Learning multiple-input operators via tensor product.
\newblock {\em https://doi.org/10.1137/22M1477751}, 44:A3490--A3514, 11 2022.

\bibitem{kingma2017adam}
Diederik~P. Kingma and Jimmy Ba.
\newblock Adam: A method for stochastic optimization, 2017.

\bibitem{kovachki2021universal}
Nikola Kovachki, Samuel Lanthaler, and Siddhartha Mishra.
\newblock On universal approximation and error bounds for fourier neural
  operators.
\newblock {\em Journal of Machine Learning Research}, 22:Art--No, 2021.

\bibitem{kovachki2021}
Nikola Kovachki, Zongyi Li, Burigede Liu, Kamyar Azizzadenesheli, Kaushik
  Bhattacharya, Andrew Stuart, and Anima Anandkumar.
\newblock Neural operator: Learning maps between function spaces, 2021.

\bibitem{krishnapriyan2021characterizing}
Aditi Krishnapriyan, Amir Gholami, Shandian Zhe, Robert Kirby, and Michael~W
  Mahoney.
\newblock Characterizing possible failure modes in physics-informed neural
  networks.
\newblock {\em Advances in Neural Information Processing Systems},
  34:26548--26560, 2021.

\bibitem{lagaris2000neural}
Isaac~E Lagaris, Aristidis~C Likas, and Dimitris~G Papageorgiou.
\newblock Neural-network methods for boundary value problems with irregular
  boundaries.
\newblock {\em IEEE Transactions on Neural Networks}, 11(5):1041--1049, 2000.

\bibitem{lanthaler2022error}
Samuel Lanthaler, Siddhartha Mishra, and George~E Karniadakis.
\newblock Error estimates for deeponets: A deep learning framework in infinite
  dimensions.
\newblock {\em Transactions of Mathematics and Its Applications}, 6(1):tnac001,
  2022.

\bibitem{li2020fourier}
Zongyi Li, Nikola Kovachki, Kamyar Azizzadenesheli, Burigede Liu, Kaushik
  Bhattacharya, Andrew Stuart, and Anima Anandkumar.
\newblock Fourier neural operator for parametric partial differential
  equations.
\newblock {\em arXiv preprint arXiv:2010.08895}, 2020.

\bibitem{li2020neural}
Zongyi Li, Nikola Kovachki, Kamyar Azizzadenesheli, Burigede Liu, Kaushik
  Bhattacharya, Andrew Stuart, and Anima Anandkumar.
\newblock Neural operator: Graph kernel network for partial differential
  equations.
\newblock {\em arXiv preprint arXiv:2003.03485}, 2020.

\bibitem{li2020multipole}
Zongyi Li, Nikola Kovachki, Kamyar Azizzadenesheli, Burigede Liu, Andrew
  Stuart, Kaushik Bhattacharya, and Anima Anandkumar.
\newblock Multipole graph neural operator for parametric partial differential
  equations.
\newblock {\em Advances in Neural Information Processing Systems},
  33:6755--6766, 2020.

\bibitem{pino}
Zongyi Li, Hongkai Zheng, Nikola Kovachki, David Jin, Haoxuan Chen, Burigede
  Liu, Kamyar Azizzadenesheli, and Anima Anandkumar.
\newblock Physics-informed neural operator for learning partial differential
  equations, 2021.

\bibitem{lu2021learning}
Lu~Lu, Pengzhan Jin, Guofei Pang, Zhongqiang Zhang, and George~Em Karniadakis.
\newblock Learning nonlinear operators via deeponet based on the universal
  approximation theorem of operators.
\newblock {\em Nature Machine Intelligence}, 3(3):218--229, 2021.

\bibitem{mao2020}
Zhiping Mao, Ameya~D. Jagtap, and George~Em Karniadakis.
\newblock Physics-informed neural networks for high-speed flows.
\newblock {\em Computer Methods in Applied Mechanics and Engineering},
  360:112789, 2020.

\bibitem{mao2021}
Zhiping Mao, Lu~Lu, Olaf Marxen, Tamer~A. Zaki, and George~Em Karniadakis.
\newblock Deepm\&mnet for hypersonics: Predicting the coupled flow and
  finite-rate chemistry behind a normal shock using neural-network
  approximation of operators.
\newblock {\em Journal of Computational Physics}, 447:110698, 2021.

\bibitem{iglesias2016}
Andrew M.~Stuart Marco A.~Iglesias, Yulong~Lu.
\newblock A bayesian level set method for geometric inverse problems.
\newblock {\em Interface and Free Boundaries}, 18(2):181--217, 2016.

\bibitem{mishra2021}
Siddhartha Mishra and Roberto Molinaro.
\newblock {Estimates on the generalization error of physics-informed neural
  networks for approximating a class of inverse problems for PDEs}.
\newblock {\em IMA Journal of Numerical Analysis}, 42(2):981--1022, 06 2021.

\bibitem{mishra2022}
Siddhartha Mishra and Roberto Molinaro.
\newblock {Estimates on the generalization error of physics-informed neural
  networks for approximating PDEs}.
\newblock {\em IMA Journal of Numerical Analysis}, 01 2022.
\newblock drab093.

\bibitem{prasthofer2022}
Michael Prasthofer, Tim De~Ryck, and Siddhartha Mishra.
\newblock Variable-input deep operator networks, 2022.

\bibitem{raissi2019physics}
Maziar Raissi, Paris Perdikaris, and George~E Karniadakis.
\newblock Physics-informed neural networks: A deep learning framework for
  solving forward and inverse problems involving nonlinear partial differential
  equations.
\newblock {\em Journal of Computational physics}, 378:686--707, 2019.

\bibitem{shukla2021}
Khemraj Shukla, Ameya~D. Jagtap, and George~Em Karniadakis.
\newblock Parallel physics-informed neural networks via domain decomposition.
\newblock {\em Journal of Computational Physics}, 447:110683, 2021.

\bibitem{tan2022enhanced}
Lesley Tan and Liang Chen.
\newblock Enhanced deeponet for modeling partial differential operators
  considering multiple input functions.
\newblock {\em arXiv preprint arXiv:2202.08942}, 2022.

\bibitem{wang2021learning}
Sifan Wang, Hanwen Wang, and Paris Perdikaris.
\newblock Learning the solution operator of parametric partial differential
  equations with physics-informed deeponets.
\newblock {\em Science advances}, 7(40):eabi8605, 2021.

\bibitem{yang2021}
Liu Yang, Xuhui Meng, and George~Em Karniadakis.
\newblock B-pinns: Bayesian physics-informed neural networks for forward and
  inverse pde problems with noisy data.
\newblock {\em Journal of Computational Physics}, 425:109913, 2021.

\bibitem{yang2022scalable}
Yibo Yang, Georgios Kissas, and Paris Perdikaris.
\newblock Scalable uncertainty quantification for deep operator networks using
  randomized priors.
\newblock {\em arXiv preprint arXiv:2203.03048}, 2022.

\bibitem{zhang2019quantifying}
Dongkun Zhang, Lu~Lu, Ling Guo, and George~Em Karniadakis.
\newblock Quantifying total uncertainty in physics-informed neural networks for
  solving forward and inverse stochastic problems.
\newblock {\em Journal of Computational Physics}, 397:108850, 2019.

\end{thebibliography}

\end{document}